\documentclass[12pt,draft,reqno]{amsart}
\usepackage{amsmath,amsfonts,amssymb}
\usepackage[utf8]{inputenc}
\usepackage[russian]{babel}
\usepackage{mathrsfs}
\def\wh{\widehat}
\def\wt{\widetilde} 
\def\R{\mathbb R}

\def\Z{\mathbb Z}

\def\N{\mathbb N}
\def\A{\mathbb A}

\def\x{\mathbf x}
\def\y{\mathbf y}
\def\n{\mathbf n}
\def\m{\mathbf m}

\def\k{\mathbf k}
\def\w{\mathbf w}
\def\z{\mathbf z}

\def\D{\mathbf D}

\def\1{\mathbf 1}

\def\H{\mathfrak H}

\def\bxi{\boldsymbol \xi}

\def\1{\bold 1}

\def\eps{\varepsilon}
\def\Dom{\mathrm{Dom}\,}
\def\Ker{\mathrm{Ker}\,}

\def\le{\leqslant}

\def\ge{\geqslant}

\usepackage[height = 23.5cm, a4paper, hmargin={2.5cm , 2.0cm}]{geometry}

\usepackage{color}

\definecolor{darkred}{rgb}{0.9,0.1,0.1}

\theoremstyle{theorem}
\newtheorem{theorem}{Theorem}[section]
\newtheorem{proposition}[theorem]{Proposition}
\newtheorem{lemma}[theorem]{Lemma}

\newtheorem{remark}[theorem]{Remark}

\numberwithin{equation}{section}

\theoremstyle{plain}
\newtoks\thehProclaim
\newtheorem*{Proclaim}{\the\thehProclaim}

\begin{document}

\bigskip

\centerline{\textbf{Homogenization of L\'evy-type operators:}}
\centerline{\textbf{operator estimates with correctors}}

\bigskip
\def\supind#1{${}^\mathrm{#1}$}
\centerline{ A.~Piatnitski\supind{1,2},
V.~Sloushch\supind{3}, T.~~Suslina\supind{3}, E.~Zhizhina\supind{1,2}}

\medskip
\centerline{\supind{1} The Arctic University of Norway, UiT,  campus Narvik,}
\centerline{Lodve Langes gate 2, Narvik 8517, Norway}

\medskip
\centerline{\supind{2}Higher School of Modern Mathematics MIPT,}
\centerline{Klimentovski per., 1, bld. 1, Moscow, 115184, Russia}

\medskip
\centerline{\supind{3}St. Petersburg State University,}
\centerline{Universitetskaia nab.,  7/9, St. Petersbirg, 199034, Russia}

\medskip
\centerline{e-mail: elena.jijina@gmail.com}
\centerline{e-mail: apiatnitski@gmail.com}
\centerline{e-mail: v.slouzh@spbu.ru}
\centerline{e-mail: t.suslina@spbu.ru}

\medskip
\centerline{Abstract}

The goal of the paper is to study in $L_2(\R^d)$ a self-adjoint operator ${\mathbb A}_\eps$, $\eps >0$, of the form
$$
({\mathbb A}_\eps u) (\x) =  \int_{\R^d} \mu(\x/\eps, \y/\eps) \frac{\left( u(\x) - u(\y) \right)}{|\x - \y|^{d+\alpha}}\,d\y
$$
with $1< \alpha < 2$;
 here the function
 $\mu(\x,\y)$ is $\Z^d$-periodic in the both variables,  satisfies the symmetry relation $\mu(\x,\y) = \mu(\y,\x)$ and
 the estimates $0< \mu_- \leqslant \mu(\x,\y) \leqslant \mu_+< \infty$.
The rigorous definition of the operator  ${\mathbb A}_\eps$ is given in terms of the corresponding quadratic form.
In the previous work of the authors it was shown that the resolvent $({\mathbb A}_\eps + I)^{-1}$ converges, as $\eps\to0$, in the operator norm in $L_2(\mathbb R^d)$ to the resolvent
of the effective operator $A^0$, and the estimate   $\|({\mathbb A}_\eps + I)^{-1} - (\A^0 + I)^{-1} \| = O(\eps^{2-\alpha})$
holds.  In the present work we achieve a more accurate approximation of the resolvent of ${\mathbb A}_\eps$
which takes into account the correctors. Namely,  for $N\in\mathbb N$ such that  $2-1/N < \alpha \le 2-1/(N+1)$,
we obtain
$$
\bigl\|({\mathbb A}_\eps + I)^{-1} - (\A^0 + I)^{-1} - \sum_{m=1}^N \eps^{m(2-\alpha)} \mathbb{K}_m \bigr\| = O(\eps).
$$


\medskip

\noindent
The work of  A. Piatnitski and E. Zhizhina was partially supported by the UiT Aurora project MASCOT.
The work of V. Sloushch  and T. Suslina was supported by RSF, project \hbox{22-11-00092-P}.

\section*{Introduction}

The work is devoted to obtaining operator estimates in homogenization problem for a nonlocal
L\'evy-type operator with a periodic coefficient.  Our goal is to achieve improved operator estimates
for the rate of convergence using the corrector technique.


\subsection{Problem setup. Main results}\label{Sec0.1}

We study an unbounded L\'evy-type operator $\A_\eps = \A_\eps(\alpha,\mu)$ in $L_2(\R^d)$ which is formally defined by
 \begin{equation}
 \label{Aeps_Intro}
 (\mathbb{A}_\eps u)( \x)= \int_{\mathbb R^d}  \mu ( \x/\eps,
  \y/\eps)\frac{( u(\x)-u(\y))}{|\x - \y|^{d+\alpha}}\,d\y,\ \ \x\in\mathbb{R}^{d};
 \end{equation}
here $0< \alpha <2$,  $\mu(\x,\y)$ is
bounded  positive definite function which is
$\mathbb{Z}^{d}$-periodic in each variable and satisfies the relation  $\mu(\x,\y)=\mu(\y,\x)$.
In the rigorous way the operator $\A_\eps$ is defined as a self-adjoint operator generated by the closed quadratic form
\begin{equation}
\label{q_form_Intro}
a_\eps [u,u] := \frac{1}{2} \int_{\R^d} \int_{\R^d} d\x\,d\y\, \mu(\x/\eps,\y/\eps) \frac{|u(\x)-u(\y)|^2}{|\x - \y|^{d+\alpha}},\ \ u\in H^{\alpha/2}(\R^d).
\end{equation}

Operator $-\A_\eps$  is the generator of a jump Markov process in $\mathbb R^d$, a detailed description
of such processes and their properties  can be found in the work  \cite{BBZK}.
The integral kernel of $\A_\eps$ shows a power law decay at infinity, moreover, it has an infinite second moment.
A characteristic property of the corresponding Markov processes is the presence of long range interactions, these processes can make long distance jumps
(L\'evy flights). Therefore, the trajectories of such processes differ significantly from continuous trajectories of diffusion
processes.
 At present Levy-type processes are widely used in modeling the behaviour of complex systems in which long distance interaction plays an essential, sometimes even key, role.
In particular, many models in population biology and ecology, astrophysics, financial mathematics and mechanics of porous media are based on these processes, see, for example, \cite{CT, EP, HS, NS, UZ, W}. When  studying such models in environments with variable characteristics, we come to Markov processes with a generator of the form $-\A_\eps$.

Periodic homogenization problem for the operator  $\A_\eps$ was considered in the paper \cite{KaPiaZhi19} where it
was shown that the resolvent $(\A_\eps + I)^{-1}$ converges, as $\eps\to0$, strongly in $L_2(\mathbb R^d)$ to the
resolvent $(\A^0 + I)^{-1}$ of the effective operator.
The effective operator $\A^0$ has the same structure as $\mathbb{A}_\eps$, but with constant coefficient
$$
\mu^0 = \int_\Omega \int_{\Omega} \mu(\x,\y)\,d\x\,d\y, \quad \Omega := [0,1)^d.
$$
This operator coincides, up to a constant factor, with the fractional power of Laplacian:
$\A^0 = \mu^0 c_0(d,\alpha) (-\Delta)^{\alpha/2}$, $\Dom \A^0 = H^{\alpha}(\R^d)$.

In a different framework homogenization results for non-symmetric
L\'evy-type operators with periodic coefficients that correspond  to stable-like
  Markov processes were obtained in the work  \cite{CCKZ21_2}; for $\alpha>1$ the convergence result holds in moving
  coordinates.
  The paper  \cite{CCKZ21}  focuses on estimating the convergence rate for symmetric L\'evy-type operators in the strong topology.
  The approach used in  \cite{CCKZ21} relies on probabilistic arguments.

In our previous work \cite{JPSS24} it was shown that the resolvent  $(\A_\eps + I)^{-1}$
converges to the resolvent of the effective operator in the operator norm in $L_2(\mathbb R^d)$. Moreover,
the rate of convergence can be estimates as follows:
\begin{equation}\label{e3.1_Intro}
\|(\A_{\varepsilon}+I)^{-1}-(\A^{0}+I)^{-1}\|_{L_2(\R^d) \to L_2(\R^d)}\le
{\mathrm C}_1(\alpha,\mu) \begin{cases} \eps^\alpha, & 0 < \alpha <1,\\
\eps (1 + | \operatorname{ln} \eps|)^2, & \alpha =1,
\\ \eps^{2 - \alpha}, & 1< \alpha < 2.
\end{cases}
\end{equation}

As noted in \cite{JPSS24}, the estimate $O(\eps^\alpha)$ is optimal, at least in the framework of the approach used
in this work.
Thus, for $0< \alpha <1$ just
the leading term of the expansion already provides an optimal approximation to the resolvent
$(\A_{\varepsilon}+I)^{-1}$. For $1 \le \alpha <2$ it is not the case. Moreover, the estimate $O(\eps^{2-\alpha})$
is getting worse, as $\alpha$ is approaching 2.

In the present work we consider the case $1< \alpha < 2$ and show  that the precision of the approximation
can be improved by taking into account appropriate correctors.
The main result of the work, Theorem \ref{teor3.2}, states that for any $N \in \N$ and any
$\alpha\in\big(2 - \frac{1}{N}, 2\big)$ the following estimate holds:
\begin{equation}\label{e0.1}
\begin{aligned}
\bigl\| (\A_{\varepsilon}+I)^{-1}-(\A^{0}+I)^{-1} - \sum_{m=1}^N  \eps^{m(2-\alpha)} \mathbb{K}_m \bigr\|_{L_2(\R^d) \to L_2(\R^d)}
\\
\le
{\mathrm C}_2(\alpha,\mu) \begin{cases}
  \eps, & 2 - \frac{1}{N} < \alpha \le 2 - \frac{1}{N+1},
  \\
 \eps^{(N+1)(2-\alpha)}, & 2 - \frac{1}{N+1} < \alpha < 2.
  \end{cases}
\end{aligned}
 \end{equation}
 The correctors ${\mathbb K}_m,$ $m=1,\dots,N,$ are given by the relations
 $$
 {\mathbb K}_m :=  ( \operatorname{div} g^0 \nabla)^m (\A^{0}+I)^{-m-1},\quad m=1,\dots,N,
 $$
where (not necessary sign-definite) matrix $g^0$ is defined in terms of solutions to auxiliary problems.
For a given $\alpha$, $1 < \alpha <2$, one can choose  $N \in \N$ in such a way that
$2 - \frac{1}{N} < \alpha \le 2 - \frac{1}{N+1}$.  Then the approximation of the resolvent
$(\A_{\varepsilon}+I)^{-1}$ that takes  into account the first $N$ correctors yields a precision of order $O(\eps)$.

\subsection{Spectral method. Operator estimates }\label{Sec0.1}

Currently, the homogenization theory of periodic operators is a well developed field of mathematics
which comprises a number of different approaches and techniques, see for instance  \cite{BaPa}, \cite{BeLP}, \cite{JKO}
for further details.
One of the important methods in this field, the so-called spectral method, is based on the scaling transformation
and Floquet-Bloch theory. The first rigorous homogenization result obtained by this method was published in   \cite{Sev},
where it was shown that the resolvent of a uniformly elliptic operator
${\mathcal A}_\eps = - \operatorname{div} g(\x/\eps) \nabla$ with periodic coefficients converges strongly in
$L_2 (\mathbb{R}^d)$  to the resolvent of the effective operator. The latter has the form
${\mathcal A}^0 = - \operatorname{div} g^0 \nabla$ with a constant positive definite matrix $g^0$ which is called
the effective matrix.
Later on this approach was further developed in \cite{Zh1}, \cite{AlCo98}, \cite{AlCoVa98}, \cite{COrVa} and other papers. These papers dealt with various homogenization problems for differential operators with a periodic microstructure,
among them are boundary value problems for elliptic operators and related spectral problems, operators in perforated domains and fluid mechanics problems.
However, it should be noted  that all the mentioned works focused on proving the strong resolvent convergence.

In the works \cite{BSu1, BSu3, BSu4} M. Birman and T. Suslina introduced and developed a new approach to problems of
homogenization of periodic differential operators in $\mathbb R^d$, the so-called operator-theoretic approach, which
is a version of the spectral method.
This approach allows one to obtain order-sharp estimates for the rate of resolvent convergence in operator norms
for a wide class of homogenization problems in periodic media.
We illustrate this approach by considering in $L_2 ({\mathbb R}^d)$ a scalar divergence form elliptic operator
${\mathcal A}_\eps = - \operatorname{div} g(\x/\eps) \nabla$ with $\mathbb Z^ d$-periodic coefficients.
According to the classical homogenization theory, for such an operator the strong resolvent convergence takes place,
as $\eps\to0$.
In \cite{BSu1} it was shown that a more advanced convergence result holds. Namely, the resolvent $({\mathcal A}_\eps +I) ^{-1}$ converges to the resolvent of the effective operator
${\mathcal A}^0$
in the operator norm in  $L_2(\R^ d)$.  Furthermore, the following estimate holds
\begin{equation}
\label{BSu1}
\| ( {\mathcal A}_\eps +I)^{-1} - ( {\mathcal A}^0 +I)^{-1}\|_{L_2(\R^d) \to L_2(\R^d)} \leqslant C \eps.
\end{equation}
In the homogenization theory estimates of this type are called {\sl operator estimates for the rate of convergence}.
In  \cite{BSu3} a more precise approximation of the resolvent $( {\mathcal A}_\eps + I)^{-1}$, including additional terms with correctors, was obtained. This approximation provides the precision of order $O (\eps^2)$ in the operator norm in
$L_2(\mathbb R ^ d)$.
The rate of convergence of the resolvent $( {\mathcal A}_\eps + I)^{-1}$ in the norm of operators acting from
$L_2(\R^d)$ to the Sobolev space  $H^1 (\R^d)$ was investigated in  \cite{BSu4}.

The operator-theoretic approach is based on scaling transformation, Floquet-Bloch theory and analytic perturbation theory.
Let us clarify this method using the derivation of estimate \eqref{BSu1} as an example.
Making scaling transformation we reduce estimate \eqref{BSu1} to the inequality
\begin{equation}
\label{BSu2}
\| ( {\mathcal A} + \eps^2 I)^{-1} - ( {\mathcal A}^0 + \eps^2 I)^{-1}\|_{L_2(\R^d) \to L_2(\R^d)} \leqslant C \eps^{-1},
\end{equation}
where ${\mathcal A} = - \operatorname{div} g(\x) \nabla = \D^* g(\x) \D$, $\D = -i \nabla$.
With the help of the unitary Gelfand transform the operator ${\mathcal A}$ is decomposed into a direct integral over
the operators ${\mathcal A}(\bxi)$, acting in $L_2(\Omega)$ and depending on a parameter, the so-called quasi-momentum, $\bxi \in \widetilde{\Omega}$.
Here  $\Omega = [0,1)^d$ is a cell of the lattice
$\Z^d$, and $\widetilde{\Omega} = [-\pi,\pi)^d$ is a cell of the dual lattice. Operator
${\mathcal A}(\bxi)$ is given by the formula ${\mathcal A}(\bxi) = (\D + \bxi)^* g(\x) (\D+ \bxi)$, it acts
in the space of periodic functions.
Then estimate \eqref{BSu2} is equivalent to the estimates
\begin{equation*}
\| ( {\mathcal A}(\bxi) + \eps^2 I)^{-1} - ( {\mathcal A}^0(\bxi) + \eps^2 I)^{-1}\|_{L_2(\Omega) \to L_2(\Omega)} \leqslant C \eps^{-1},
\quad \bxi \in \widetilde{\Omega},
\end{equation*}
for the operators ${\mathcal A}(\bxi)$ depending on a quasi-momentum.
The main part of the study is  investigating the operator family  ${\mathcal A}(\bxi)$ which is analytic and consists
of operators with compact resolvent.
Thus the methods of the analytic perturbation theory can be applied. It turns out that  the resolvent
\hbox{$({\mathcal A}(\bxi) + \eps^2 I )^{-1}$} can be approximated in terms of the spectral characteristics
of the operator at the spectral edge.
Thus, the effect of homogenization is a {\sl spectral threshold effect} at the spectral edge of elliptic operator.

A different approach to obtaining operator estimates
 of the approximation discrepancy in homogenization problems, the so-called ``shift method'', was proposed in the works of Zhikov and Pastukhova, see \cite{Zh, ZhPas1}, as well as the review \cite{ZhPas3} and the literature cited there.  In the recent years operator estimates for the rate of convergence in  homogenization problems for various differential operators attract the attention of a growing number of researches.
A number of significant results has been obtained in this topic. A detailed survey of the state of art in this field can be found in \cite[Introduction]{Su_UMN2023}.

\subsection
{Operator estimates in homogenization problems for nonlocal convolution-type operators}

The study of operator estimates in  homogenization problems for periodic convolution-type operators
was initiated  in the recent works by the authors  \cite{PSlSuZh, PSlSuZh2},
 where an operator $A_\eps$ of the form
 \begin{equation}
 \label{Sus1}
 ({A}_\eps u)( \x)=\frac1{\eps^{d+2}}\int_{\mathbb R^d} a((\x-\y)/\eps) \mu ( \x/\eps,
  \y/\eps)\big( u(\x)-u(\y)\big)\,d\y,\ \ \x\in\mathbb{R}^{d},\ \ u\in L_{2}(\mathbb{R}^{d}),
 \end{equation}
 was considered in $L_2(\R^d)$.
 It was assumed that $a(\x)$ is an even non-negative function from  $L_{1}(\mathbb{R}^{d})$,
 $\mu(\x,\y)$ is a bounded positive definite $\mathbb{Z}^{d}$-periodic in the variables $\x$ and $\y$ function
 such that $\mu(\x,\y)=\mu(\y,\x)$. Under these conditions the operator  ${A}_{\eps}$ is bounded, self-adjoint
 and non-negative.
It was also assumed that $a(\cdot)$ has finite moments
$M_{k}(a) =\int_{\mathbb{R}^{d}}|\x|^{k}a(\x)\,d\x$ up to order 3 or 4.

Convolution-type operators with integrable kernels appear in various models of mathematical biology and population dynamics, in the recent years these models were intensively studied  in the mathematical literature, see
\cite{KPMZh, PZh, PiaZhi19}.
The work \cite{PZh} focused on periodic homogenization of such operators, it was proved in this work that in the case
$M_2(a) < \infty$ the resolvent $({A}_{\varepsilon}+I)^{-1}$ converges strongly in $L_2(\mathbb R^d)$ to the resolvent
\hbox{$({A}^{0}+I)^{-1}$} of the effective operator. The effective operator takes the form
${A}^{0}=-\operatorname{div}g^{0}\nabla$ with a positive definite constant matrix   $g^0$.
 It is interesting to observe that the effective operator is local and unbounded while the original operators  ${A}_{\eps}$
 are nonlocal and bounded.

Similar problems in perforated domains were investigated by variational methods in  \cite{BraPia21}.
The case of  non-symmetric convolution type kernels was addressed in  \cite{PiaZhi19}, where it was proved that
for the corresponding parabolic semigroups the homogenization result holds in moving coordinates.

In the papers \cite{PSlSuZh, PSlSuZh2} the operator-theoretic approach originally developed for differential operators
was modified and successfully adapted to the case of convolution-type operators with integrable kernels.
As in the case of differential equations, homogenization problem for convolution-type operators is reduced to studying
the family of operators $A(\bxi)$, $\bxi\in\widetilde\Omega$, obtained by applying the scaling transformation and the Gelfand transform to the original
operator. However, in contrast with differential operators,  this family is not analytic in $\bxi$ and thus the analytic
perturbation theory does not apply.  Instead, the authors used a different approach  that relies on finite regularity of $A(\bxi)$.  This regularity is ensured by
the condition of finiteness of several moments of  $a(\x)$.
Under the assumption $M_3(a) < \infty$ the order sharp estimate for the rate of convergence in the operator norm was deduced in \cite{PSlSuZh}. This estimate reads
\begin{equation*}\label{Sloushch2}
\|({A}_{\varepsilon}+I)^{-1}-({A}^{0}+I)^{-1}\|_{L_{2}(\mathbb{R}^{d})\to
L_{2}(\mathbb{R}^{d})}\leqslant C(a,\mu)\varepsilon,\ \ \varepsilon>0.
\end{equation*}
In the case \hbox{$M_4(a) < \infty$} a more accurate approximation of the resolvent $({A}_{\varepsilon}+I)^{-1}$
was constructed in \cite{PSlSuZh2} and \cite{PSSZ_Ur}. By taking into account the correctors, this approximation provides a precision of order
$O(\eps^2)$.

\subsection{Method}
In order to obtain quantitative homogenization results for the operator $\A_\eps$
we modify the operator-theoretic approach and adapt it to the setting of L\'evy-type operators.

At the first step, making the scaling transformation, we derive the relation
\begin{equation}\label{e3.2_Intro}
\|(\A_{\varepsilon}+I)^{-1}-(\A^{0}+I)^{-1}\|_{L_2(\R^d) \to L_2(\R^d)} = \eps^\alpha
\|(\A + \eps^\alpha I)^{-1}-(\A^{0}+ \eps^\alpha I)^{-1}\|_{L_2(\R^d) \to L_2(\R^d)}.
\end{equation}
Here $\A = \A_{\eps_0},\ \eps_0 =1$.
Then, by means of the Gelfand transform, the operator  $\A$ is decomposed into a direct integral over the operators
 $\A(\bxi)$,   $\bxi \in \wt{\Omega}$, that act in the space $L_2(\Omega)$. For each  $\bxi\in \wt{\Omega}$  the spectrum
 of  $\A(\bxi)$ is discrete and belongs to $\mathbb R_ +$. Moreover, the first eigenvalue is of order $O(|\bxi|^\alpha)$,
 while the remaining eigenvalues are separated from zero.

Thus studying the limit behaviour of the resolvent $(\mathbb{A}_\eps+ I) ^{-1}$ as $\eps\to0$ is reduced to
obtaining  the asymptotics of the resolvent \hbox{$(\mathbb{A}(\bxi)+\eps^{\alpha}I)^{-1}$} for small $\eps$.
It is clear that the main contribution to the asymptotics of interest comes from the bottom of the spectrum of   $\A(\bxi)$.
It should be emphasized that, in contrast with the case of elliptic differential operators, the family  $\A(\bxi)$ is not
analytic and has low regularity.

It should  also be noted that the studied L\'evy-type operators differ significantly from the convolution-type operators of the form \eqref{Sus1}, for which the finite differentiability of the family  $\A(\bxi)$ is ensured by the finiteness of the corresponding number of  moments of $a(\x)$.
Nevertheless, we succeeded to obtain ''threshold approximations'' required for constructing an approximation
of the resolvent    $(\mathbb{A}(\bxi)+\eps^{\alpha}I)^{-1}$ for small $\eps$.
To this end we characterized  the behaviour of the operators  $F(\bxi)$ and $\mathbb{A}(\bxi)F(\bxi)$ as $\bxi\to0$;
here $F(\bxi)$ is the spectral projection of the operator $\mathbb{A}(\bxi)$ that corresponds to some neighbourhood  of zero.
In the existing literature the asymptotics of the operator $\mathbb{A}(\bxi)F(\bxi)$ for small  $\bxi$ is usually determined
in terms of the  behaviour of the principal eigenvalue of the operator   $\mathbb{A}(\bxi)$ in the vicinity of zero.
In the present work we use an alternative approach that relies on integrating  the resolvent   $(\A(\bxi) - \zeta I)^{-1}$
over a proper contour on the complex plane.

Since $\| (\A(\bxi) + \eps^\alpha I)^{-1} F(\bxi)^\perp\| \le C$, the best accuracy that we can obtain when approximating the resolvent $(\A_{\varepsilon}+I)^{-1}$ by the described above method  is of order $O(\eps^\alpha)$;  see \eqref{e3.2_Intro}.
Therefore,  in the case $0< \alpha < 1$,  the leading term of the approximation already provides the best precision,
see \eqref{e3.1_Intro}. For \hbox{$1 \le \alpha <2$} this is not the case. In the present work we consider the case
 $1< \alpha < 2$, our goal is to improve the precision of the approximation by taking into account the correctors,
 see \eqref{e0.1}.

\subsection{Plan of the paper}
The paper consists of Introduction and five sections.
In Section 1 we introduce operator ${\mathbb A}$, represent it as a direct integral over the family of operators  ${\mathbb A}(\bxi)$ and obtain a lower bound for the quadratic form of ${\mathbb A}(\bxi)$.
In Section 2 we derive a representation for the difference of the quadratic forms $a(\bxi)$ and $a(\mathbf{0})$
which is used in the further analysis.
In Section 3 the threshold characteristics of the operator family  ${\mathbb A}(\bxi)$ are studied in the vicinity of the lower edge
of the spectrum, here approximations for the spectral projector $F(\bxi)$ and for the operator $\A(\bxi) F(\bxi)$ are constructed for small  $|\bxi|$.
In  Section 4 we first construct an approximation of the resolvent $( {\mathbb A}(\bxi) + \eps^\alpha I)^{-1}$
for small $\eps$, and then, using the decomposition of  $\A$ into a direct integral, obtain an approximation of the resolvent
$( {\mathbb A} + \eps^\alpha I)^{-1}$.
Finally, in Section 5, combining the results of Section 4 and the scaling transformation, we
approximate the resolvent $({\mathbb A}_\eps + I)^{-1}$ in the operator norm in  $L_2(\R^d)$, this is the main result of the work.

\subsection{Notation}
A norm in a linear normed space $X$ is denoted by $\|\cdot\|_{X}$, or without lower index if it does nor lead to ambiguity.
If $X$ and $Y$ are normed spaces, the standard norm of a linear operator  $T:X\to Y$
is denoted by $\|T\|_{X\to
Y}$, or just  $\|T\|$.  The notation $\mathcal{L}\{F\}$ stands for the linear span of a collection of vectors  $F\subset X$.

Let $\H$, $\H_*$ be  separable complex Hilbert spaces. For a linear operator $A: \H \to \H_*$ we denote by $\operatorname{Dom} A$ and $\operatorname{Ker} A$ the domain and the kernel of $A$, respectively.
For a domain $\mathcal O\subset\R^d$ the notation $L_{p}({\mathcal O})$,
\hbox{$1 \le p \le \infty$}, is used for the standard  $L_p$ spaces.
If $p=2$, the corresponding inner product in $L_{2}({\mathcal O})$ is denoted by
 $(\cdot,\cdot)_{L_{2}({\mathcal O})}$ or just  $(\cdot,\cdot)$.
The notation $H^s({\mathcal O})$ stands for the standard Sobolev class of order $s>0$
in a domain $\mathcal O$.

We also use the following notation: $\mathbf{x} = (x_1,\dots, x_d)^{t} \in \R^d$, $i D_j = \partial_j = \partial / \partial x_j$, $j=1,\dots,d$; $\mathbf{D} = - i \nabla = (D_1,\dots,D_d)^t$.
For the Schwartz class in  $\R^{d}$ the standard notation $\mathcal{S}(\R^{d})$ is used,  the characteristic function
of a set  ${\mathcal O}\subset\R^d$ is denoted  $\mathbf{1}_{\mathcal O}$.

Finally, $B_r(\x_0)$ denotes an open ball in $\R^d$ of radius $r$ centered at  $\x_0$, and
 $\omega_{d}$ is the surface area of the unit sphere ${\mathbb S}^{d-1}$ in $\R^d$.

\section{L\'evy-type operators with periodic coefficients: \\ decomposition in direct integral and estimates}

\subsection{Operator $\A(\alpha,\mu)$}
Let   $\mu\in L_{\infty}(\R^d\times \R^d)$ be such that
\begin{gather}
\label{e1.1}
0<\mu_{-}\le\mu(\x,\y)\le\mu_{+}< \infty,\ \ \mu(\x,\y)=\mu(\y,\x),\ \ \x, \y\in\R^d;
\\
\label{e1.2}
\mu(\x+\m,\y+\n)=\mu(\x,\y),\ \ \x, \y\in\R^d,\ \ \m,\n\in\Z^d.
\end{gather}
We assume that $1 < \alpha < 2$,  set $\gamma := \frac{\alpha}{2}$, and
consider in the space $L_{2}(\R^d)$ the quadratic form
\begin{equation}
\label{e1.3}
a(\alpha,\mu) [u,u] := \frac{1}{2} \int_{\R^d} \int_{\R^d} d\x\,d\y\, \mu(\x,\y) \frac{|u(\x)-u(\y)|^2}{|\x - \y|^{d+\alpha}},\ \ u\in H^\gamma(\R^d).
\end{equation}
In view of  \eqref{e1.1} and \eqref{e1.3} the form $a(\alpha,\mu)$ is densely defined in $L_2(\R^d)$, non-negative
and satisfies the estimates
 \begin{equation}
\label{e1.4}
\mu_- a_0(\alpha) [u,u] \le a(\alpha,\mu) [u,u] \le \mu_+ a_0(\alpha) [u,u],\ \ u\in H^\gamma(\R^d),
\end{equation}
with
\begin{equation}
\label{e1.5}
a_0(\alpha) [u,u] := \frac{1}{2} \int_{\R^d} \int_{\R^d} d\x\,d\y\, \frac{|u(\x)-u(\y)|^2}{|\x - \y|^{d+\alpha}},\ \ u\in H^\gamma(\R^d).
\end{equation}
The following statement is well-known, see, for instance, \cite[\S~6.31]{Sa}.

\begin{lemma}
\label{lem1.1}
The form in  \eqref{e1.5} admits the representation
\begin{equation}
\label{e1.6}
a_0(\alpha) [u,u] = c_0(d,\alpha)  \int_{\R^d} d\k\, |\k|^\alpha | \wh{u}(\k)|^2,\ \ u\in H^\gamma(\R^d),
\end{equation}
where $\wh{u}(\k)$ is the Fourier image of a function $u(\x)$, and the constant $c_0 = c_0(d,\alpha)$
is defined by
\begin{equation}
\label{e1.6a}
 c_0 = c_0(d,\alpha) =  \int_{\R^d} \frac{1 - \cos z_1}{ | \z|^{d+\alpha}} \,d\z = \frac{ \pi^{d/2} |\Gamma(-\alpha/2)|}{2^\alpha \Gamma((d+\alpha)/2)}.
\end{equation}
\end{lemma}

It should be noted that $c_0(d,\alpha) = O((2-\alpha)^{-1})$, as $\alpha \to 2$.

By Lemma  \ref{lem1.1} the form $a_0(\alpha)$  is closed, and, due to the estimates in \eqref{e1.4}, the form  $a(\alpha,\mu)$ is also closed.

By definition, ${\mathbb A} = {\mathbb A}(\alpha,\mu)$ \emph{is the self-adjoint operator in $L_2(\R^d)$ generated by the closed form} in  \eqref{e1.3}.  Formally, one can write, see \cite{KaPiaZhi19},
\begin{equation*}
({\mathbb A} u) (\x) =  \int_{\R^d} \mu(\x, \y) \frac{\left( u(\x) - u(\y) \right)}{|\x - \y|^{d+\alpha}}\,d\y.
\end{equation*}
Let ${\mathbb A}_0 = {\mathbb A}_0(\alpha)$ be the self-adjoint operator in $L_2(\R^d)$  generated by the closed form
in  \eqref{e1.5}. Due to representation \eqref{e1.6}, operator  ${\mathbb A}_0(\alpha)$ coincides, up to a multiplicative constant,  with the fractional power of the Laplacian:
\begin{equation*}
{\mathbb A}_0(\alpha)  = c_0(d,\alpha) ( - \Delta)^{\gamma},  \quad \Dom{\mathbb A}_0(\alpha) = H^{\alpha}(\R^d).
\end{equation*}

From representation  \eqref{e1.6} it follows that the point $\lambda_{0}=0$ is the lower edge of the spectrum of the operator ${\mathbb A}_0(\alpha)$, and, due to estimates  \eqref{e1.4}, this point is also the lower edge of the spectrum
of $\A(\alpha,\mu)$.

\subsection{The family of operators $\A(\bxi;\alpha,\mu)$}\label{sec1.2}
Denote by $\Omega:=[0,1)^{d}$ the periodicity cell of the lattice $\Z^d$, and let  $\widetilde\Omega:=[-\pi,\pi)^{d}$ be
the cell of the dual lattice $(2\pi\mathbb{Z})^{d}$. For $s>0$ we denote by  $\wt{H}^s(\Omega)$ the subspace of $H^s(\Omega)$ that consists of the functions whose $\Z^d$-periodic extension belongs to $H^s_{\operatorname{loc}}(\R^d)$.

In the space of  $\Z^d$-periodic functions the standard discrete Fourier transform ${\mathcal F}: L_{2}(\Omega)\to \ell_{2}(\Z^d)$
is defined by the formula
\begin{gather*}
{\mathcal F} u(\n)= \wh{u}_\n = \int_{\Omega}u(\x)e^{-2\pi i \langle \n, \x \rangle}d\x,\ \ \n\in\Z^d,\ \ u\in L_{2}(\Omega);
\\
u(\x) = \sum_{\n\in\Z^d} \wh{u}_{\n}e^{2\pi i \langle \n, \x\rangle},\ \ \x\in\Omega.
\end{gather*}

Then the relation   $u \in \wt{H}^s(\Omega)$ is equivalent to the convergence of the series
$$
\sum_{\n \in \Z^d} (1 + |2\pi \n|^2)^s | \wh{u}_\n|^2.
$$
Moreover, this sum admits two-sided estimates by $\| {u} \|^2_{{H}^s(\Omega)}$.

In the space $L_2(\Omega)$ we consider a family of  quadratic forms $a(\bxi)= a(\bxi;\alpha,\mu)$ depending on a parameter $\bxi \in \widetilde{\Omega}$ and defined by
 \begin{equation}
\label{e1.9}
a(\bxi;\alpha,\mu) [u,u] := \frac{1}{2} \int_{\R^d} d\y \int_{\Omega} d\x\, \mu(\x,\y)
\frac{| e^{i \langle \bxi,\x \rangle}u(\x)- e^{i \langle \bxi,\y \rangle} u(\y)|^2}{|\x - \y|^{d+\alpha}},\ \
u\in \wt{H}^\gamma(\Omega);
\end{equation}
it is assumed here that the function $u\in \wt{H}^\gamma(\Omega)$ is extended to $\R^d$ as  $\Z^d$-periodic function.
According to \eqref{e1.1}, the form $a(\bxi;\alpha,\mu)$ is densely defined in  $L_2(\Omega)$, non-negative, and satisfies
the estimates
 \begin{equation}
\label{e1.10}
\mu_- a_0(\bxi;\alpha) [u,u] \le a(\bxi;\alpha,\mu) [u,u] \le \mu_+ a_0(\bxi;\alpha) [u,u],\ \ u\in \wt{H}^\gamma(\Omega),
\end{equation}
with
\begin{equation}
\label{e1.11}
a_0(\bxi;\alpha) [u,u] := \frac{1}{2} \int_{\R^d} d\y \int_{\Omega} d\x\,
\frac{| e^{i \langle \bxi,\x \rangle}u(\x)- e^{i \langle \bxi,\y \rangle} u(\y)|^2}{|\x - \y|^{d+\alpha}},\ \
u\in \wt{H}^\gamma(\Omega).
\end{equation}

A proof of the following statement can be found in  \cite[Lemma 1.2]{JPSS24}.
\begin{lemma}[\cite{JPSS24}]
\label{lem1.2}
The form in  \eqref{e1.11} admits the representation
\begin{equation}
\label{e1.12}
a_0(\bxi;\alpha) [u,u] =  c_0(d,\alpha) \sum_{\n \in \Z^d} | 2\pi \n + \bxi|^{\alpha} | \wh{u}_\n|^2,\ \
u\in \wt{H}^\gamma(\Omega).
\end{equation}
Here $\wh{u}_\n$, $\n \in \Z^d$, are the Fourier coefficients of the function $u$, and $c_0(d,\alpha)$
is a constant given by \eqref{e1.6a}.
\end{lemma}

As a consequence of this Lemma, the form $a_0(\bxi;\alpha)$ and, in view of estimates  \eqref{e1.10}, the form
$a(\bxi;\alpha,\mu)$ are closed.

We then define ${\mathbb A}(\bxi) = {\mathbb A}(\bxi;\alpha,\mu)$ \emph{as a self-adjoint operator in $L_2(\Omega)$
that corresponds to the closed form in} \eqref{e1.9}, and write formally
\begin{equation*}
({\mathbb A}(\bxi) u) (\x) =  \int_{\R^d} \mu(\x, \y) \frac{\left( u(\x) - e^{- i \langle \bxi, \x-\y\rangle}u(\y) \right)}{|\x - \y|^{d+\alpha}}\,d\y.
\end{equation*}
Let ${\mathbb A}_0(\bxi) = {\mathbb A}_0(\bxi;\alpha)$ be a self-adjoint operator in
$L_2(\R^d)$ generated by the closed form in   \eqref{e1.11}.
Then, thanks to representation \eqref{e1.12}, the operator ${\mathbb A}_0(\bxi;\alpha)$  coincides with the fractional
power of the operator $|\D +\bxi|$ up to a multiplicative constant:
\begin{equation}
\label{e1.14}
{\mathbb A}_0(\bxi;\alpha)  = c_0(d,\alpha) | \D + \bxi|^{\alpha},  \quad \Dom{\mathbb A}_0(\bxi;\alpha) =
\wt{H}^{\alpha}(\Omega).
\end{equation}
Due to the compactness of embedding of the space $\wt{H}^{\gamma}(\Omega)$ (being the domain of the form
$a(\bxi;\alpha,\mu)$) into $L_2(\Omega)$,  the spectrum of both operators  $\A(\bxi)$ and  $\A_0(\bxi)$
is discrete for each $\bxi \in \wt{\Omega}$.

\subsection{Decomposition of the operator $\A(\alpha,\mu)$ into a direct integral}
For $\n\in\Z^d$ consider the unitary shift operators  $S_{\n}$ in $L_2(\R^d)$ defined by the relation
\begin{equation*}
S_{\n}u(\x) =u(\x+\n),\ \ \x\in\R^d,\ \ u \in L_2(\R^d).
\end{equation*}
It is straightforward to check that under conditions  (\ref{e1.1}), (\ref{e1.2})  we have
$$
a(\alpha,\mu) [S_{\n} u, S_{\n}u] = a(\alpha,\mu) [ u, u],\quad u \in H^\gamma(\R^d), \quad \n \in \Z^d.
$$
This implies that the operator  $\A(\alpha,\mu)$ commutes with the operators $S_{\n}$ for all $\n \in \Z^d$,
that is  $\A(\alpha,\mu)$ \emph{is a   $\Z^d$-periodic operator}.

Next we recall the definition of the Gelfand transform $\mathcal{G}$, see, for instance,  \cite{Sk} or \cite[Chapter~2]{BSu1}).  For functions from the Schwartz class  ${\mathcal S}(\R^d)$ it is defined by
\begin{equation*}
\mathcal{G}u(\bxi,\x) = \wt{u}(\bxi,\x):=(2\pi)^{-d/2}\sum_{\n\in\Z^d} u(\x+\n) e^{-i \langle \bxi, \x+\n\rangle},\
\ \bxi\in\widetilde\Omega,\ \ \x\in\Omega,\ \ u\in {\mathcal S}(\R^d).
\end{equation*}
Then $\mathcal{G}$  is extended by continuity to the unitary mapping
 $$
 \mathcal{G}: L_{2}(\R^d) \to \int_{\widetilde\Omega}\oplus L_{2}(\Omega)\, d\bxi=L_{2}(\widetilde\Omega\times\Omega).
 $$
Under the Gelfand transform the Sobolev space $H^s(\R^d)$,  $s>0$, turns into a direct integral
of the spaces $\wt{H}^s(\Omega)$:
 $$
 \mathcal{G}: H^s(\R^d) \to \int_{\widetilde\Omega}\oplus \wt{H}^s(\Omega)\, d\bxi
 =L_{2}(\widetilde\Omega; \wt{H}^s(\Omega)).
 $$

Like any periodic operator, $\A(\alpha,\mu)$  is decomposed into a direct integral by means of the Gelfand transform.
This is the subject of the following Lemma proved in   \cite[Lemma 1.3]{JPSS24}:

\begin{lemma}
[\cite{JPSS24}]
\label{lem1.3}
Let conditions  \eqref{e1.1}and  \eqref{e1.2}  be fulfilled, and assume that $1< \alpha< 2$.
Assume, moreover that the form $a = a(\alpha,\mu)$ in $L_2(\R^d)$ is given by \eqref{e1.3},
and the family of forms $a(\bxi) = a(\bxi;\alpha,\mu)$ in $L_2(\Omega)$ is defined in \eqref{e1.9}, here $\bxi \in \wt{\Omega}$.
Then the relation   $u \in H^\gamma(\R^d)$ is equivalent to the relation
$\mathcal{G} u = \wt{u} \in L_2(\wt{\Omega}; \wt{H}^\gamma(\Omega))$, and for almost all
 $\bxi\in \wt{\Omega}$ we have
$\wt{u}(\bxi,\cdot ) \in \wt{H}^\gamma(\Omega)$ and
\begin{equation*}
 a[u,u] = \int_{\wt{\Omega}} a(\bxi) [ \wt{u}(\bxi,\cdot), \wt{u}(\bxi,\cdot)] \, d\bxi,
 \quad u \in H^\gamma(\R^d).
 \end{equation*}
\end{lemma}
 Since both the operator $\A$ and operators  $\A(\bxi)$ are generated by the corresponding quadratic forms, from Lemma
  \ref{lem1.3} we deduce that
\begin{equation}
\label{e1.15}
\A(\alpha,\mu) = {\mathcal G}^* \Bigl( \int_{\widetilde\Omega} \oplus  \A(\bxi;\alpha,\mu) \,d\bxi\Bigr) {\mathcal G}. \end{equation}

\subsection{Estimates of the quadratic form of operator  $\A(\bxi;\alpha,\mu)$}\label{sec1.4}

Due to Lemma  \ref{lem1.2} the operators $\A_{0}(\bxi; \alpha)$, $\bxi\in\widetilde\Omega$
can be diagonalized by the discrete Fourier transform ${\mathcal F}$ as follows:
\begin{equation}
\label{e1.16}
\A_{0}(\bxi; \alpha)= c_0(d,\alpha) {\mathcal F}^{*} \bigl[   |2\pi \n+\bxi|^\alpha \bigr] {\mathcal F},\ \ \bxi\in\wt{\Omega}.
\end{equation}
Here $[  |2\pi \n+\bxi|^\alpha ]$ denotes the operator of multiplication by the function
$ |2\pi \n+\bxi|^\alpha$, $\n \in \Z^d$, in the space $\ell_2(\Z^d)$.

As an immediate consequence of this diagonalization we have
$\operatorname{Ker}\A_{0}(\mathbf{0};\alpha)=\mathcal{L}\{\mathbf{1}_{\Omega}\}$.
Combining this relation with \eqref{e1.10} yields $\operatorname{Ker}\A(\mathbf{0}; \alpha,\mu)=\mathcal{L}\{\mathbf{1}_{\Omega}\}$, and we arrive at the following statement,
see also \cite[Lemma 1.4]{JPSS24}:
\begin{lemma}[\cite{JPSS24}]
\label{lem1.4}
Let conditions \eqref{e1.1} and \eqref{e1.2} be fulfilled,  and assume that $1< \alpha < 2$.
Then  $\lambda_{0}=0$ is a simple eigenvalue of the operator
$\A(\mathbf{0};\alpha,\mu)$, and
$\operatorname{Ker}\A(\mathbf{0};\alpha,\mu)=\mathcal{L}\{\mathbf{1}_{\Omega}\}$.
\end{lemma}

Due to elementary estimates
\begin{align}
\label{1.18a}
& |2\pi \n+\bxi|^\alpha \ge |\bxi|^\alpha, \quad \bxi\in\widetilde\Omega,\ \ \n\in\Z^d,
\\
\label{1.18b}
& |2\pi \n+\bxi|^\alpha \ge \pi^\alpha, \quad \bxi\in\widetilde\Omega,\ \ \n\in\Z^d \setminus{\mathbf{0}},
\\
\label{1.18c}
& \min_{\n\in\Z^d \setminus{\mathbf{0}}} |2\pi \n|^\alpha = (2\pi)^\alpha,
\end{align}
and by Lemma \ref{lem1.2} the quadratic form $a_0(\bxi;\alpha)$ admits the following lower bounds:
\begin{align}
\label{e1.17}
 a_0(\bxi;\alpha)[u,u]& \ge c_0(d,\alpha) |\bxi|^{\alpha} \| u \|_{L_2(\Omega)}^2,\ \ u \in \wt{H}^\gamma(\Omega), \quad \bxi\in\widetilde\Omega,
 \\
 \label{e1.18}
 a_0(\bxi;\alpha)[u,u] &\ge c_0(d,\alpha) \pi^{\alpha} \| u \|_{L_2(\Omega)}^2,\ \ u \in \wt{H}^\gamma(\Omega), \ \ \int_\Omega u(\x)\,d\x =0, \quad \bxi\in\widetilde\Omega.
 \end{align}
Combining \eqref{e1.10}, \eqref{e1.17} and  \eqref{e1.18} we obtain, see also \cite[Proposition 1.5]{JPSS24},

\begin{proposition}[\cite{JPSS24}]
\label{prop1.5}
Let conditions  \eqref{e1.1} and \eqref{e1.2} hold,  and assume that  $1 < \alpha < 2$.
Then the form  \eqref{e1.9} admits the lower bounds
\begin{align*}
 a(\bxi;\alpha,\mu)[u,u] &\ge \mu_- c_0(d,\alpha) |\bxi|^{\alpha} \| u \|_{L_2(\Omega)}^2,\ \ u \in \wt{H}^\gamma(\Omega), \quad \bxi\in\widetilde\Omega,
 \\
 a(\bxi;\alpha,\mu)[u,u] &\ge \mu_- c_0(d,\alpha) \pi^{\alpha} \| u \|_{L_2(\Omega)}^2,\ \ u \in \wt{H}^\gamma(\Omega), \quad \int_\Omega u(\x)\,d\x =0, \quad  \bxi\in\widetilde\Omega.
 \end{align*}
\end{proposition}

\section{Representation for the difference of quadratic forms $a(\bxi)$ and $a(\mathbf{0})$}

\subsection{
Estimate of the difference between quadratic forms $a(\bxi)$ and $a(\mathbf{0})$}
The difference of the quadratic forms $a(\bxi)$ and $a(\mathbf{0})$ was estimated in \cite[Lemma 2.3]{JPSS24}.
The corresponding statement reads.
\begin{lemma}[\cite{JPSS24}]
\label{lem2.4}
Let conditions \eqref{e1.1} and \eqref{e1.2} be satisfied, and assume that $1< \alpha <2$.
Then for the form  $a(\bxi)$ defined in \eqref{e1.9} the following estimate holds:
\begin{equation*}
\left|a(\bxi) [u,u] - a(\mathbf{0}) [u,u] \right| \le \check{c}(d,\alpha) |\bxi| \left( a(\mathbf{0}) [u,u] + \mu_+ \|u\|^2_{L_2(\Omega)}\right), \quad u\in \wt{H}^\gamma(\Omega), \quad \bxi \in \wt{\Omega}.
\end{equation*}
\end{lemma}

\begin{remark}
In \cite{JPSS24} an explicit formula for the constant $\check{c}(d,\alpha)$  is provided. Analysing this formula one can observe that $\check{c}(d,\alpha) \to \infty$, as $\alpha \to 1$ and as  $\alpha \to 2$.
\end{remark}

Our goal is to collect terms of order $|\bxi|$ in the difference $a(\bxi) [u,u] - a(\mathbf{0}) [u,u]$ and  estimate the remainder. To this end, it is convenient to consider the contributions of neighbourhoods of zero and infinity separately.
Letting
\begin{equation*}
b_1(\z)  = \frac{\1_{B_1(\mathbf{0})}(\z)}{|\z|^{d+\alpha}},
\quad  b_2(\z)  = \frac{\1_{\R^d \setminus B_1(\mathbf{0})}(\z)}{|\z|^{d+\alpha}}, \quad \z \in \R^d,
\end{equation*}
we represent the form $a(\bxi) [u,u]$ as
\begin{equation}
\label{a=a1+a2}
a(\bxi) [u,u] = a_1(\bxi) [u,u] + a_2(\bxi)[u,u], \quad u\in \wt{H}^\gamma(\Omega),
\end{equation}
with
 \begin{align}
 \label{a1=}
a_1(\bxi) [u,u]  := \frac{1}{2} \int_{\R^d} d\y \int_{\Omega} d\x\, \mu(\x,\y) b_1(\x - \y)
 \bigl| u(\x)- e^{i \langle \bxi,\y -\x\rangle} u(\y) \bigr|^2,\quad u\in \wt{H}^\gamma(\Omega),
 \\
 \label{a2=}
a_2(\bxi) [u,u]  := \frac{1}{2} \int_{\R^d} d\y \int_{\Omega} d\x\, \mu(\x,\y) b_2(\x - \y)
\bigl| u(\x)- e^{i \langle \bxi,\y -\x\rangle} u(\y) \bigr|^2,\quad  u\in \wt{H}^\gamma(\Omega).
\end{align}

\subsection{Representation of the quadratic form $a_1(\bxi)$}

\begin{lemma}
\label{lem2.2_new}
Let conditions \eqref{e1.1} and \eqref{e1.2} be fulfilled, and assume that  $1< \alpha <2$.
Then the form $a_1(\bxi)$ introduced in \eqref{a1=} admits the following representation{\rm :}
\begin{equation}
\label{2.5_new}
a_1(\bxi) [u,u] =  a_1(\mathbf{0}) [u,u] + \sum_{j=1}^d \xi_j a_1^{(j)}[u,u] +  \wt{a}_1(\bxi) [u,u], \quad
u\in \wt{H}^\gamma(\Omega), \quad \bxi \in \wt{\Omega}.
\end{equation}
Here
\begin{equation}
\label{2.6_new}
a_1^{(j)} [u,u] := - \int_{\R^d} d\y \int_{\Omega} d\x\, \mu(\x,\y) b_1(\x - \y) (x_j-y_j)
 \operatorname{Re} \left( i u(\x) \overline{u(\y)} \right) ,\quad u\in \wt{H}^\gamma(\Omega).
\end{equation}
Moreover,
 \begin{align}
 \label{a1j_le}
\bigl| a_1^{(j)} [u,u] \bigr| &\le c_1(d,\alpha)  \left( a_1(\mathbf{0}) [u,u] + \mu_+ \|u\|^2_{L_2(\Omega)}\right), \quad u\in \wt{H}^\gamma(\Omega), \quad j=1,\dots,d,
\\
\label{a1tilde}
\left| \wt{a}_1(\bxi) [u,u] \right| &\le \wt{c}_1(d,\alpha) |\bxi |^2 \left( a_1(\mathbf{0}) [u,u] + \mu_+ \|u\|^2_{L_2(\Omega)}\right), \quad u\in \wt{H}^\gamma(\Omega), \quad  \bxi \in \wt{\Omega}.
\end{align}
\end{lemma}

\begin{proof}[Proof]
We begin the proof by recalling an elementary representation
 \begin{equation}
 \label{exp=}
e^{i\lambda} = 1 + i \lambda + \lambda^2 F(\lambda), \quad \lambda \in \R,
\end{equation}
where the function $F$ satisfies the inequality
 \begin{equation}
 \label{2.10_new}
 | F(\lambda)| \le \frac{1}{2}, \quad \lambda \in \R.
\end{equation}
According to \eqref{a1=} and \eqref{exp=}, for functions  $u\in \wt{H}^\gamma(\Omega)$  we have
 \begin{equation*}
 \begin{aligned}
a_1(\bxi) [u,u] =&
 \frac{1}{2} \int_{\R^d} d\y \int_{\Omega} d\x\, \mu(\x,\y) b_1(\x - \y)
 \\
&\qquad\qquad \times \left| u(\x)- u(\y) + i \langle \bxi,\x -\y \rangle u(\y) -  \langle \bxi,\x -\y \rangle^2 F(\langle \bxi,\y -\x \rangle)  u(\y)\right|^2.
 \end{aligned}
\end{equation*}
Therefore,
 \begin{equation}
\label{2.11_new}
a_1(\bxi) [u,u] = a_1(\mathbf{0}) [u,u] + \sum_{j=1}^d \xi_j a_1^{(j)}[u,u] + \wt{a}_1(\bxi) [u,u],
\end{equation}
where
 \begin{align}
 \label{2.12_new}
 a_1^{(j)}[u,u] & = - \int_{\R^d}\! d\y\! \int_{\Omega} d\x\, \mu(\x,\y) b_1(\x - \y) (x_j-y_j) \operatorname{Re} \left( i
 (u(\x) - u(\y)) \overline{u(\y)} \right)\!\!,
 \\
 \label{2.13_new}
 \wt{a}_1(\bxi) [u,u] & =  \wt{a}_1'(\bxi) [u,u] + \wt{a}_1''(\bxi) [u,u],
 \\
 \label{2.14_new}
  \wt{a}_1'(\bxi) [u,u] &\! =  - \int_{\R^d}\! d\y\! \int_{\Omega}\! d\x\, \mu(\x,\y) b_1(\x - \y) \langle \bxi,\x -\y \rangle^2
   \operatorname{Re} \left(  \overline{F(  \langle \bxi,\y -\x \rangle )}
 (u(\x) - u(\y)) \overline{u(\y)} \right)\!\!,
  \\
  \label{2.15_new}
   \wt{a}_1''(\bxi) [u,u] &\! =  \frac{1}{2} \int_{\R^d}\! d\y\! \int_{\Omega} d\x\, \mu(\x,\y) b_1(\x - \y)
 \left|  i \langle \bxi,\x -\y \rangle -  \langle \bxi,\x -\y \rangle^2 F(\langle \bxi,\y -\x \rangle) \right|^2 |u(\y)|^2\!\!.
 \end{align}
 Now the representations in \eqref{2.5_new} and \eqref{2.6_new} follow from \eqref{2.11_new},  \eqref{2.12_new}
 and a trivial relation $ \operatorname{Re} (-i |u(\y)|^2)=0$.

We rearrange the form in \eqref{2.12_new} taking into account the periodicity of functions $\mu(\x,\y)$ и $u(\y)$:
\begin{equation}
\label{2.16_new}
\begin{aligned}
 &a_1^{(j)}[u,u]  = - \sum_{\n \in \Z^d} \int_{\Omega + \n} d\y \int_{\Omega} d\x\, \mu(\x,\y) b_1(\x - \y) (x_j-y_j) \operatorname{Re} \left( i  (u(\x) - u(\y)) \overline{u(\y)} \right)
\\
 &= - \sum_{\n \in \Z^d} \int_{\Omega} d\y \int_{\Omega} d\x\, \mu(\x,\y) b_1(\x - \y - \n) (x_j-y_j - n_j) \operatorname{Re} \left( i  (u(\x) - u(\y)) \overline{u(\y)} \right)
 \\
 &= - \operatorname{Re} \left( i \int_{\Omega} d\y\, \overline{u(\y)} \int_{\R^d} d\x\, \mu(\x,\y) b_1(\x - \y) (x_j-y_j)   (u(\x) - u(\y)) \right).
 \end{aligned}
\end{equation}
By the Cauchy-Schwartz inequality
\begin{equation}
\label{2.17_newnew}
\begin{aligned}
&\left| \int_{\R^d} d\x\, \mu(\x,\y) b_1(\x - \y) (x_j-y_j)   (u(\x) - u(\y))   \right|^2
\\
&\le \left( \int_{\R^d} d\x\, \mu(\x,\y) b_1(\x - \y) (x_j-y_j)^2 \right) \left( \int_{\R^d} d\x\, \mu(\x,\y) b_1(\x - \y)
|u(\x) - u(\y)|^2 \right).
 \end{aligned}
\end{equation}
Considering  \eqref{a1=} with $\bxi = \mathbf{0}$  and the relation
$$
\int_{\R^d} d\x\, \mu(\x,\y) b_1(\x - \y) (x_j-y_j)^2 \le \mu_+ \int_{|\z| <1} \frac{d\z}{ |\z|^{d+\alpha -2}}
= \mu_+ \omega_d \int_0^1 \frac{dr}{r^{\alpha -1}} = \mu_+ \frac{\omega_d}{2-\alpha},
$$
from \eqref{2.16_new} and \eqref{2.17_newnew} we obtain
$$
\begin{aligned}
&\bigl| a_1^{(j)}[u,u] \bigr| \le \Big(\mu_+ \frac{\omega_d}{2-\alpha}\Big)^{1/2} \int_\Omega d\y |u(\y)|
\left( \int_{\R^d} d\x\, \mu(\x,\y) b_1(\x - \y) |u(\x) - u(\y)|^2 \right)^{1/2}
\\
&\le \left(\mu_+ \frac{\omega_d}{2-\alpha}\right)^{1/2} \|u\|_{L_2(\Omega)} \left(2 a_1(\mathbf{0})[u,u] \right)^{1/2}
\le \left(\frac{\omega_d}{2(2-\alpha)}\right)^{1/2} \left( a_1(\mathbf{0})[u,u] + \mu_+  \|u\|_{L_2(\Omega)}^2 \right).
\end{aligned}
$$
This yields inequality \eqref{a1j_le} with a constant $c_1(d,\alpha) = \omega_d^{1/2} (2(2-\alpha))^{-1/2}$.

We turn to estimating the form in \eqref{2.14_new}.  Similarly to \eqref{2.16_new} we have
$$
  \wt{a}_1'(\bxi) [u,u] \! =  - \operatorname{Re} \left( \int_{\Omega}\!\! d\y\, \overline{u(\y)}\! \int_{\R^d}\!\! d\x\, \mu(\x,\y) b_1(\x - \y)  \langle \bxi,\x -\y \rangle^2
  \overline{F(  \langle \bxi,\y -\x \rangle )} (u(\x) - u(\y))\! \right)\!.
$$
Using the Cauchy-Schwartz inequality and considering \eqref{2.10_new} one can estimate the integral over $\mathbb R^d$
on the right-hand side as follows:
$$
\begin{aligned}
&\left| \int_{\R^d} d\x\, \mu(\x,\y) b_1(\x - \y)  \langle \bxi,\x -\y \rangle^2
  \overline{F(  \langle \bxi,\y -\x \rangle )} (u(\x) - u(\y)) \right|^2
  \\
   &\le
  \left( \frac{1}{4} \int_{\R^d} d\x\, \mu(\x,\y) b_1(\x - \y)  \langle \bxi,\x -\y \rangle^4 \right)
  \left( \int_{\R^d} d\x\, \mu(\x,\y) b_1(\x - \y) |u(\x) - u(\y)|^2\right).
  \end{aligned}
$$
Combining  \eqref{a1=} with $\bxi = \mathbf{0}$ and the relation
$$
\begin{array}{rl}
\displaystyle
\frac{1}{4}\int_{\R^d} d\x\, \mu(\x,\y) b_1(\x - \y)\langle \bxi,\x -\y \rangle^4 \!\!&
\displaystyle
\!
 \le \frac{\mu_+}{4}|\bxi|^4 \int_{|\z| <1} \frac{d\z}{ |\z|^{d+\alpha -4}}=\\[3mm]
&\displaystyle
= \frac{\mu_+}{4}|\bxi|^4 \omega_d \int_0^1 \frac{dr}{r^{\alpha -3}} = \mu_+|\bxi|^4 \frac{\omega_d}{4(4-\alpha)},
\end{array}
$$
we conclude that
\begin{equation}
\label{2.18_new}
\begin{aligned}
&  \left| \wt{a}_1'(\bxi) [u,u] \right| \le |\bxi|^2\! \left( \frac{\mu_+ \omega_d}{4(4-\alpha)} \right)^{\frac12}\!\!\!
\int_\Omega \!d\y |u(\y)| \left(\int_{\R^d}\!\! d\x\, \mu(\x,\y) b_1(\x - \y) |u(\x) - u(\y)|^2 \!  \right)^{\frac12}
\\
&\le  |\bxi|^2 \Big( \frac{\mu_+ \omega_d}{4(4-\alpha)}\! \Big)^{\frac12}\! \|u\|\big._{L_2(\Omega)}
\left( 2 a_1(\mathbf{0})[u,u]\right)^{1/2}\! \le  \wt{c}_1'(d,\alpha) |\bxi|^2\!
\left(  a_1(\mathbf{0})[u,u] + \mu_+ \|u\|^2_{L_2(\Omega)}\right)
 \end{aligned}
\end{equation}
with a constant $\wt{c}_1'(d,\alpha) = \frac{1}{2\sqrt{2}}  \omega_d^{1/2}(4-\alpha)^{-1/2}$.

It remains to estimate the form in  \eqref{2.15_new}. In the same way as in  \eqref{2.16_new} we have
$$
   \wt{a}_1''(\bxi) [u,u] =  \frac{1}{2} \intop_{\Omega} d\y  |u(\y)|^2 \intop_{\R^d} d\x\, \mu(\x,\y) b_1(\x - \y)
 \left|  i \langle \bxi,\x -\y \rangle -  \langle \bxi,\x -\y \rangle^2 F(\langle \bxi,\y -\x \rangle) \right|^2.
$$
In view of \eqref{2.10_new}, the integral over $\R^d$ on the right-hand side admits the estimate
$$
\begin{aligned}
& \int_{\R^d} d\x\, \mu(\x,\y) b_1(\x - \y)
 \left|  i \langle \bxi,\x -\y \rangle -  \langle \bxi,\x -\y \rangle^2 F(\langle \bxi,\y -\x \rangle) \right|^2
 \\
 & \le \mu_+ \int_{|\z| <1} d\z\, \left(\frac{2 |\bxi|^2}{|\z|^{d+\alpha-2}} + \frac{|\bxi|^4}{2 |\z|^{d+\alpha -4}}\right)
 = \mu_+ \omega_d \left( \frac{2|\bxi|^2}{2-\alpha} + \frac{|\bxi|^4}{2(4-\alpha)} \right).
 \end{aligned}
$$
Since $|\bxi| \le \pi \sqrt{d}$ for all $\bxi \in \wt{\Omega}$, we finally arrive at the estimate
\begin{equation}
\label{2.19_new}
 \left| \wt{a}_1''(\bxi) [u,u] \right| \le \wt{c}_1''(d,\alpha) |\bxi|^2 \mu_+ \| u\|^2_{L_2(\Omega)},
\end{equation}
where $\wt{c}_1''(d,\alpha) = \omega_d ((2-\alpha)^{-1} + \frac{1}{4} \pi^2 d (4-\alpha)^{-1})$.

Combining  \eqref{2.13_new}, \eqref{2.18_new} and \eqref{2.19_new} yields the desired estimate \eqref{a1tilde}
with a constant $\wt{c}_1(d,\alpha) = \wt{c}_1'(d,\alpha) + \wt{c}_1''(d,\alpha)$.
\end{proof}

\begin{remark}
\label{rem2.4}
From the explicit expressions for the constants ${c}_1(d,\alpha)$ and $\wt{c}_1(d,\alpha)$ it is clear that
\hbox{${c}_1(d,\alpha)\to \infty$} and $\wt{c}_1(d,\alpha) \to \infty$, as $\alpha \to 2$.
\end{remark}

\subsection{Representation of the quadratic form $a_2(\bxi)$}
After straightforward rearrangements the quadratic form in  \eqref{a2=} takes the form
\begin{equation}
\label{2.22_new}
\begin{aligned}
a_2(\bxi)[u,u] &=  \frac{1}{2} \int_{\R^d} d\y \int_{\Omega} d\x\, \mu(\x,\y) b_2(\x - \y) |u(\x)|^2
+  \frac{1}{2} \int_{\R^d} d\y \int_{\Omega} d\x\, \mu(\x,\y) b_2(\x - \y) |u(\y)|^2
\\
&-  \operatorname{Re} \int_{\R^d} d\y \int_{\Omega} d\x\, \mu(\x,\y) b_2(\x - \y) e^{i \langle \bxi,\x -\y\rangle}
  u(\x) \overline{u(\y)}.
  \end{aligned}
\end{equation}
Letting
\begin{equation*}
V(\x) := \int_{\R^d} \mu(\x,\y) b_2(\x - \y)\,d\y
\end{equation*}
and considering the relations
$$
\int_{\R^d}  b_2(\z) \,d\z = \int_{|\z| > 1}  \frac{d\z}{|\z|^{d+\alpha}} =
 \omega_d \int_1^\infty \frac{dr}{r^{1+\alpha}} = \frac{\omega_d}{\alpha},
$$
we conclude that $V(\x)$ satisfies the estimates
\begin{equation}
\label{2.23a_new}
\mu_- \frac{\omega_d}{\alpha} \le V(\x) \le \mu_+ \frac{\omega_d}{\alpha},\quad \x \in \Omega.
\end{equation}
Clearly, the first term on the right-hand side of \eqref{2.22_new} is equal to
$
 \frac{1}{2} \int_{\Omega}  V(\x) |u(\x)|^2\, d\x.
$
The second term can be rearranged in the same way as in \eqref{2.16_new}:
$$
\begin{aligned}
 & \frac{1}{2} \int_{\R^d} d\y \int_{\Omega} d\x\, \mu(\x,\y) b_2(\x - \y) |u(\y)|^2
 \\
& = \frac{1}{2} \int_{\Omega} d\y \int_{\R^d} d\x\, \mu(\x,\y) b_2(\x - \y) |u(\y)|^2 =&
   \frac{1}{2} \int_{\Omega}  V(\y)  |u(\y)|^2 \,d\y.
   \end{aligned}
$$
The third term admits the representation
$$
\begin{aligned}
&\operatorname{Re} \int_{\R^d} d\y \int_{\Omega} d\x\, \mu(\x,\y) b_2(\x - \y) e^{i \langle \bxi,\x -\y\rangle}
  u(\x) \overline{u(\y)}
  \\
  &= \operatorname{Re}  \int_{\Omega} d\y \int_{\Omega} d\x\, \mu(\x,\y) \wt{b}_2(\bxi,\x - \y)
  u(\x) \overline{u(\y)} =  \operatorname{Re} \,(K(\bxi) u,u)_{L_2(\Omega)},
  \end{aligned}
$$
where
$$
\wt{b}_2(\bxi,\z) := \sum_{\n \in \Z^d} b_2(\z+\n) e^{i \langle \bxi,\z +\n\rangle},
$$
and $K(\bxi)$ is the integral operator given by
$$
(K(\bxi) u)(\y) = \int_{\Omega} d\x\, \mu(\x,\y) \wt{b}_2(\bxi,\x - \y)  u(\x).
$$
The operator $K(\bxi)$ is bounded in $L_2(\Omega)$, its norm can be estimated by means of the Shur test, see,
for example, \cite[Appendix A]{Grafakos}:
\begin{equation}
\label{2.23b_new}
\| K(\bxi) \|_{L_2(\Omega) \to L_2(\Omega)}^2 \le
\left( \sup_{\y \in \Omega} \int_\Omega \mu(\x,\y) | \wt{b}_2(\bxi,\x - \y) | \,d\x \right)
\left( \sup_{\x \in \Omega} \int_\Omega \mu(\x,\y)  |\wt{b}_2(\bxi,\x - \y)| \,d\y \right).
\end{equation}
The first term on the right-hand side can be estimated as follows:
\begin{equation}
\label{2.23c_new}
 \sup_{\y \in \Omega} \int_\Omega \mu(\x,\y) |\wt{b}_2(\bxi,\x - \y)| \,d\x \le \mu_+
\sup_{\y \in \Omega} \int_{\Omega} | \wt{b}_2(\bxi,\x - \y)|\,d\x \le \mu_+  \int_{\R^d}  {b}_2(\z)\,d\z = \mu_+ \frac{\omega_d}{\alpha}.
\end{equation}
The second term admits a similar estimate, and we obtain
\begin{equation}
\label{2.23d_new}
\| K(\bxi) \|_{L_2(\Omega) \to L_2(\Omega)} \le  \mu_+ \frac{\omega_d}{\alpha}.
\end{equation}
Finally, we have
\begin{equation}
\label{2.24_new}
a_2(\bxi)[u,u] =  \int_{\Omega}  V(\x)  |u(\x)|^2\,d\x
-  \operatorname{Re} \, (K(\bxi) u,u)_{L_2(\Omega)}.
\end{equation}
Therefore, the form $a_2(\bxi)[u,u]$ is bounded and can be extended to the whole $ L_2(\Omega)$.
Moreover, thanks to \eqref{2.23a_new}, \eqref{2.23d_new} and \eqref{2.24_new} we have
\begin{equation*}
a_2(\bxi)[u,u] \le 2 \mu_+ \frac{\omega_d}{\alpha} \|u\|_{L_2(\Omega)}^2, \quad u\in L_2(\Omega).
\end{equation*}

\begin{lemma}
\label{lem2.3_new}
Let conditions  \eqref{e1.1}, \eqref{e1.2} be fulfilled, and assume that $1< \alpha <2$.
Then the form $a_2(\bxi)$  defined in \eqref{a2=} can be extended to the whole space $ L_2(\Omega)$, and
the following representation is valid{\rm :}
\begin{equation}
\label{2.20_new}
a_2(\bxi) [u,u] =  a_2(\mathbf{0}) [u,u] + \sum_{j=1}^d \xi_j a_2^{(j)}[u,u] +  \wt{a}_2(\bxi) [u,u], \quad u \in L_2(\Omega),\quad \bxi \in \wt{\Omega};
\end{equation}
here
\begin{equation}
\label{2.21_new}
a_2^{(j)} [u,u] := - \int_{\R^d} d\y \int_{\Omega} d\x\, \mu(\x,\y) b_2(\x - \y) (x_j-y_j)
 \operatorname{Re} \left( i u(\x) \overline{u(\y)} \right) ,\quad u\in L_2(\Omega).
\end{equation}
Moreover,
 \begin{align}
 \label{a2j_le}
\bigl| a_2^{(j)} [u,u] \bigr| & \le c_2(d,\alpha) \mu_+ \|u\|^2_{L_2(\Omega)}, \quad u\in L_2(\Omega), \quad j=1,\dots,d,
\\
\label{a2tilde}
\left| \wt{a}_2(\bxi) [u,u] \right|  & \le \wt{c}_2(d,\alpha) |\bxi |^\alpha  \mu_+ \|u\|^2_{L_2(\Omega)}, \quad u\in L_2(\Omega), \quad \bxi \in \wt{\Omega}.
\end{align}
\end{lemma}

\begin{proof}[Proof]
By \eqref{2.24_new},
 \begin{equation}
 \label{2.25}
 \begin{aligned}
& a_2(\bxi) [u,u] - a_2(\mathbf{0}) [u,u] = - \operatorname{Re}\, ((K(\bxi) - K(\mathbf{0}))u,u)_{L_2(\Omega)}
 \\
 &= - \operatorname{Re} \sum_{\n \in \Z^d} \int_{\Omega} d\y \int_{\Omega} d\x\, \mu(\x,\y) b_2(\x-\y + \n)
 ( e^{i \langle \bxi, \x - \y + \n\rangle}- 1) u(\x) \overline{u(\y)}.
 \end{aligned}
 \end{equation}
We also use an elementary expansion
 \begin{equation}
 \label{2.26_new}
e^{i\lambda} = 1 + i \lambda + \frac{|\lambda|^\alpha F_p(\lambda)}{(1 + |\operatorname{ln} |\lambda || )^p}, \quad \lambda \in \R, \quad p>1,
\end{equation}
where $F_p \in L_\infty(\R)$ and  satisfies the uniform in  $\lambda$ estimate
 \begin{equation}
 \label{2.27_new}
 | F_p(\lambda)| \le \max \left\{ \sup_{0< x \le 1} \frac{x^{2-\alpha}}{2} (1 + |\operatorname{ln} x | )^p ,
 \sup_{1 \le  x  < \infty} \frac{(2 + x)}{x^\alpha} (1 + |\operatorname{ln} x | )^p
    \right\} =: \mathfrak{c}(p,\alpha), \quad \lambda \in \R.
\end{equation}
Combining \eqref{2.25} and  \eqref{2.26_new} one has
 \begin{equation}
 \label{2.28_new}
  a_2(\bxi) [u,u] - a_2(\mathbf{0}) [u,u] =
  \sum_{j=1}^d \xi_j a^{(j)}_2[u,u] + \wt{a}_2(\bxi) [u,u]
 \end{equation}
 with
 \begin{equation}
 \label{2.29_new}
 \begin{aligned}
 a^{(j)}_2[u,u] &= - \sum_{\n \in \Z^d} \int_{\Omega} d\y \int_{\Omega} d\x \, \mu(\x,\y)b_2(\x-\y + \n)
 (x_j - y_j + n_j)  \operatorname{Re} (i  u(\x) \overline{u(\y)})
 \\
 &=- \int_{\R^d} d\y \int_{\Omega} d\x\, \mu(\x,\y) b_2(\x - \y) (x_j-y_j)
 \operatorname{Re} \left( i u(\x) \overline{u(\y)} \right),
 \end{aligned}
 \end{equation}
 \begin{equation}
 \label{2.30_new}
 \begin{aligned}
 \wt{a}_2(\bxi)[u,u] &= - \sum_{\n \in \Z^d} \int_{\Omega} d\y \int_{\Omega} d\x \, \mu(\x,\y) b_2(\x-\y + \n)
\frac{ |\langle \bxi, \x - \y + \n\rangle|^\alpha }
{(1 + |\operatorname{ln} |\langle \bxi, \x - \y + \n\rangle| | )^p}
 \\
 &\qquad \qquad \qquad \times \operatorname{Re} \bigl(  {F_p (\langle \bxi, \x - \y + \n\rangle)} u(\x) \overline{u(\y)} \bigr).
 \end{aligned}
 \end{equation}
Now representations \eqref{2.20_new} and  \eqref{2.21_new} can be deduced from  \eqref{2.28_new}, \eqref{2.29_new}.

In order to estimate the form in \eqref{2.29_new} we rewrite it as follows:
 \begin{equation}
 \label{2.31_new}
 a^{(j)}_2[u,u] = - \operatorname{Re}\,( i(K_{2,j} u,u)_{L_2(\Omega)} ),
 \end{equation}
 where
 $$
 (K_{2,j} u)(\y) = \int_\Omega \mu(\x,\y) \wt{b}_{2,j}(\x -\y) u(\x) \,d\x,
 $$
 $$
  \wt{b}_{2,j}(\z) :=  \sum_{\n \in \Z^d}  b_2(\z + \n) (z_j + n_j).
 $$
 According to the Shur test, the integral operator $K_{2,j}$ is bounded, and
(cf. \eqref{2.23b_new}--\eqref{2.23d_new})
\begin{equation}
\label{2.31b_new}
\| K_{2,j} \|_{L_2(\Omega) \to L_2(\Omega)}^2 \le
\left( \sup_{\y \in \Omega}\! \int_\Omega\! \mu(\x,\y) |\wt{b}_{2,j}(\x - \y)| \,d\x\! \right)
\left( \sup_{\x \in \Omega}\! \int_\Omega\! \mu(\x,\y) |\wt{b}_{2,j}(\x - \y)| \,d\y\! \right)\!.
\end{equation}
Considering the estimate
$$
 \begin{aligned}
 \sup_{\y \in \Omega} \int_\Omega  \mu(\x,\y) \bigl| \wt{b}_{2,j}(\x -\y) \bigr|\,d\x \le
  \sup_{\y \in \Omega} \mu_+ \sum_{\n \in \Z^d} \int_\Omega
 b_2(\x - \y+\n) |x_j -y_j + n_j|  \,d\x
 \\
 = \mu_+ \intop_{\R^d}  b_2(\z) |z_j| \,d\z \le \mu_+ \int_{|\z| >1} \frac{d\z}{|\z|^{d+\alpha -1}}
 = \mu_+ \omega_d \int_1^\infty \frac{dr}{r^\alpha} = \mu_+ \frac{\omega_d}{\alpha -1}
 \end{aligned}
 $$
 and a similar estimate for the second integral on the right-hand side of \eqref{2.31b_new}, we conclude that
 estimate \eqref{a2j_le} holds with the constant $c_2(d,\alpha) = \frac{\omega_d}{\alpha -1}$.

We proceed with estimating the form   \eqref{2.30_new}, which can be written as
 \begin{equation}
 \label{2.32_new}
 \wt{a}_2(\bxi)[u,u] = - \operatorname{Re}\, (\wt{K}_{2}(\bxi) u,u)_{L_2(\Omega)}
 \end{equation}
 with
 $$
 (\wt{K}_{2}(\bxi) u)(\y) = \int_\Omega \mu(\x,\y) T(\bxi, \x -\y) u(\x)\,d\x,
 $$
 and
 $$
  T(\bxi, \z) :=  \sum_{\n \in \Z^d}  b_2(\z + \n) \frac{ |\langle \bxi, \z + \n\rangle|^\alpha {F_p(\langle \bxi, \z+\n \rangle)}}
{(1 + |\operatorname{ln} |\langle \bxi, \z + \n\rangle| | )^p}.
 $$
 The norm of the integral operator  $\wt{K}_{2}(\bxi)$ can be estimated with the help of the Shur test:
\begin{equation}
\label{2.32b_new}
\| \wt{K}_{2}(\bxi)  \|_{L_2(\Omega) \to L_2(\Omega)}^2 \le
\left( \sup_{\y \in \Omega} \int_\Omega \mu(\x,\y) | T(\bxi, \x - \y)| \,d\x \right)
\left( \sup_{\x \in \Omega} \int_\Omega \mu(\x,\y)  | T(\bxi, \x - \y)| \,d\y \right).
\end{equation}
By  \eqref{2.27_new},
 \begin{equation}
 \label{2.33_new}
\begin{aligned}
& \sup_{\y \in \Omega} \int_\Omega \mu(\x,\y) |T(\bxi, \x -\y)| \,d\x \le \mathfrak{c}(p,\alpha) \mu_+ \int_{\R^d}  b_2(\z) \frac{ |\langle \bxi, \z \rangle|^\alpha } {(1 + |\operatorname{ln} |\langle \bxi, \z \rangle| | )^p}\,d\z
\\
 & = \mathfrak{c}(p,\alpha)
\mu_+ |\bxi|^{\alpha}\!\!\! \intop_{|\z| >1}\!\! \frac{|z_1|^\alpha d\z}{|\z|^{d+\alpha} (1 + |\operatorname{ln} (|\bxi| |z_1|) |  )^p}
\le  \mathfrak{c}(p,\alpha)
\mu_+ |\bxi|^{\alpha}\!\!\!\! \intop_{|\z| >1}\!\!\! \frac{ d\z}{|\z|^{d} (1 + |\operatorname{ln} (|\bxi| |z_1|) |  )^p}.
\end{aligned}
 \end{equation}
 Here, for $\bxi \ne \mathbf{0}$, the new coordinate system is chosen in such a way that its first axis is directed along $\bxi$.
 Then  $\langle \bxi, \z \rangle = |\bxi| z_1$.
In order to estimate the last integral in  \eqref{2.33_new}, we let $\z=(z_1,\z'), \z' =(z_2,\dots,z_n)$, and
observe that the cylinder
$$
\left\{\z\in \R^d:\  |z_1| < \frac{1}{\sqrt{2}}, \ |\z'| < \frac{1}{\sqrt{2}} \right\}
$$
is contained in the ball $B_1(\mathbf{0})$ and, therefore, the integral over the complement to the ball is bounded from above by that over the complement to the cylinder. This yields
$$
\begin{aligned}
\intop_{|\z| >1} \frac{ d\z}{|\z|^{d} (1 + |\operatorname{ln} (|\bxi| |z_1|) |  )^p} \le
\intop_{|z_1| < \frac{1}{\sqrt{2}} } dz_1 \intop_{|\z'|> \frac{1}{\sqrt{2}}} \frac{d\z'}{|\z|^{d} (1 + |\operatorname{ln} (|\bxi| |z_1|) |  )^p}
\\
+ \intop_{|z_1| > \frac{1}{\sqrt{2}} } dz_1 \intop_{\R^{d-1}} \frac{d\z'}{|\z|^{d} (1 + |\operatorname{ln} (|\bxi| |z_1|) |  )^p}.
\end{aligned}
$$
The first integral on the right-hand side does not exceed the quantity
$$
 \sqrt{2} \intop_{|\z'| > \frac{1}{\sqrt{2}}} \frac{ d\z'}{|\z'|^{d}}  =  \sqrt{2}\, \omega_{d-1}
 \intop_{\frac{1}{\sqrt{2}}}^\infty \frac{ dr}{r^{2}} = 2 \omega_{d-1}.
$$
Let us estimate the second one:
$$
\begin{aligned}
& \intop_{|z_1| > \frac{1}{\sqrt{2}} } dz_1 \intop_{\R^{d-1}} \frac{d\z'}{|\z|^{d} (1 + |\operatorname{ln} (|\bxi| |z_1|) |  )^p}
 =  \intop_{|z_1| > \frac{1}{\sqrt{2}} } \frac{dz_1}{(1 + |\operatorname{ln} (|\bxi| |z_1|) |  )^p}
 \intop_{\R^{d-1}} \frac{d\z'}{(z_1^2 + |\z'|^2 )^{d/2} }
 \\
 &=\!
2 \! \intop_{\frac{1}{\sqrt{2}} }^\infty\! \frac{dz_1}{z_1 (1 + |\operatorname{ln} (|\bxi| z_1) |  )^p}
 \intop_{\R^{d-1}}\!\! \frac{d\w'}{(1 + |\w'|^2 )^{\frac d2} }\le 2{\mathfrak c}'_d \intop_{0 }^\infty \frac{d\tau}{\tau
 (1+  |\operatorname{ln} \tau |)^p} = 2{\mathfrak c}'_d \intop_{\R }\! \frac{d t}{(1 + |t|)^p} = \frac{4{\mathfrak c}'_d}{p-1}.
 \end{aligned}
$$
Here we performed the change of variables $\z' = |z_1| \w'$, then  $\tau = |\bxi| z_1$ and $t = \operatorname{ln} \tau$;
also we used the notation
$$
{\mathfrak c}'_d =  \int_{\R^{d-1}} \frac{d\w'}{(1 + |\w'|^2 )^{d/2} }.
$$
Finally, we have
$$
\int_{|\z| >1} \frac{ d\z}{|\z|^{d} (1 + |\operatorname{ln} (|\bxi| |z_1|) |  )^p} \le
2 \omega_{d-1} + \frac{4{\mathfrak c}'_d}{p-1}
$$
and, due to  \eqref{2.33_new},
 \begin{equation*}
\sup_{\y \in \Omega} \int_\Omega \mu(\x,\y) |T(\bxi, \x -\y)|  \,d\x \le \hat{c}_2(p,d,\alpha) \mu_+ |\bxi|^{\alpha},\quad
\hat{c}_2(p,d,\alpha) = {\mathfrak c}(p,\alpha)  \Bigl(2 \omega_{d-1} + \frac{4{\mathfrak c}'_d}{p-1}\Bigr).
\end{equation*}
The second term on the right-hand side of  \eqref{2.32b_new} admits a similar estimate. Consequently,
the norm  $\| \wt{K}_2(\bxi)\|$ is estimated from above by the quantity $\hat{c}_2(p,d,\alpha) \mu_+ |\bxi|^{\alpha}$,
and, according to  \eqref{2.32_new},
$$
\wt{a}_2(\bxi)[u,u] \le  \hat{c}_2(p,d,\alpha) \mu_+ |\bxi|^{\alpha} \| u \|^2_{L_2(\Omega)}, \quad u \in L_2(\Omega).
$$
Letting $p=2$ in \eqref{2.26_new}, we derive from the last inequality the desired estimate \eqref{a2tilde}
with a constant $\wt{c}_2(d,\alpha) = \hat{c}_2(2,d,\alpha)$.
\end{proof}

\begin{remark}
\label{rem2.6}
Notice that  ${c}_2(d,\alpha)\to \infty$ as $\alpha \to 1$.
\end{remark}

\subsection{Representation for the quadratic form $a(\bxi)$}
Combining  \eqref{a=a1+a2} and Lemmata \ref{lem2.2_new}, \ref{lem2.3_new}, we obtain the following statement:

\begin{lemma}
\label{lem2.4_new}
Let conditions \eqref{e1.1}, \eqref{e1.2} be satisfied,  and assume that $1< \alpha <2$.
Then the form $a(\bxi)$ defined in \eqref{e1.9} admits the representation
\begin{equation*}
a(\bxi) [u,u] =  a(\mathbf{0}) [u,u] + \sum_{j=1}^d \xi_j a^{(j)}[u,u] +  \wt{a}(\bxi) [u,u],
\quad u\in \wt{H}^\gamma(\Omega), \quad \bxi \in \wt{\Omega}.
\end{equation*}
Here
\begin{equation}
\label{2.36_new}
a^{(j)} [u,u] := - \int_{\R^d} d\y \int_{\Omega} d\x\, \mu(\x,\y)  \frac{(x_j-y_j)}{|\x-\y|^{d+\alpha}}
 \operatorname{Re} \left( i u(\x) \overline{u(\y)} \right) ,\quad u\in \wt{H}^\gamma(\Omega).
\end{equation}
Moreover, the following estimates hold:
 \begin{align}
 \label{aj_le}
\bigl| a^{(j)} [u,u] \bigr| &\le c_3(d,\alpha)  \left( a(\mathbf{0}) [u,u] + \mu_+ \|u\|^2_{L_2(\Omega)}\right), \quad u\in \wt{H}^\gamma(\Omega), \quad j=1,\dots,d,
\\
\label{atilde}
\left| \wt{a}(\bxi) [u,u] \right| &\le \wt{c}_3(d,\alpha) |\bxi |^\alpha \left( a(\mathbf{0}) [u,u] + \mu_+ \|u\|^2_{L_2(\Omega)}\right), \quad u\in \wt{H}^\gamma(\Omega), \quad  \bxi \in \wt{\Omega}.
\end{align}
\end{lemma}

\begin{remark}
The constants ${c}_3(d,\alpha)$, $\wt{c}_3(d,\alpha)$ can be expressed in terms of the constants from Lemmata
\emph{\ref{lem2.2_new} and \ref{lem2.3_new}:}   $c_3(d,\alpha) = c_1(d,\alpha) + c_2(d,\alpha)$,
$\wt{c}_3(d,\alpha) = \wt{c}_1(d,\alpha) (\pi \sqrt{d})^{2-\alpha} + \wt{c}_2(d,\alpha)$. According to Remarks
\emph{\ref{rem2.4} and \ref{rem2.6}},
$c_3(d,\alpha) \to \infty$ as $\alpha \to 1$ and as $\alpha \to 2$, while  $\wt{c}_3(d,\alpha) \to \infty$ as $\alpha \to 2$.
\end{remark}

\section{Threshold characteristics of L\'evy-type operators near the lower edge of the spectrum.}

\subsection{The edge of the spectrum of operator $\A(\bxi;\alpha,\mu)$}
Denote by  $\lambda_j(\bxi)$, $j \in \N$, the eigenvalues of the operator $\A(\bxi)$ enumerated in non-decreasing order
taking into account the multiplicities.
From \eqref{e1.10} and the variational principle for the  eigenvalues of $\A(\bxi)$ we have
\begin{equation}
\label{e1.20}
  \mu_- \lambda_j^0(\bxi) \le \lambda_j(\bxi) \le \mu_+ \lambda_j^0(\bxi),
 \quad  j \in \N, \quad  \bxi\in\widetilde\Omega.
 \end{equation}
Due to  diagonalization, the eigenpairs of the operator  $\A_0(\bxi)$ can be determined explicitly: the eigenvalues
are given by $c_0(d,\alpha) |2\pi \n + \bxi|^\alpha$, $\n \in \Z^d$, and the corresponding eigenfunctions are
$e^{2\pi i \langle \n, \x \rangle}$. The formula for the principal eigenvalue reads
\begin{equation}
\label{e1.21}
  \lambda_1^0(\bxi) = c_0(d,\alpha) |\bxi|^\alpha, \quad \bxi \in \wt{\Omega},
 \end{equation}
 and  $\mathbf{1}_{\Omega}$ is the corresponding eigenfunction.
Since
   $$
   |\bxi| < |2\pi \n + \bxi|, \quad \bxi \in \operatorname{Int} \wt{\Omega} = (-\pi,\pi)^d, \ \ \n \in \Z^d \setminus \mathbf{0},
   $$
then, for $\bxi \in (-\pi,\pi)^d$, the first eigenvalue of the operator $\A_0(\bxi)$ is simple, and the corresponding
eigenspace coincides with   $\mathcal{L}\{\mathbf{1}_{\Omega}\}$. By \eqref{1.18b}, \eqref{1.18c} the following
relations hold:
\begin{align}
\label{e1.22}
  \lambda_2^0(\bxi) &=  c_0(d,\alpha) \min_{\n \in \Z^d \setminus \mathbf{0}}
   |2\pi \n + \bxi|^\alpha \ge c_0(d,\alpha) \pi^\alpha, \quad \bxi \in \wt{\Omega},
   \\
   \label{e1.22a}
   \lambda_2^0(\mathbf{0}) &=  c_0(d,\alpha) (2\pi)^\alpha.
  \end{align}
As a consequence of  \eqref{e1.20}--\eqref{e1.22a} we have
\begin{align}
\label{e1.23}
 \mu_- c_0(d,\alpha) |\bxi|^\alpha \le \,&\lambda_1(\bxi) \le \mu_+ c_0(d,\alpha) |\bxi|^\alpha,
 \quad   \bxi\in\widetilde\Omega,
\\
\label{e1.24}
  &\lambda_2(\bxi) \ge \mu_- c_0(d,\alpha) \pi^\alpha =: d_0,
 \quad   \bxi\in\widetilde\Omega,
 \\
\label{e1.24a}
  &\lambda_2(\mathbf{0}) \ge \mu_- c_0(d,\alpha) (2\pi)^\alpha = 2^\alpha d_0.
\end{align}
By Lemma  \ref{lem1.4}, the lower edge of the spectrum of the operator $\A(\mathbf{0};\alpha,\mu)$ consists of
a simple isolated eigenvalue $\lambda_1(\mathbf{0}) =0$, and the corresponding eigenspace is $\mathcal{L}\{\mathbf{1}_{\Omega}\}$. Furthermore, by virtue of  \eqref{e1.24a}, the distance from
the point  $\lambda_1(\mathbf{0}) =0$ to the rest of the spectrum of the operator   $\A(\mathbf{0};\alpha,\mu)$
is not less than $2^\alpha d_0$.
Denote
\begin{equation}
\label{delta0}
\delta_{0}(\alpha,\mu):=  \pi \Big( \frac{\mu_{-}}{3 \mu_{+}}\Big)^{1/\alpha}.
\end{equation}
Since $\mu_{-}\leqslant \mu_+$, then $\delta_{0}(\alpha,\mu) < \pi$ and thus the ball $|\bxi| \le \delta_{0}(\alpha,\mu)$
is a subset of  $\widetilde\Omega$.
It follows from inequalities \eqref{e1.23}, \eqref{e1.24} that for  $|\bxi|\le \delta_0(\alpha, \mu)$
the first eigenvalue of the operator $\A(\bxi;\alpha,\mu)$ belongs to the interval $[0,d_0/3]$,
while the remaining part of the spectrum is situated on the semi-axis $[d_0,\infty)$.
Recalling that $d_0 := \mu_- c_0(d,\alpha) \pi^\alpha$ and $\delta_0(\alpha,\mu)$ is defined in \eqref{delta0},
 we arrive at the following statement:

\begin{proposition}
\label{prop2.1}
Let conditions  \eqref{e1.1}, \eqref{e1.2} be fulfilled, and assume that $1< \alpha < 2$.
Then, for  $|\bxi|\le\delta_{0}(\alpha,\mu)$ the spectrum of the operator
$\A(\bxi) = \A(\bxi;\alpha,\mu)$ on the segment  $[0,d_{0}/3]$ consists of one simple eigenvalue,
while the interval  $(d_{0}/3,d_{0})$ has no  common points with the spectrum of
 $\A(\bxi;\alpha,\mu)$.
\end{proposition}

\subsection{Contour $\Gamma$. Resolvent identity}
Denote by $F(\bxi)$ the spectral projector of the operator  $\A(\bxi;\alpha,\mu)$ corresponding  to the segment $[0,d_{0}/3]$, and let  $\Gamma$ be a contour on the complex plane that passes through the midpoint of the interval
 $(d_{0}/3,d_{0})$ and encloses the segment  $[0,d_{0}/3]$ equidistantly.
The length of the contour $\Gamma$ can be easily calculated:
\begin{equation*}
l_\Gamma = \frac{d_0 (2\pi+2)}{3}.
\end{equation*}
Due to Proposition \ref{prop2.1} and by the Riesz formula, we have
\begin{align}
\label{e2.6}
F(\bxi)  &= - \frac{1}{2\pi i}\oint_{\Gamma}(\A(\bxi)-\zeta I)^{-1}\, d\zeta,\ \ |\bxi|\le\delta_{0}(\alpha,\mu),
\\
\label{e2.6a}
\A(\bxi) F(\bxi) &= - \frac{1}{2\pi i}\oint_{\Gamma}(\A(\bxi)-\zeta I)^{-1}\zeta\,d\zeta,\ \ |\bxi|\le\delta_{0}(\alpha,\mu);
\end{align}
here integration along the contour is performed counterclockwise.
Our next goal is to construct approximations to the operators $F(\bxi)$ and $\A(\bxi)F(\bxi)$ for  $|\bxi|\le\delta_{0}(\alpha,\mu)$. To this end we use
relations \eqref{e2.6}, \eqref{e2.6a} and a proper version of the resolvent identity for the difference of the resolvents
of  the operators  $\A(\bxi)$ and $\A(\mathbf{0})$.
The usual resolvent identity is inapplicable because the difference $\A(\bxi) - \A(\mathbf{0})$ need not be well-defined.
Instead, we use the resolvent identity for operators generated by closed non-negative quadratic forms with a common domain,
see \cite[Ch. 1, \S 2]{BSu1}.
Letting
\begin{align*}
R(\bxi,\zeta)&:=(\A(\bxi)-\zeta I)^{-1},\ \
|\bxi|\le\delta_{0}(\alpha,\mu),\ \ \zeta\in\Gamma;
\\
R_{0}(\zeta)&:=R(\mathbf{0},\zeta),\ \ \zeta\in\Gamma,
\end{align*}
we  notice that, by Proposition \ref{prop2.1}, both resolvents $R(\bxi,\zeta)$ and $R_{0}(\zeta)$ are well defined on the contour $\Gamma$ and admit the estimates
\begin{equation}\label{e2.32}
\| R(\bxi,\zeta)\|_{L_2(\Omega)\to L_2(\Omega)}\le 3d_{0}^{-1},\ \
\| R_0 (\zeta)\|_{L_2(\Omega)\to L_2(\Omega)} \le 3d_{0}^{-1}, \ \ |\bxi|\le\delta_{0}(\alpha,\mu),\quad \zeta \in \Gamma.
\end{equation}
Denote by $\mathfrak{D}$ the Hilbert space $\Dom a(\mathbf{0}) = \wt{H}^\gamma(\Omega)$ equipped with the inner product
\begin{equation*}
\label{inner}
(u,v)_{\mathfrak D} :=   a(\mathbf{0})[u,v] + \mu_+ (u,v)_{L_2(\Omega)},\quad u,v \in \wt{H}^\gamma(\Omega).
\end{equation*}
As a consequence of this definition we have
\begin{equation}
\label{L2.6_1}
\|u\|_{L_2(\Omega)} \le \mu_+^{-1/2}\| u \|_{\mathfrak D},\quad u \in \mathfrak{D}.
\end{equation}
The form $a(\bxi) - a(\mathbf{0})$ is continuous in  $\mathfrak D$ and thus generates a bounded self-adjoint operator ${\mathbb T}(\bxi)$ in  $\mathfrak D$. Therefore,
\begin{align}
\label{L2.6_2}
a(\bxi)[u,v] - a(\mathbf{0})[u,v] = ( {\mathbb T}(\bxi) u, v)_{\mathfrak D}, \quad u,v \in {\mathfrak D},
\\
\nonumber
\|{\mathbb T}(\bxi)\|_{{\mathfrak D} \to {\mathfrak D}} = \sup_{0 \ne u \in {\mathfrak D}}
\frac{| a(\bxi)[u,u] - a(\mathbf{0})[u,u] | }{\|u\|^2_{\mathfrak D}}.
\end{align}
By virtue of Lemma \ref{lem2.4} this yields
\begin{equation}
\label{L2.6_4}
\|{\mathbb T}(\bxi)\|_{{\mathfrak D} \to {\mathfrak D}} \le \check{c}(d,\alpha) |\bxi |, \quad \bxi \in \wt{\Omega}.
\end{equation}
Due to \eqref{L2.6_1} we also have
\begin{equation}
\label{L2.6_5}
\|{\mathbb T}(\bxi)\|_{{\mathfrak D} \to L_2(\Omega)} \le \mu_+^{-1/2} \check{c}(d,\alpha) |\bxi|,
\quad \bxi \in \wt{\Omega}.
\end{equation}
As was shown in  \cite[Ch. 1, \S 2]{BSu1} the following resolvent identity holds:
\begin{equation}
\label{L2.6_6}
R(\bxi,\zeta) - R_0(\zeta) = - \Upsilon(\zeta) {\mathbb T}(\bxi) R(\bxi,\zeta),
\end{equation}
where
\begin{equation}
\label{L2.6_7}
\Upsilon(\zeta) := I + (\zeta + \mu_+) R_0(\zeta).
\end{equation}
Taking into account  \eqref{e2.32}, \eqref{L2.6_7} and the fact that  $|\zeta| \le 2d_0/3$ for $\zeta \in \Gamma$, we conclude that
\begin{equation}
\label{L2.6_11}
\|  \Upsilon(\zeta) \|_{L_2(\Omega) \to L_2(\Omega)} \le 1 + |\zeta + \mu_+| \|R_0(\zeta)\|_{L_2(\Omega) \to L_2(\Omega)} \le 3 + 3 \mu_+ d_0^{-1}, \quad \zeta \in \Gamma.
\end{equation}
Later on we use the estimates proved in  \cite[(3.40), (3.46)]{JPSS24}. They read
\begin{align}
\label{R0_leq}
\| R_0(\zeta)\|_{L_2(\Omega) \to \mathfrak{D}} &\le  \mu_+^{-1/2} (3 + 3 \mu_+ d_0^{-1}),
\quad \zeta \in \Gamma,
\\
\label{R_leq}
\| R(\bxi,\zeta)\|_{L_2(\Omega) \to \mathfrak{D}} &\le \beta_0(d,\alpha) \mu_+^{-1/2} (3 + 3 \mu_+ d_0^{-1}),
\quad \zeta \in \Gamma, \quad  |\bxi| \le \delta_0(\alpha,\mu);
\end{align}
here
 \begin{equation*}
 \beta_0^2(d,\alpha) := \max \{2, 1+ 2 c_0(d,\alpha) \pi^\alpha d^{\alpha/2} \}.
\end{equation*}
Making use of the resolvent identity in \eqref{L2.6_6} and  the above estimates for the operators
$R(\bxi,\zeta)$, $R_0(\zeta)$, $\Upsilon(\zeta)$, ${\mathbb T}(\bxi)$, in the same way as in
\cite[(3.41), (3.47)]{JPSS24}  we obtain

\begin{lemma}[\cite{JPSS24}]
\label{lem3.1a}
For $|\bxi| \le \delta_0(\alpha,\mu)$ and $\zeta \in \Gamma$ the following inequalities hold:
\begin{align*}
\| R(\bxi,\zeta) - R_0(\zeta) \|_{L_2(\Omega) \to L_2(\Omega)} & \le C_1 |\bxi|,
\\
\| R(\bxi,\zeta) - R_0(\zeta) +   \Upsilon(\zeta)   {\mathbb T}(\bxi) R_0(\zeta) \|_{L_2(\Omega) \to L_2(\Omega)}
& \le C_2 |\bxi|^2,
\end{align*}
where the constants $C_1$, $C_2$ are given by the expressions
\begin{equation*}
\begin{aligned}
C_1 &= \check{c}(d,\alpha) \beta_0(d,\alpha)\mu_+^{-1}   (3 + 3 \mu_+ d_0^{-1})^2,
\\
C_2 &= \check{c}(d,\alpha)^2 \beta_0(d,\alpha)\mu_+^{-1}  (3 + 3 \mu_+ d_0^{-1})^2
\left( 1 + ( \mu_+ + 2d_0/3)\mu_+^{-1}(3 + 3 \mu_+ d_0^{-1})\right).
\end{aligned}
\end{equation*}
\end{lemma}
We need more precise approximation of the resolvent $R(\bxi,\zeta)$ for small $|\bxi |$.
\begin{lemma}
\label{lem3.2}
For  $|\bxi| \le \delta_0(\alpha,\mu)$ and $\zeta \in \Gamma$ the following representation holds:
\begin{equation}
\label{3.25}
R(\bxi,\zeta) = R_0(\zeta) - \Upsilon(\zeta)   {\mathbb T}(\bxi) R_0(\zeta)
+  \Upsilon(\zeta)   {\mathbb T}(\bxi) \Upsilon(\zeta)   {\mathbb T}(\bxi) R_0(\zeta) + Z(\bxi,\zeta),
\end{equation}
where the remainder $Z(\bxi,\zeta)$ is subject to the upper bound
\begin{equation}
\label{3.25a}
\| Z(\bxi,\zeta) \|_{L_2(\Omega) \to L_2(\Omega)} \le C_3 |\bxi|^3, \quad  |\bxi| \le \delta_0(\alpha,\mu),
\quad \zeta \in \Gamma.
\end{equation}
\end{lemma}

\begin{proof}[Proof]
Iterating the identity in \eqref{L2.6_6} two times we obtain representation \eqref{3.25} with the residual term
$$
Z(\bxi,\zeta) = -  \Upsilon(\zeta)   {\mathbb T}(\bxi) \Upsilon(\zeta)   {\mathbb T}(\bxi)  \Upsilon(\zeta)   {\mathbb T}(\bxi) R(\bxi,\zeta).
$$
Due to \eqref{L2.6_7}, this term can be rearranged as follows:
\begin{equation*}
\begin{aligned}
Z(\bxi,\zeta) = &- \Upsilon(\zeta) {\mathbb T}(\bxi)^3  R(\bxi,\zeta)
- (\zeta+\mu_+)  \Upsilon(\zeta) \left({\mathbb T}(\bxi) R_0(\zeta) {\mathbb T}(\bxi)^2 + {\mathbb T}(\bxi)^2 R_0(\zeta)  {\mathbb T}(\bxi) \right) R(\bxi,\zeta)
\\
&- (\zeta+\mu_+)^2  \Upsilon(\zeta) {\mathbb T}(\bxi) R_0(\zeta)  {\mathbb T}(\bxi) R_0(\zeta)
{\mathbb T}(\bxi) R(\bxi,\zeta).
\end{aligned}
\end{equation*}
Denoting the terms on the right-hand side by  $Z_l(\bxi,\zeta)$, $l=1,2,3$, and taking into account
\eqref{L2.6_4}, \eqref{L2.6_5}, \eqref{L2.6_11}--\eqref{R_leq} as well as the inequality
$|\zeta| \le 2d_0/3$ which is valid for all $\zeta \in \Gamma$, we conclude that for all $|\bxi| \le \delta_0(\alpha,\mu)$,
$$
\begin{aligned}
\| Z_1(\bxi,\zeta)\|_{L_2(\Omega) \to L_2(\Omega)}& \le \|  \Upsilon(\zeta) \|_{L_2(\Omega) \to L_2(\Omega)}
\| {\mathbb T}(\bxi)\|_{\mathfrak{D} \to L_2(\Omega)}
\| {\mathbb T}(\bxi)\|^2_{\mathfrak{D} \to \mathfrak{D}} \|R(\bxi,\zeta)\|_{L_2(\Omega) \to \mathfrak{D}}
\\
&\le C_{3}^{(1)}|\bxi|^3,\quad  C_{3}^{(1)} =  \check{c}(d,\alpha)^3 \mu_+^{-1}  \beta_0(d,\alpha)(3 + 3 \mu_+ d_0^{-1})^2.
\end{aligned}
$$
$$
\begin{aligned}
\| Z_2(\bxi,\zeta)\|_{L_2(\Omega) \to L_2(\Omega)}& \le 2 |\zeta+\mu_+| \|  \Upsilon(\zeta) \|_{L_2(\Omega) \to L_2(\Omega)}
\| {\mathbb T}(\bxi)\|^2_{\mathfrak{D} \to L_2(\Omega)}
\| {\mathbb T}(\bxi)\|_{\mathfrak{D} \to \mathfrak{D}}
\\
&\times \|R_0(\zeta)\|_{L_2(\Omega) \to \mathfrak{D}}\|R(\bxi,\zeta)\|_{L_2(\Omega) \to \mathfrak{D}}
\\
&\le C_{3}^{(2)} |\bxi|^3,\quad  C_{3}^{(2)} =  2 \check{c}(d,\alpha)^3 \mu_+^{-2}  \beta_0(d,\alpha)(3 + 3 \mu_+ d_0^{-1})^3 (\mu_+ + 2d_0/3),
\end{aligned}
$$
$$
\begin{aligned}
\| Z_3(\bxi,\zeta)\|_{L_2(\Omega) \to L_2(\Omega)}& \le  |\zeta+\mu_+|^2 \|  \Upsilon(\zeta) \|_{L_2(\Omega) \to L_2(\Omega)}
\| {\mathbb T}(\bxi)\|^3_{\mathfrak{D} \to L_2(\Omega)}
\\
&\times
\|R_0(\zeta)\|^2_{L_2(\Omega) \to \mathfrak{D}}\|R(\bxi,\zeta)\|_{L_2(\Omega) \to \mathfrak{D}}
\\
&\le C_{3}^{(3)}|\bxi|^3,\quad  C_{3}^{(3)} =   \check{c}(d,\alpha)^3 \mu_+^{-3}  \beta_0(d,\alpha)(3 + 3 \mu_+ d_0^{-1})^4 (\mu_+ + 2d_0/3)^2.
\end{aligned}
$$
This yields the desired estimate \eqref{3.25a} with the constant $C_3 = C_{3}^{(1)} + C_{3}^{(2)} + C_{3}^{(3)}$.
\end{proof}

Next, according to Lemma  \ref{lem2.4_new}, the operator  ${\mathbb T}(\bxi)$ can be represented in the form
\begin{equation}
\label{3.28}
{\mathbb T}(\bxi) = \sum_{j=1}^d \xi_j {\mathbb T}_j + \wt{\mathbb T}(\bxi),
\end{equation}
where ${\mathbb T}_j$ and  $\wt{\mathbb T}(\bxi)$ are  self-adjoint operators in $\mathfrak D$ corresponding to the forms
$a^{(j)}[u,u]$ and $\wt{a}(\bxi)[u,u]$, respectively.
By  \eqref{aj_le} and \eqref{atilde} we have
\vskip -2mm
\begin{align}
\label{3.29}
\| {\mathbb T}_j \|_{\mathfrak{D} \to \mathfrak{D}} & \le c_3(d,\alpha), \quad j=1,\dots,d,
\\
\label{3.30}
\| \wt{\mathbb T}(\bxi) \|_{\mathfrak{D} \to \mathfrak{D}} & \le \wt{c}_3(d,\alpha) |\bxi|^\alpha, \quad  \bxi \in \wt{\Omega}.
\end{align}
Combining these inequalities with   \eqref{L2.6_1} yields
\begin{align}
\label{3.29a}
\| {\mathbb T}_j \|_{\mathfrak{D} \to L_2(\Omega)} & \le \mu_+^{-1/2}c_3(d,\alpha), \quad j=1,\dots,d,
\\
\label{3.30a}
\| \wt{\mathbb T}(\bxi) \|_{\mathfrak{D} \to L_2(\Omega)} & \le\mu_+^{-1/2} \wt{c}_3(d,\alpha) |\bxi|^\alpha, \quad  \bxi \in \wt{\Omega}.
\end{align}
Now from Lemma \ref{lem3.2} and representation  \eqref{3.28} we deduce
\begin{lemma}
For $|\bxi| \le \delta_0(\alpha,\mu)$ and $\zeta \in \Gamma$ the resolvent $R(\bxi,\zeta)$ admits the
representation
\begin{equation}
\label{3.31}
R(\bxi,\zeta) = R_0(\zeta) - \Upsilon(\zeta)   {\mathbb T}(\bxi) R_0(\zeta)
+ \sum_{j,k=1}^d \xi_j \xi_k \Upsilon(\zeta)   {\mathbb T}_j \Upsilon(\zeta)   {\mathbb T}_k R_0(\zeta) +
Z(\bxi,\zeta) + \wt{Z}(\bxi,\zeta),
\end{equation}
where the operator $Z(\bxi,\zeta)$ obeys the upper bound  \eqref{3.25a}, and the operator $\wt{Z}(\bxi,\zeta)$
satisfies the estimate
\begin{equation}
\label{3.32}
\| \wt{Z}(\bxi,\zeta) \|_{L_2(\Omega) \to L_2(\Omega)} \le C_{4} |\bxi|^{1+\alpha}, \quad  |\bxi| \le \delta_0(\alpha,\mu),
\quad \zeta \in \Gamma.
\end{equation}
\end{lemma}

\begin{proof}[Proof]
Substituting  \eqref{3.28} into the third term on the right-hand side of \eqref{3.25} and considering \eqref{L2.6_7},
we obtain the representation in \eqref{3.31} with
$$
\begin{aligned}
\wt{Z}(\bxi,\zeta) &= \sum_{j=1}^d \xi_j \Upsilon(\zeta)   {\mathbb T}_j \Upsilon(\zeta) \wt{\mathbb T}(\bxi) R_0(\zeta)
+ \Upsilon(\zeta) \wt{\mathbb T}(\bxi) \Upsilon(\zeta) {\mathbb T}(\bxi)R_0(\zeta)
\\
&= \sum_{j=1}^d \xi_j \Upsilon(\zeta)   {\mathbb T}_j \wt{\mathbb T}(\bxi) R_0(\zeta) +
(\zeta + \mu_+) \sum_{j=1}^d \xi_j \Upsilon(\zeta)   {\mathbb T}_j R_0(\zeta) \wt{\mathbb T}(\bxi) R_0(\zeta)
\\
&+ \Upsilon(\zeta) \wt{\mathbb T}(\bxi) {\mathbb T}(\bxi) R_0(\zeta) +
(\zeta + \mu_+) \Upsilon(\zeta) \wt{\mathbb T}(\bxi) R_0(\zeta)  {\mathbb T}(\bxi) R_0(\zeta).
\end{aligned}
$$
Denote the terms on the right-hand side by $\wt{Z}_l(\bxi,\zeta)$, $l=1,2,3,4$. Then, due to
 \eqref{L2.6_4}, \eqref{L2.6_5}, \eqref{L2.6_11}, \eqref{R0_leq} and \eqref{3.30}--\eqref{3.30a}, for $|\bxi| \le \delta_0(\alpha,\mu)$ and $\zeta \in \Gamma$ one has
\begin{align*}
\begin{split}
&\| \wt{Z}_1(\bxi,\zeta)\|_{L_2(\Omega) \to L_2(\Omega)} \le \sum_{j=1}^d |\xi_j| \|  \Upsilon(\zeta) \|_{L_2(\Omega) \to L_2(\Omega)}
\| {\mathbb T}_j\|_{\mathfrak{D} \to L_2(\Omega)}
\| \wt{\mathbb T}(\bxi)\|_{\mathfrak{D} \to \mathfrak{D}} \|R_0(\zeta)\|_{L_2(\Omega) \to \mathfrak{D}}
\\
&\qquad \le C_{4}^{(1)} |\bxi|^{1+\alpha},\quad  C_{4}^{(1)} =  \sqrt{d}\,{c}_3(d,\alpha) \wt{c}_3(d,\alpha) \mu_+^{-1}  (3 + 3 \mu_+ d_0^{-1})^2,
\end{split}
\\
\begin{split}
&\| \wt{Z}_2(\bxi,\zeta)\|_{L_2(\Omega) \to L_2(\Omega)}
\\
&\qquad \le  |\zeta + \mu_+| \sum_{j=1}^d |\xi_j| \|  \Upsilon(\zeta) \|_{L_2(\Omega) \to L_2(\Omega)}
\| {\mathbb T}_j\|_{\mathfrak{D} \to L_2(\Omega)}
 \| \wt{\mathbb T}(\bxi)\|_{\mathfrak{D} \to L_2(\Omega)} \|R_0(\zeta)\|^2_{L_2(\Omega) \to \mathfrak{D}}
\\
&\qquad \le C_{4}^{(2)} |\bxi|^{1+\alpha},\quad  C_{4}^{(2)} = \sqrt{d}\,{c}_3(d,\alpha) \wt{c}_3(d,\alpha)
\mu_+^{-2}  (3 + 3 \mu_+ d_0^{-1})^3 (\mu_+ + 2d_0/3),
\end{split}
\\
\begin{split}
&\| \wt{Z}_3(\bxi,\zeta)\|_{L_2(\Omega) \to L_2(\Omega)} \le \|  \Upsilon(\zeta) \|_{L_2(\Omega) \to L_2(\Omega)}
\| \wt{\mathbb T}(\bxi)\|_{\mathfrak{D} \to L_2(\Omega)}
\| {\mathbb T}(\bxi)\|_{\mathfrak{D} \to \mathfrak{D}} \|R_0(\zeta)\|_{L_2(\Omega) \to \mathfrak{D}}
\\
&\qquad \le C_{4}^{(3)} |\bxi|^{1+\alpha},\quad  C_{4}^{(3)}  =  \check{c}(d,\alpha) \wt{c}_3(d,\alpha)
\mu_+^{-1}  (3 + 3 \mu_+ d_0^{-1})^2,
\end{split}
\\
\begin{split}
&\| \wt{Z}_4(\bxi,\zeta)\|_{L_2(\Omega) \to L_2(\Omega)} \le |\zeta + \mu_+|
 \|  \Upsilon(\zeta) \|_{L_2(\Omega) \to L_2(\Omega)}
\| \wt{\mathbb T}(\bxi)\|_{\mathfrak{D} \to L_2(\Omega)}
\| {\mathbb T}(\bxi)\|_{\mathfrak{D} \to L_2(\Omega)}
\\
&\qquad  \times\!\|R_0(\zeta)\|^2_{L_2(\Omega) \to \mathfrak{D}}\le\! C_{4}^{(4)}\! |\bxi|^{1+\alpha},\quad\!  C_{4}^{(4)}\!\!  =  \check{c}(d,\alpha) \wt{c}_3(d,\alpha) \mu_+^{-2}  (3 + 3 \mu_+ d_0^{-1})^3 (\mu_+\! + 2d_0/3).
\end{split}
\end{align*}
This yields the required estimate  \eqref{3.32} with a constant
$C_{4} = C_{4}^{(1)} +C_{4}^{(2)} +C_{4}^{(3)} +C_{4}^{(4)}$.
\end{proof}

\subsection{Threshold approximations}
We use the notation $\mathfrak{N}$ for the kernel $\operatorname{Ker}\A(\mathbf{0};\alpha,\mu)=\mathcal{L}\{\mathbf{1}_{\Omega}\}$,
and denote by $P$ the orthogonal projection on  $\mathfrak{N}$:  $P = (\cdot, \1_\Omega)\1_\Omega$.
The following statement is a consequence of the Riesz formulae in  \eqref{e2.6}, \eqref{e2.6a} and Lemma \ref{lem3.1a},
see also \cite[Propositions 3.4, 3.5]{JPSS24}:

\begin{proposition}[\cite{JPSS24}]
\label{prop2.3a}
Let conditions  \eqref{e1.1}, \eqref{e1.2} be fulfilled, and assume that $1 <  \alpha < 2$.
Then the following estimates hold:
\begin{align}
\label{F-P_a}
\| F(\bxi) - P \|_{L_2(\Omega) \to L_2(\Omega)} \le C_5 |\bxi|, \quad  |\bxi| \le \delta_0(\alpha,\mu),
\\
\nonumber
\| \A(\bxi) F(\bxi) - \mu_+ P \mathbb{T}(\bxi) P \|_{L_2(\Omega) \to L_2(\Omega)} \le C_6 |\bxi|^2, \quad  |\bxi| \le \delta_0(\alpha,\mu),
\end{align}
where $C_5 =  (2\pi)^{-1} l_\Gamma C_1$ and $C_6 = (3 \pi)^{-1} l_\Gamma d_0 C_2$. These constants depend
only on  $d$, $\alpha$, $\mu_-$ and $\mu_+$.
\end{proposition}

We also need more accurate approximation of the operator  $\A(\bxi) F(\bxi)$.

\begin{proposition}
\label{prop2.5_2}
Let conditions  \eqref{e1.1}, \eqref{e1.2} be fulfilled, and assume that $1 <  \alpha < 2$. Then, for $|\bxi| \le \delta_0(\alpha,\mu)$ the operator $\A(\bxi) F(\bxi)$ admits the representation
\begin{equation}
\label{F= P+ F1_2}
\A(\bxi) F(\bxi) =   \mu_+ P  \mathbb{T}(\bxi)  P  + \sum_{j,k =1}^d \xi_j \xi_k G_{jk}
+ {\Phi}(\bxi),
 \end{equation}
 where
 \begin{equation}
\label{Gjk==}
 G_{jk} =-\mu_+ P \mathbb{T}_j  P^\perp \mathbb{T}_k P - \mu_+^2 P \mathbb{T}_j  R_0^\perp(0) \mathbb{T}_k P,
 \quad j,k =1,\dots,d,
 \end{equation}
 $R_0^\perp(0) = P^\perp \A(\mathbf{0})^{-1}P^\perp$, and $ \A(\mathbf{0})^{-1}$ stands for the operator which is inverse to $ \A(\mathbf{0})\vert_{{\mathfrak N}^\perp}$ and well-defined as a bounded operator from ${\mathfrak N}^\perp$ to ${\mathfrak N}^\perp$.
The following estimates hold:
\begin{align}
\label{Gjk_le}
\bigl\| G_{jk} \bigr\|_{L_2(\Omega) \to L_2(\Omega)} &\le C_7,\quad j,k=1,\dots,d,
\\
\label{F-P-O(xi)_2}
\bigl\| {\Phi}(\bxi) \bigr\|_{L_2(\Omega) \to L_2(\Omega)} &\le C_8 |\bxi|^{1+\alpha}, \quad  |\bxi| \le \delta_0(\alpha,\mu).
\end{align}
The constants $C_7, C_8$ depend only on  $d$, $\alpha$, $\mu_-$, $\mu_+$.
\end{proposition}

\begin{proof}[Proof]
By the Riesz formula \eqref{e2.6a} considering \eqref{3.31} we obtain
\begin{equation}
\label{AF==}
\A(\bxi)F(\bxi) = G_0  + {G}_1(\bxi) + \sum_{j,k=1}^d \xi_j \xi_k G_{jk} + {\Phi}(\bxi), \quad |\bxi|\le\delta_{0}(\alpha,\mu),
\end{equation}
with
\begin{align}
\nonumber
{G}_0 &=  - \frac{1}{2\pi i}\oint_{\Gamma}  R_0(\zeta) \zeta  \, d\zeta,
\\
\label{G1==}
{G}_1(\bxi) &=  \frac{1}{2\pi i}\oint_{\Gamma} \Upsilon(\zeta) {\mathbb T}(\bxi) R_0(\zeta) \zeta  \, d\zeta,
\\
\label{Gjk=}
{G}_{jk} &=  - \frac{1}{2\pi i}\oint_{\Gamma} \Upsilon(\zeta) {\mathbb T}_j \Upsilon(\zeta) {\mathbb T}_k
R_0(\zeta) \zeta  \, d\zeta,
\\
\label{Phi==}
{\Phi}(\bxi) &= - \frac{1}{2\pi i}\oint_{\Gamma} ({Z}(\bxi,\zeta) + \wt{Z}(\bxi,\zeta))  \zeta  \, d\zeta.
\end{align}
Letting $\bxi = {\mathbf 0}$ in \eqref{e2.6a} yields
\begin{equation}
\label{3.45aaa}
G_0 = \A(\mathbf{0})P =0.
\end{equation}
Considering \eqref{L2.6_7}--\eqref{R0_leq}, \eqref{3.29} and \eqref{3.29a} it is straightforward to estimate the operator
 \eqref{Gjk=}:
$$
\| G_{jk}\|_{L_2(\Omega) \to L_2(\Omega)} \le
 (3 + 3 \mu_+ d_0^{-1})^2 c_3^2 \mu_+^{-1} \left( 1 +  (\mu_+ + 2d_0/3) (3 + 3 \mu_+ d_0^{-1}) \mu_+^{-1} \right)
 \frac{l_\Gamma d_0}{3\pi} =: C_7.
$$
This gives inequlity \eqref{Gjk_le}.
Due to \eqref{3.25a} and \eqref{3.32}, for $|\bxi| \le \delta_0(\alpha,\mu)$, the operator in   \eqref{Phi==} can be
estimated as follows:
\begin{equation*}
\label{L3.5_6}
\| {\Phi}(\bxi) \|_{L_2(\Omega) \to L_2(\Omega)} \le
\left(C_3 |\bxi|^3 + C_4 |\bxi|^{1+\alpha} \right) \frac{l_\Gamma d_0}{3\pi} \le
\frac{2 (\pi+1) d_0^2}{9\pi} \left( C_3 (\pi \sqrt{d})^{2-\alpha} + C_4 \right) |\bxi|^{1+\alpha},
\end{equation*}
this proves \eqref{F-P-O(xi)_2} with the constant $C_8 = \frac{2 (\pi+1) d_0^2}{9\pi} \left( C_3 (\pi \sqrt{d})^{2-\alpha} + C_4 \right)$.

To calculate the integrals in   \eqref{G1==} and \eqref{Gjk=} we decompose the resolvent of the operator
$\mathbb{A}(\mathbf{0})$:
\begin{equation}
\label{4.10}
R_{0}(\zeta)=R_{0}(\zeta)P+R_{0}(\zeta)P^{\bot}=-\frac{1}{\zeta}P+R_{0}(\zeta)P^{\bot},\ \ \zeta\in\Gamma,
\end{equation}
and substitute the expression on right-hand side of \eqref{4.10} for $R_{0}(\zeta)$ in the contour integral  \eqref{G1==}.
Since the operator-function  $R_{0}^{\bot}(\zeta):=R_{0}(\zeta)P^{\bot}$ is holomorphic  inside the contour $\Gamma$,
this yields
\begin{equation}
\label{3.45aa}
\begin{aligned}
{G}_{1}(\bxi) &=  \frac{1}{2\pi i}\oint_{\Gamma} \Bigl( I + (\zeta + \mu_+) \Bigl(-\frac{1}{\zeta}P+R_{0}^{\bot}(\zeta)\Bigr) \Bigr) \mathbb{T}(\bxi)
\Bigl(-\frac{1}{\zeta}P+R_{0}^{\bot}(\zeta)\Bigr) \zeta \,d\zeta
\\
&=  \frac{1}{2\pi i}\oint_{\Gamma} \frac{1}{\zeta} \mu_+ P \mathbb{T}(\bxi)  P \,d\zeta =
 \mu_+ P \mathbb{T}(\bxi)  P;
\end{aligned}
  \end{equation}
compare with  \cite[Proposition 3.5]{JPSS24}.

The integral in \eqref{Gjk=} can be calculated in a similar way:
\begin{equation}
\label{3.45a}
\begin{aligned}
{G}_{jk} &= - \frac{1}{2\pi i}\oint_{\Gamma} \Bigl( I + (\zeta + \mu_+) \Bigl(-\frac{1}{\zeta}P+R_{0}^{\bot}(\zeta)\Bigr) \Bigr) \mathbb{T}_j  \Bigl( I + (\zeta + \mu_+) \Bigl(-\frac{1}{\zeta}P+R_{0}^{\bot}(\zeta)\Bigr) \Bigr)
\\
& \qquad \qquad  \times \mathbb{T}_k \Bigl(-\frac{1}{\zeta}P+R_{0}^{\bot}(\zeta)\Bigr) \zeta \,d\zeta.
\end{aligned}
  \end{equation}
  Notice that
\begin{equation}
\label{3.47}
P \mathbb{T}_j P = 0,\quad j=1,\dots,d.
  \end{equation}
Indeed, according to \eqref{2.36_new},  we have  $a^{(j)}[ \1_\Omega, \1_\Omega] = 0$,  and thus
$$
P \mathbb{T}_j P = ( \mathbb{T}_j \1_\Omega, \1_\Omega)_{L_2(\Omega)} P =
a^{(j)}[ \1_\Omega, \1_\Omega] P = 0.
$$
By \eqref{3.45a} and \eqref{3.47} we obtain
\begin{equation}
\label{3.48}
\begin{aligned}
{G}_{jk} &= - \frac{1}{2\pi i}\oint_{\Gamma}
 \frac{1}{\zeta} \bigl(\mu_+ P \mathbb{T}_j  P^\perp \mathbb{T}_k P +
 \mu_+^2 P \mathbb{T}_j  R_0^\perp(\zeta) \mathbb{T}_k P \bigr) \,d\zeta
\\
&= - \mu_+ P \mathbb{T}_j  P^\perp \mathbb{T}_k P - \mu_+^2 P \mathbb{T}_j  R_0^\perp(0) \mathbb{T}_k P.
\end{aligned}
  \end{equation}
Finally, the representation \eqref{F= P+ F1_2}, \eqref{Gjk==} follows from \eqref{AF==}, \eqref{3.45aaa}, \eqref{3.45aa} and \eqref{3.48}.
\end{proof}

\begin{remark}
\label{rem3.6}
Considering the definition of the operator ${\mathbb T}(\bxi)$ in {\rm \eqref{L2.6_2}} and the fact that  $P$ is the orthogonal projection on $\Ker \A(\mathbf{0})$, we conclude that
$a(\mathbf{0}) [{\mathbb T}(\bxi)P u, Pv] =0$. Therefore,
\begin{equation}
\label{remm1}
a(\bxi)[Pu,Pv] - a(\mathbf{0})[Pu,Pv] = \mu_+ ( {\mathbb T}(\bxi) Pu, Pv)_{L_2(\Omega)}, \quad u,v \in L_2(\Omega).
  \end{equation}
  This implies, in particular, that the operator $\mu_+ P {\mathbb T}(\bxi) P$ is bounded in $L_2(\Omega)$
  and generated by the form on the left-hand side of \eqref{remm1}.
 \end{remark}

\begin{proposition}
\label{prop3.8}
For the operators $G_{jk},$ $j,k=1,\dots,d,$ defined in  \eqref{Gjk==} the following representation is valid:
\begin{align}
\label{pr3.8_1}
G_{jk} &= g_{jk} P, \quad j,k=1,\dots,d,
  \\
\label{pr3.8_2}
g_{jk} &=  \frac{1}{2} \int_{\R^d} d\y \int_{\Omega} d\x\, \mu(\x,\y)  \frac{(x_j-y_j)}{|\x-\y|^{d+\alpha}}
 \left(  v_k(\x)   -  v_k(\y)\right),\quad j,k=1,\dots,d;
  \end{align}
here the functions $v_k \in \wt{H}^\gamma(\Omega)$ are centered,  $\int_\Omega v_k(\x)\,d\x =0$, and satisfy the integral identity
  \begin{equation}
\label{vk_problem}
\int_{\R^d} d\y \int_\Omega d\x \frac{\mu(\x,\y)}{|\x - \y|^{d+\alpha}}
\bigl( v_k(\x) - v_k(\y) +x_k - y_k \bigr) \bigl( \overline{\eta(\x) - \eta(\y)} \bigr)= 0, \quad \eta \in \wt{H}^\gamma(\Omega).
  \end{equation}
\end{proposition}

\begin{proof}[Proof]
According to \eqref{Gjk==} we have
\begin{align}
\label{pr3.8_3}
G_{jk} =& - \mu_+ P \mathbb{T}_j  \bigl( P^\perp + \mu_+ R_0^\perp(0)\bigr) \mathbb{T}_k P = g_{jk} P,
 \\
 \label{pr3.8_4}
 g_{jk} :=& - \mu_+ \bigl(\mathbb{T}_j  \bigl( P^\perp + \mu_+ R_0^\perp(0)\bigr) \mathbb{T}_k \1_\Omega, \1_\Omega
 \bigr)_{L_2(\Omega)}=  i \mu_+ (\mathbb{T}_j  v_k, \1_\Omega  )_{L_2(\Omega)};
 \end{align}
here we used the notation
 \begin{equation}
 \label{pr3.8_5}
  v_k := i \bigl( P^\perp + \mu_+ R_0^\perp(0)\bigr) \mathbb{T}_k \1_\Omega.
 \end{equation}

Observe that the sesquilinear form   $a^{(j)} [v,w]$ corresponding to the quadratic form
 \eqref{2.36_new} is given by
 \begin{equation}
\label{pr3.8_6}
a^{(j)} [v,w] := - \frac{1}{2} \int_{\R^d} d\y \int_{\Omega} d\x\, \mu(\x,\y)  \frac{(x_j-y_j)}{|\x-\y|^{d+\alpha}}
 \left( i v(\x) \overline{w(\y)}  - i v(\y) \overline{w(\x)}\right) ,\quad v,w \in \wt{H}^\gamma(\Omega).
\end{equation}
By the definition of the operator $\mathbb{T}_j$,
$$
a(0)[\mathbb{T}_j v_k, \1_\Omega] + \mu_+ (\mathbb{T}_j v_k, \1_\Omega)_{L_2(\Omega)}
= a^{(j)}[v_k, \1_\Omega].
$$
Since $\A(\mathbf{0})\1_\Omega =0$, the first term on the left-hand side vanishes. Therefore,
$$
 \mu_+ (\mathbb{T}_j v_k, \1_\Omega)_{L_2(\Omega)}
= a^{(j)}[v_k, \1_\Omega].
$$
From this relation and \eqref{pr3.8_4}, \eqref{pr3.8_6} we deduce that
$$
g_{jk} = i a^{(j)}[v_k, \1_\Omega] =
 \frac{1}{2} \int_{\R^d} d\y \int_{\Omega} d\x\, \mu(\x,\y)  \frac{(x_j-y_j)}{|\x-\y|^{d+\alpha}}
 \left(  v_k(\x)   -  v_k(\y)\right),\quad j,k=1,\dots,d.
$$
Then \eqref{pr3.8_1} and \eqref{pr3.8_2} follow from the latter relation and \eqref{pr3.8_3}.

It remains to determine which problem the function \eqref{pr3.8_5} solves.
By its definition,
$v_k \in \wt{H}^\gamma(\Omega)$, $\int_\Omega v_k(\x)\,d\x =0$, and for any test function
 $\eta \in \wt{H}^\gamma(\Omega)$, $\int_\Omega \eta(\x)\,d\x =0$, it holds
$$
\begin{aligned}
a(\mathbf{0})[v_k, \eta] =  i a(\mathbf{0})[P^\perp \mathbb{T}_k \1_\Omega, \eta]
+ i \mu_+ a(\mathbf{0})[R_0^\perp(0) \mathbb{T}_k \1_\Omega, \eta]
\\
=  i a(\mathbf{0})[ \mathbb{T}_k \1_\Omega, \eta] + i \mu_+ \bigl(\mathbb{T}_k \1_\Omega, \eta \bigr)_{L_2(\Omega)}
= i a^{(k)}[\1_\Omega, \eta].
\end{aligned}
$$
Here we used the evident relation $a(\mathbf{0})[P \mathbb{T}_k \1_\Omega, \eta] =0$ and the definition of the operator $\mathbb{T}_k$.
Rewriting the relation $a(\mathbf{0})[v_k, \eta] = i a^{(k)}[\1_\Omega, \eta]$ in the detailed form leads to the integral  identity in
\eqref{vk_problem} for all test functions $\eta \in \wt{H}^\gamma(\Omega)$, $\int_\Omega \eta(\x)\,d\x =0$.
In fact, this identity is satisfied for all $\eta \in \wt{H}^\gamma(\Omega)$, since its validity for $\eta = \1_\Omega$ is
evident.
\end{proof}

\subsection{Study of the operator  $\mu_+ P  \mathbb{T}(\bxi)  P$}

Introducing the notation
\begin{equation}\label{2.31}
\begin{aligned}
\rho(\bxi) = a(\bxi)[{\mathbf 1}_\Omega,{\mathbf 1}_\Omega] - a(\mathbf{0})[{\mathbf 1}_\Omega,{\mathbf 1}_\Omega] =a(\bxi)[{\mathbf 1}_\Omega,{\mathbf 1}_\Omega]=
\intop_{\R^d} d\y \intop_\Omega d\x\, \mu(\x,\y) \frac{ 1 - \cos(\langle \bxi, \x -\y\rangle)}{|\x -\y |^{d+\alpha}}
 \end{aligned}
  \end{equation}
  and considering Remark  \ref{rem3.6} and the relation $P = (\cdot, \mathbf{1}_\Omega)  \mathbf{1}_\Omega$, we obtain
\begin{equation}\label{2.30}
 \mu_+ P  \mathbb{T}(\bxi)  P = \rho(\bxi) P.
  \end{equation}
The Fourier series of the periodic function $\mu(\x,\y)$ reads
\begin{equation}\label{2.32}
\mu(\x,\y) = \sum_{\m, \mathbf{l} \in \Z^d} \wh{\mu}_{\m, \mathbf{l}} e^{2\pi i (\langle \m, \x\rangle + \langle \mathbf{l}, \y\rangle)},
  \end{equation}
  where the coefficients $ \wh{\mu}_{\m, \mathbf{l}} $ are given by the expressions
  $$
  \wh{\mu}_{\m, \mathbf{l}} = \int_\Omega \int_\Omega d\x\, d\y\, \mu(\x,\y) e^{-2\pi i (\langle \m, \x\rangle + \langle \mathbf{l}, \y\rangle)},\quad \m, \mathbf{l} \in \Z^d.
  $$
  The symmetry condition $\mu(\x,\y) = \mu(\y,\x)$ implies that
$ \wh{\mu}_{\m, \mathbf{l}} =  \wh{\mu}_{ \mathbf{l}, \m}$, $\m, \mathbf{l} \in \Z^d$.
Denote
\begin{equation}\label{2.33}
\mu^0  :=  \wh{\mu}_{\mathbf{0}, \mathbf{0}} = \int_\Omega \int_\Omega d\x\, d\y\, \mu(\x,\y).
  \end{equation}
We formulate here the statement that was proved in \cite[Lemmata 3.7 and 3.8]{JPSS24}.
\begin{lemma}[\cite{JPSS24}]
\label{lem2.6}
Let conditions  \eqref{e1.1}, \eqref{e1.2} be satisfied, and assume that $1< \alpha< 2$.
Then the function  $\rho(\bxi)$ defined in \eqref{2.31} admits the following two-term expansion{\rm :}
\begin{align}
\nonumber
\rho(\bxi)& = \mu^0 c_0(d,\alpha) |\bxi|^\alpha +  \rho_*(\bxi),
\\
\label{2.35aa}
\rho_*(\bxi) &= \int_{\R^d}   \mu_*(\z)
 \frac{ 1 - \cos (\langle \bxi, \z\rangle)}{|\z |^{d+\alpha}}\,d\z,
 \end{align}
 where $\mu^0$ is defined in \eqref{2.33},  and $\mu_*$ is an even,  $\Z^d$-periodic, zero average function given by
$\mu_*(\z) := \int_\Omega \mu(\x,\z+\x) \,d\x - \mu^0$.
The function $\mu_*$ satisfies the estimates
\begin{equation*}
 \mu_- - \mu^0 \le \mu_*(\z) \le \mu_+ - \mu^0,
  \end{equation*}
and its Fourier series reads
\begin{equation*}
\mu_*(\z) = \sum_{\m \in \Z^d \setminus \{ \mathbf{0} \} } \wh{\mu}_{\m,-\m}  \cos (2\pi  \langle\m,\z\rangle ).
\end{equation*}
Moreover, the following upper bound holds for all $\bxi \in \wt{\Omega}${\rm :}
\begin{equation*}
|\rho_*(\bxi) | \le C_9 |\bxi|^2.
  \end{equation*}
  Here the constant $C_9$ depends only on  $d$, $\alpha$,  $\mu_+$.
\end{lemma}

Next we are going to extract the quadratic terms of the function $\rho_*(\bxi)$ and estimate the remainder.
\begin{lemma}
\label{lem3.10}
Let conditions \eqref{e1.1}, \eqref{e1.2} be fulfilled, and assume that $1< \alpha< 2$.
Then for all $\bxi \in \wt{\Omega}$ the function $\rho_*(\bxi)$ defined in \eqref{2.35aa} admits the representation
\begin{equation*}
\rho_*(\bxi) = \langle g_* \bxi, \bxi \rangle + \wt{\rho}_*(\bxi);
  \end{equation*}
here the quadratic form  $\langle g_* \bxi, \bxi \rangle$ is given by the expression
\begin{equation*}
\langle g_* \bxi, \bxi \rangle  = \int_{\R^d} \frac{\mu_*(\z) \langle \bxi, \z \rangle^2 }{2|\z|^{d+\alpha}} \,d\z,
  \end{equation*}
and the integral on the right-hand side is understood in the following sense{\rm :}
\begin{equation*}
 \int_{\R^d} \frac{\mu_*(\z) \langle \bxi, \z \rangle^2 }{2|\z|^{d+\alpha}} \,d\z := \sum_{\n \in \Z^d}
  \int_{\Omega + \n} \frac{\mu_*(\z) \langle \bxi, \z \rangle^2 }{2|\z|^{d+\alpha}} \,d\z.
  \end{equation*}
  The sum on the right-hand side converges absolutely. The following estimates hold{\rm :}
\begin{align}
\label{3.65a}
|g_*| &\le c_4(d,\alpha) \mu_+,
\\
\label{3.66}
| \wt{\rho}_*(\bxi)| &\le C_{10} |\bxi|^{1+\alpha}, \quad  \bxi \in \wt{\Omega}.
  \end{align}
  The constant $C_{10}$  depends only on $d$, $\alpha$,  $\mu_+$.
\end{lemma}

\begin{proof}[Proof]
Let the number  $R = R(d,\alpha)$ be a positive root of the equation
$$
\Big(\frac{R}{R - \sqrt{d}}\Big)^{d+\alpha} = 2,
$$
then
$$
R(d,\alpha) = \frac{\sqrt{d} \,2^{\frac{1}{d+\alpha}} }{2^{\frac{1}{d+\alpha}}-1}.
$$
Denote by $\Sigma_0$ the set of indices $\{\n \in \Z^d\,:\, (\Omega+n)\cap B_R(\mathbf{0})\not=\emptyset\}$,
 and by $\Sigma_1$ -- its complement  $\Sigma_1 = \Z^d \setminus \Sigma_0$.
It is straightforward to check that $R >2 \sqrt{2d}$, and
\begin{equation}
\label{2.41}
 \frac{|\z_1|^{d+\alpha}}{|\z_2|^{d+\alpha}} \le 2 \ \ \ \hbox{for all }\  \n \in \Sigma_1, \quad \z_1,\z_2 \in \Omega + \n.
  \end{equation}

It is convenient to represent the function $\rho_*(\bxi)$  as the sum
\begin{equation}
\label{3.67}
{\rho}_*(\bxi) = {\rho}^{(0)}_*(\bxi) + {\rho}^{(1)}_*(\bxi)
  \end{equation}
with
\begin{align}
\label{3.68}
{\rho}^{(0)}_*(\bxi) := \sum_{\n \in \Sigma_0} \int_{\Omega + \n}  \mu_*(\z)
 \frac{ 1 - \cos (\langle \bxi, \z\rangle)}{|\z |^{d+\alpha}} \,d\z,
 \\
 \label{3.69}
{\rho}^{(1)}_*(\bxi) := \sum_{\n \in \Sigma_1} \int_{\Omega + \n}   \mu_*(\z)
 \frac{ 1 - \cos (\langle \bxi, \z\rangle)}{|\z |^{d+\alpha}} \,d\z.
\end{align}

Rewriting \eqref{3.68} as
\begin{equation}
\label{3.70}
{\rho}^{(0)}_*(\bxi) := \sum_{\n \in \Sigma_0} \int_{\Omega + \n}   \mu_*(\z)
 \frac{\langle \bxi, \z\rangle^2 }{2|\z |^{d+\alpha}}\,d\z + \wt{\rho}_*^{(0)}(\bxi),
  \end{equation}
and applying the obvious inequality
$$
\left|1 - \cos \lambda - \frac{\lambda^2}{2}\right| \le  \frac{\lambda^4}{24}, \quad \lambda \in \R,
$$
we conclude that, for all $\bxi\in\wt{\Omega}$, the remainder $\wt{\rho}_*^{(0)}(\bxi)$ admits the upper bound
\begin{equation*}
\bigl| \wt{\rho}^{(0)}_*(\bxi)\bigr| \le |\bxi|^4\! \sum_{\n \in \Sigma_0} \intop_{\Omega + \n} \!\!  |\mu_*(\z)|
 \frac{ d\z}{24|\z |^{d+\alpha -4}} \le |\bxi|^4  \mu_+\! \!\!\!\intop_{|\z| < R+ \sqrt{d}}\!\! \frac{ d\z}{24|\z |^{d+\alpha -4}}=
 |\bxi|^4  \mu_+ \frac{\omega_d (R+ \sqrt{d})^{4-\alpha}}{24(4-\alpha)}.
  \end{equation*}
  Thus
\begin{equation}
\label{3.71}
\bigl| \wt{\rho}^{(0)}_*(\bxi)\bigr| \le c_5(d,\alpha) \mu_+ |\bxi|^4, \quad \bxi \in \wt{\Omega},
\quad c_5(d,\alpha) = \frac{\omega_d (R+ \sqrt{d})^{4-\alpha}}{24(4-\alpha)}.
  \end{equation}

We turn to the function ${\rho}^{(1)}_*(\bxi)$ defined in  \eqref{3.69}. Letting
$$
\varphi_{\bxi}(\z) :=  \frac{ 1 - \cos (\langle \bxi, \z\rangle)}{|\z |^{d+\alpha}}
$$
and taking into account  $\Z^d$-periodicity of the function $\mu_*(\z)$ and the relation
$\int_\Omega \mu_*(\z)\,d\z =0$, we have
\begin{equation}
\label{3.72a}
\begin{aligned}
{\rho}^{(1)}_*(\bxi) &= \sum_{\n \in \Sigma_1} \int_{\Omega + \n}   \mu_*(\z) \varphi_{\bxi}(\z) \,d\z
= \sum_{\n \in \Sigma_1} \int_{\Omega + \n}   \mu_*(\z) \left(\varphi_{\bxi}(\z) - \varphi_{\bxi}(\n) \right) \,d\z
\\
&= \sum_{\n \in \Sigma_1} \int_{\Omega + \n}  d\z\, \mu_*(\z) \sum_{j=1}^d(z_j -n_j)
\int_0^1 \partial_j \varphi_{\bxi}(\n + t (\z-\n))\,dt .
\end{aligned}
 \end{equation}
 Since
$$
 \partial_j \varphi_{\bxi}(\z) =  \frac{\xi_j \sin (\langle \bxi, \z\rangle)}{|\z |^{d+\alpha}}
 - \frac{ (d+\alpha) z_j (1 - \cos (\langle \bxi, \z\rangle))}{|\z |^{d+\alpha +2}},
$$
then
\begin{equation}
\label{3.72}
\begin{aligned}
&{\rho}^{(1)}_*(\bxi) = \sum_{\n \in \Sigma_1} \intop_{\Omega + \n}  d\z\, \mu_*(\z)
\langle \z - \n, \bxi\rangle  \intop_0^1 \frac{\sin (\langle \bxi, \n + t (\z-\n)\rangle)}{|\n + t (\z-\n)|^{d+\alpha}}\,dt
\\
&- (d+\alpha) \sum_{\n \in \Sigma_1} \intop_{\Omega + \n}  d\z\, \mu_*(\z)
\sum_{j=1}^d(z_j -n_j)
\intop_0^1 \frac{(n_j + t (z_j - n_j))  (1 - \cos (\langle \bxi, \n + t (\z-\n)\rangle)}{|\n + t (\z-\n)|^{d+\alpha+2}}
\,dt.
\end{aligned}
  \end{equation}

For an arbitrary $p>1$ let us introduce the functions
$$
F_{s,p}(\lambda)=\frac{(\sin \lambda -  \lambda)(1 + |\operatorname{ln} |\lambda || )^p}{|\lambda|^\alpha }\ \ \hbox{and}\ \
F_{c,p}(\lambda)=\frac{\big(1-\cos \lambda -  \frac{\lambda^2}2\big)(1 + |\operatorname{ln} |\lambda || )^p}{|\lambda|^\alpha };
$$
for $\lambda=0$ we set $F_{s,p}(0)=F_{c,p}(0)=0$. By direct inspection we see that $F_{s,p}, F_{c,p} \in L_\infty(\R)$,
and
 \begin{align}
 \label{3.75}
 | F_{s,p}(\lambda)| \le \max \left\{ \sup_{0< x \le 1} \frac{x^{3-\alpha}}{6} (1 + |\operatorname{ln} x | )^p ,
 \sup_{1 \le  x  < \infty} \frac{(1 + x)}{x^\alpha} (1 + |\operatorname{ln} x | )^p
    \right\} =: \mathfrak{c}_s(p,\alpha), \quad \lambda \in \R,
    \\
 \label{3.76}
 | F_{c,p}(\lambda)| \le \max \left\{ \sup_{0< x \le 1} \frac{x^{3-\alpha}}{24} (1 + |\operatorname{ln} x | )^p ,
 \sup_{1 \le  x  < \infty} \frac{(4 + x^2)}{2 x^{1+\alpha}} (1 + |\operatorname{ln} x | )^p
    \right\} =: \mathfrak{c}_c(p,\alpha), \quad \lambda \in \R.
\end{align}
Moreover, by the definition of $F_{s,p}$ and $F_{c,p}$,
 \begin{align}
 \label{3.73}
\sin \lambda =  \lambda + \frac{|\lambda|^\alpha F_{s,p}(\lambda)}{(1 + |\operatorname{ln} |\lambda || )^p}, \quad \lambda \in \R,
\\
\label{3.74}
1 - \cos \lambda =  \frac{\lambda^2}{2} + \frac{|\lambda|^{1+\alpha} F_{c,p}(\lambda)}{(1 + |\operatorname{ln} |\lambda || )^p}, \quad \lambda \in \R.
\end{align}
Now, as a consequence of   \eqref{3.72}, one has
\begin{equation}
\label{3.77}
{\rho}^{(1)}_*(\bxi) =  \wh{\rho}^{(1)}_*(\bxi)+ \wt{\rho}^{(1)}_*(\bxi)
  \end{equation}
with
\begin{equation}
\label{3.78}
\begin{aligned}
\wh{\rho}^{(1)}_*(\bxi) &= \sum_{\n \in \Sigma_1} \intop_{\Omega + \n}  d\z\, \mu_*(\z)
\langle \z - \n, \bxi\rangle  \intop_0^1
 \frac{ \langle \bxi, \n + t (\z-\n)\rangle }{|\n + t (\z-\n)|^{d + \alpha} }\,dt
 \\-
&  (d+\alpha) \sum_{\n \in \Sigma_1} \intop_{\Omega + \n}  d\z\, \mu_*(\z)
\sum_{j=1}^d (z_j - n_j)  \intop_0^1
 \frac{(n_j + t(z_j - n_j)) \langle \bxi, \n + t (\z-\n)\rangle^2 }{2 |\n + t (\z-\n)|^{d + \alpha+2}}\,dt
\end{aligned}
  \end{equation}
and
\begin{equation}
\label{3.79}
\begin{aligned}
\wt{\rho}^{(1)}_*(\bxi) &\!=\!\! \sum_{\n \in \Sigma_1} \intop_{\Omega + \n} \!\! \!d\z\, \mu_*(\z)
\langle \z - \n, \bxi\rangle \!\! \intop_0^1 \!\!
 \frac{ |\langle \bxi, \n + t (\z-\n)\rangle |^\alpha F_{s,p}( \langle \bxi, \n + t (\z-\n)\rangle)}{|\n + t (\z-\n)|^{d + \alpha}
 (1+ |\operatorname{ln} |\langle \bxi, \n + t (\z-\n)\rangle| |)^p}\,dt
 \\
 &- (d+\alpha) \sum_{\n \in \Sigma_1} \intop_{\Omega + \n}  d\z\, \mu_*(\z)
\sum_{j=1}^d (z_j - n_j)
\\
&\times \int_0^1
 \frac{(n_j + t(z_j - n_j)) |\langle \bxi, \n + t (\z-\n)\rangle |^{1+\alpha} F_{c,p}( \langle \bxi, \n + t (\z-\n)\rangle)}{|\n + t (\z-\n)|^{d + \alpha+2}
 (1+ |\operatorname{ln} |\langle \bxi, \n + t (\z-\n)\rangle| |)^p}\,dt.
\end{aligned}
  \end{equation}
  Let us check that the function in  \eqref{3.78} can be written in the form
  \begin{equation}
\label{3.80a}
 \wh{\rho}^{(1)}_*(\bxi) = \sum_{\n \in \Sigma_1}
  \int_{\Omega + \n} \frac{\mu_*(\z) \langle \bxi, \z \rangle^2 }{2|\z|^{d+\alpha}} \,d\z.
  \end{equation}
 To this end we denote
 $$
\psi_{\bxi}(\z) :=  \frac{\langle \bxi, \z \rangle^2 }{2|\z|^{d+\alpha}}
$$
and transform the right-hand side of \eqref{3.80a} by analogy with \eqref{3.72a}:
\begin{equation*}
\begin{aligned}
&\int_{\Omega + \n}   \mu_*(\z) \frac{\langle \bxi, \z \rangle^2 }{2|\z|^{d+\alpha}} \,d\z
= \int_{\Omega + \n}   \mu_*(\z) \left(\psi_{\bxi}(\z) - \psi_{\bxi}(\n) \right) \,d\z
\\
&=  \int_{\Omega + \n}  d\z\, \mu_*(\z) \sum_{j=1}^d(z_j -n_j)
\int_0^1 \partial_j \psi_{\bxi}(\n + t (\z-\n))\,dt.
\end{aligned}
 \end{equation*}
 Since
$$
 \partial_j \psi_{\bxi}(\z) =  \frac{\xi_j \langle \bxi, \z\rangle}{|\z |^{d+\alpha}}
 - \frac{ (d+\alpha) z_j \langle \bxi, \z\rangle^2}{2 |\z |^{d+\alpha +2}},
$$
this yields
\begin{equation}
\label{3.79aa}
\begin{aligned}
& \int_{\Omega + \n}   \mu_*(\z) \frac{\langle \bxi, \z \rangle^2 }{2|\z|^{d+\alpha}} \,d\z  =
 \int_{\Omega + \n}  d\z\, \mu_*(\z)
\langle \z - \n, \bxi\rangle  \int_0^1 \frac{\langle \bxi, \n + t (\z-\n)\rangle}{|\n + t (\z-\n)|^{d+\alpha}}\,dt
\\
&- (d+\alpha)  \int_{\Omega + \n}  d\z\, \mu_*(\z)
\sum_{j=1}^d(z_j -n_j)
\int_0^1 \frac{(n_j + t (z_j - n_j)) \langle \bxi, \n + t (\z-\n)\rangle^2}{2 |\n + t (\z-\n)|^{d+\alpha+2}}
\,dt.
\end{aligned}
  \end{equation}
  From this relation, considering \eqref{3.78}, we derive representation \eqref{3.80a}.

Combining \eqref{3.67}, \eqref{3.70}, \eqref{3.77} and \eqref{3.80a} results in the relation
\begin{equation}
\label{3.79a}
{\rho}_*(\bxi) = \sum_{\n \in \Z^d}
\int_{\Omega + \n}   \mu_*(\z) \frac{\langle \bxi, \z \rangle^2 }{2|\z|^{d+\alpha}} \,d\z +
\wt{\rho}_*(\bxi),
    \end{equation}
where
 \begin{equation}
 \label{3.79ab}
\wt{\rho}_*(\bxi) = \wt{\rho}^{(0)}_*(\bxi) + \wt{\rho}^{(1)}_*(\bxi).
    \end{equation}
    Notice that the sum in  \eqref{3.79a}  converges absolutely. Indeed, summing over $\Sigma_0$ and $\Sigma_1$ separately
    we have
\begin{equation}
\label{3.79b}
\begin{aligned}
 \sum_{\n \in \Sigma_0}
\Bigl| \intop_{\Omega + \n}   \mu_*(\z) \frac{\langle \bxi, \z \rangle^2 }{2|\z|^{d+\alpha}} \,d\z \Bigr|
\le \mu_+ |\bxi|^2 \sum_{\n \in \Sigma_0} \intop_{\Omega + \n}    \frac{d\z  }{2|\z|^{d+\alpha-2}}
\\
\le \mu_+ |\bxi|^2 \intop_{|\z|\le R+ \sqrt{d}} \frac{d\z  }{2|\z|^{d+\alpha-2}} =
\mu_+ |\bxi|^2 \omega_d \frac{(R+ \sqrt{d})^{2-\alpha}}{2(2-\alpha)}
\end{aligned}
    \end{equation}
and, due to     \eqref{2.41}  и \eqref{3.79aa},
\begin{equation}
\label{3.79c}
\begin{aligned}
 \sum_{\n \in \Sigma_1}
\Bigl| \intop_{\Omega + \n}   \mu_*(\z) \frac{\langle \bxi, \z \rangle^2 }{2|\z|^{d+\alpha}} \,d\z \Bigr|
\le \mu_+ |\bxi|^2 (d+\alpha +1) \sum_{\n \in \Sigma_1} \intop_{\Omega + \n}    \frac{ d\z  }{|\z|^{d+\alpha-1}}
\\
\le \mu_+ |\bxi|^2 (d+\alpha +1) \intop_{|\z| > R} \frac{d\z  }{|\z|^{d+\alpha-1}} =
\mu_+ |\bxi|^2 \frac{\omega_d  (d+\alpha +1)}{(\alpha -1) R^{\alpha -1}}.
\end{aligned}
    \end{equation}
Summing up the last two inequalities yields
    $$
    \begin{aligned}
    & \sum_{\n \in \Z^d}
\Bigl| \int_{\Omega + \n}   \mu_*(\z) \frac{\langle \bxi, \z \rangle^2 }{2|\z|^{d+\alpha}} \,d\z \Bigr|
\le \mu_+  c_4(d,\alpha) |\bxi|^2,
\\
&c_4(d,\alpha) = \omega_d  \bigg(\frac{(R(d,\alpha)+ \sqrt{d})^{2-\alpha}}{2(2-\alpha)}
+ \frac{ d+\alpha +1}{(\alpha -1) R(d,\alpha)^{\alpha -1}}\bigg).
\end{aligned}
    $$
    This completes the proof of estimate \eqref{3.65a}.

It remains to estimate the form  in \eqref{3.79}.  Since $\wt{\rho}^{(1)}_*(\mathbf{0}) =0$,
it is sufficient to consider the case $\bxi \ne {\mathbf 0}$.
By \eqref{3.75}, \eqref{3.76}  we have
\begin{equation}
\label{3.80}
\begin{aligned}
|\wt{\rho}^{(1)}_*(\bxi)|& \le \mu_+ \left({\mathfrak c}_s(p,\alpha) + (d+\alpha) {\mathfrak c}_c(p,\alpha)\right)
|\bxi|^{1+\alpha}
\\ \times
 & \sum_{\n \in \Sigma_1} \intop_{\Omega + \n}  d\z\,   \intop_0^1
 \frac{ |\langle \wh{\bxi}, \n + t (\z-\n)\rangle |^\alpha}{|\n + t (\z-\n)|^{d + \alpha}
 (1+ |\operatorname{ln} |\langle \bxi, \n + t (\z-\n)\rangle| |)^p}\,dt, \quad \wh{\bxi}: = \frac{\bxi}{|\bxi|}.
 \end{aligned}
  \end{equation}

Representing $\z$ as $\z = \langle \wh{\bxi},\z\rangle \wh{\bxi} + \z_\perp$, notice that the cylinder
$$
{\mathcal C}_R := \left\{\z\in \R^d:\  |  \langle \wh{\bxi},\z\rangle | < \frac{R}{\sqrt{2}}, \ |\z_\perp| < \frac{R}{\sqrt{2}} \right\}
$$
is a subset of the ball $B_R(\mathbf{0})$ and, thus,  $(\Omega + \n) \subset \R^d \setminus {\mathcal C}_R$ for
$\n \in \Sigma_1$.
We consider  the following two cases:

Case 1. If $|\bxi| \le \min\left\{ \frac{1}{6\sqrt{d}}, \frac{1}{\sqrt{2} R} \right\} =: \delta_1(d,\alpha)$, we divide
the set $\R^d \setminus {\mathcal C}_R$  into three parts:
$$
\R^d \setminus {\mathcal C}_R = \Xi_1 \cup \Xi_2 \cup \Xi_3,
$$
where
\vskip -2mm
$$
\begin{aligned}
\Xi_1 := & \Big\{ \z \in \R^d:  |  \langle \wh{\bxi},\z\rangle | \le \frac{R}{\sqrt{2}}, \ |\z_\perp| \ge  \frac{R}{\sqrt{2}} \Big\},
\\
 \Xi_2 := & \Big\{ \z \in \R^d:  \frac{R}{\sqrt{2}} < |  \langle \wh{\bxi},\z\rangle | \le \frac{1}{2|\bxi|}, \ \z_\perp \in \R^{d-1} \Big\},
 \\
  \Xi_3 := & \Big\{ \z \in \R^d:   |  \langle \wh{\bxi},\z\rangle | > \frac{1}{2|\bxi|}, \ \z_\perp \in \R^{d-1} \Big\}.
 \end{aligned}
$$
Let $\Sigma_1^{(1)}$ be the collection of indices $\n \in \Sigma_1$ such that the cell $\Omega + \n$ has a non-trivial intersection with $\Xi_1$,
and let $\Sigma_1^{(2)}$ be the collection of indices  $\n \in \Sigma_1 \setminus \Sigma_1^{(1)}$ for which
the whole cell $\Omega + \n$ belongs to $\Xi_2$.
Denote  $\Sigma_1^{(3)} = \Sigma_1 \setminus (\Sigma_1^{(1)} \cup \Sigma_1^{(2)})$.

Using the notation $\widetilde{z}_1 = \langle \wh{\bxi}, \z\rangle$ and considering \eqref{2.41}, we have:
\begin{equation}
\label{3.81}
\begin{aligned}
&\sum_{\n \in \Sigma_1^{(1)}} \int_{\Omega + \n}  d\z\,   \int_0^1
 \frac{ |\langle \wh{\bxi}, \n + t (\z-\n)\rangle |^\alpha}{|\n + t (\z-\n)|^{d + \alpha}
 (1+ |\operatorname{ln} |\langle \bxi, \n + t (\z-\n)\rangle| |)^p}\,dt
 \\
&\le \sum_{\n \in \Sigma_1^{(1)}} \int_{\Omega + \n}  d\z\,   \int_0^1
 \frac{ dt}{|\n + t (\z-\n)|^{d} } \le  \sum_{\n \in \Sigma_1^{(1)}} \int_{\Omega + \n}  \frac{2d\z}{|\z|^d}
\\
& \le \intop_{|\wt{z}_1|  \le\frac{R}{\sqrt{2}} + \sqrt{d}} d\wt{z}_1 \intop_{|\z_\perp| \ge \frac{R}{\sqrt{2}}-\sqrt{d}}
 \frac{2\,d\z_\perp}{|\z_\perp|^d} = 4 \omega_{d-1} \Bigl( \frac{R}{\sqrt{2}} + \sqrt{d}\Bigr) \Bigl( \frac{R}{\sqrt{2}} - \sqrt{d}\Bigr)^{-1} =: c_6^{(1)}(d,\alpha).
 \end{aligned}
  \end{equation}

Notice that, for $\n \in \Sigma_1^{(2)}$ and $\z \in \Omega +\n$, the inequality
$ |  \langle \wh{\bxi}, \z\rangle | = |\wt{z}_1| > \frac{R}{\sqrt{2}} > \sqrt{d}$ holds. Therefore,
$$
| \langle \wh{\bxi}, \n + t (\z-\n)\rangle | \le |  \langle \wh{\bxi}, \z\rangle | + (1-t) | \langle \wh{\bxi}, \n -\z)\rangle|
\le |\wt{z}_1| + \sqrt{d} \le 2  |\widetilde{z}_1|.
$$
Since $ |\wt{z}_1| \le (2 |\bxi|)^{-1}$, then
$$
| \langle {\bxi}, \n + t (\z-\n)\rangle | \le 2 |\bxi|  |\wt{z}_1| \le 1,
$$
and, consequently,
$$
| \operatorname{ln} | \langle {\bxi}, \n + t (\z-\n)\rangle || \ge | \operatorname{ln} 2 |\bxi|  |\wt{z}_1|   |.
$$
Here we have used the fact that $ | \operatorname{ln} \lambda |$ is a decreasing function for $0< \lambda \le 1$.
Combining the above inequalities with  \eqref{2.41} yields
\begin{equation}
\label{3.82}
\begin{aligned}
&\sum_{\n \in \Sigma_1^{(2)}} \int_{\Omega + \n}  d\z\,   \int_0^1
 \frac{ |\langle \wh{\bxi}, \n + t (\z-\n)\rangle |^\alpha}{|\n + t (\z-\n)|^{d + \alpha}
 (1+ |\operatorname{ln} |\langle \bxi, \n + t (\z-\n)\rangle| |)^p}\,dt
 \\
& \le \sum_{\n \in \Sigma_1^{(2)}} \int_{\Omega + \n}
 \frac{ 2 (2 |\langle \wh{\bxi}, \z\rangle |)^\alpha \,d\z}{|\z|^{d + \alpha}
 (1+ |\operatorname{ln} 2 |\bxi|  |\wt{z}_1|  |)^p}
\\
& \le \intop_{ \frac{R}{\sqrt{2}} \le |\wt{z}_1|  \le\frac{1}{2|\bxi|} }
\frac{ 2 (2 | \wt{z}_1  |)^\alpha \,d \wt{z}_1}{ (1+ |\operatorname{ln} 2 |\bxi|  |\wt{z}_1|  |)^p}
 \intop_{\R^{d-1}}
 \frac{d\z_\perp}{ (\wt{z}_1^2 + |\z_\perp|^2)^{(d+\alpha)/2} }
\\
& = 2^{1+ \alpha} \mathfrak{c}'(d,\alpha)\intop_{ \frac{R}{\sqrt{2}} \le |\wt{z}_1|  \le\frac{1}{2|\bxi|} }
\frac{ d \wt{z}_1}{ |\wt{z}_1|(1+ |\operatorname{ln} 2 |\bxi|  |\wt{z}_1|  |)^p}
\\
&= 2^{1+ \alpha} \mathfrak{c}'(d,\alpha)\intop_{ \sqrt{2} R|\bxi| \le | s |  \le 1  }
\frac{ d s}{ | s|(1+ |\operatorname{ln} | s | |)^p}
\\
& \le 2^{1+ \alpha} \mathfrak{c}'(d,\alpha)\int_{-\infty}^0
\frac{ d \tau}{ (1+ |\tau| )^p} = 2^{1+ \alpha} \frac{\mathfrak{c}'(d,\alpha)}{p-1}
= c_6^{(2)}(p,d,\alpha),
 \end{aligned}
  \end{equation}
  where
$$
\mathfrak{c}'(d,\alpha):= \int_{\R^{d-1}}
 \frac{d\w'}{ (1  + |\w'|^2)^{(d+\alpha)/2} }.
$$
Here we have also performed the change of variables $s = 2 |\bxi|  \wt{z}_1$ and $\tau = \operatorname{ln} | s |$.

If  $\n \in \Sigma_1^{(3)}$ and $\z \in \Omega +\n$, then, taking into account the relation
$|\bxi|\le \frac{1}{6\sqrt{d}}$, we have $ |  \langle \wh{\bxi}, \z\rangle | = |\wt{z}_1| > \frac{1}{2|\bxi|} -\sqrt{d} \ge 2 \sqrt{d}$
and, therefore,
$$
| \langle \wh{\bxi}, \n + t (\z-\n)\rangle | \ge |  \langle \wh{\bxi}, \z\rangle | - (1-t) | \langle \wh{\bxi}, \n -\z)\rangle|
\ge |\wt{z}_1| - \sqrt{d} \ge \frac{1}{2}  |\wt{z}_1| \ge \frac{1}{2} \left(\frac{1}{2|\bxi|} - \sqrt{d} \right) \ge
\frac{1}{6|\bxi|}.
$$
Since $\wh{\bxi}=\frac\bxi{|\bxi|}$, this yields
$$
\textstyle
| \langle {\bxi}, \n + t (\z-\n)\rangle | \ge \frac{1}{2}  |\bxi| |\wt{z}_1| \ge \frac{1}{6},
$$
and, using the elementary inequality
$$
\textstyle
1 + \left| \operatorname{ln} \lambda_1  \right| \le (1 + \operatorname{ln} 6) \Big( 1 + \left| \operatorname{ln} \lambda_2  \right|\Big), \quad \frac{1}{6} \le \lambda_1 \le \lambda_2,
$$
we arrive at the estimate
\begin{equation}
\label{3.99a}
1 + \left| \operatorname{ln} | \langle {\bxi}, \n + t (\z-\n)\rangle | \right| \ge (1 + \operatorname{ln} 6)^{-1}
\left(1+ \bigl| \operatorname{ln} \frac{1}{2} |\bxi| |\wt{z}_1|   \bigr| \right).
  \end{equation}
Next, considering  \eqref{2.41} and \eqref{3.99a} we obtain
\begin{equation}
\label{3.83}
\begin{aligned}
&\sum_{\n \in \Sigma_1^{(3)}} \intop_{\Omega + \n}  d\z\,   \intop_0^1
 \frac{ |\langle \wh{\bxi}, \n + t (\z-\n)\rangle |^\alpha}{|\n + t (\z-\n)|^{d + \alpha}
 (1+ |\operatorname{ln} |\langle \bxi, \n + t (\z-\n)\rangle| |)^p}\,dt
 \\
&\le \sum_{\n \in \Sigma_1^{(3)}} \intop_{\Omega + \n}
 \frac{ 2  (1 + \operatorname{ln} 6)^p d\z}{|\z|^{d}\left(1+ \bigl| \operatorname{ln} \frac{1}{2} |\bxi| |\wt{z}_1|   \bigr| \right)^p }
 \\
& \le 2  (1 + \operatorname{ln} 6)^p \intop_{|\wt{z}_1|  \ge \frac{1}{2|\bxi|} - \sqrt{d}}
\frac{d\wt{z}_1}{\left(1+ \bigl| \operatorname{ln} \frac{1}{2} |\bxi| |\wt{z}_1|   \bigr| \right)^p}
\intop_{\R^{d-1}}
 \frac{d\z_\perp}{(\wt{z}_1^2 + |\z_\perp|^2)^{d/2}}
 \\
 & = 2  (1 + \operatorname{ln} 6)^p {\mathfrak c}'_d
 \intop_{|\wt{z}_1|  \ge \frac{1}{2|\bxi|} - \sqrt{d}}
\frac{d\wt{z}_1}{ |\wt{z}_1| \left(1+ \bigl| \operatorname{ln} \frac{1}{2} |\bxi| |\wt{z}_1|   \bigr| \right)^p}
\\
&= 2  (1 + \operatorname{ln} 6)^p {\mathfrak c}'_d
 \intop_{|s|  \ge \frac{1}{4} - |\bxi| \frac{\sqrt{d}}{2}}
\frac{d s}{ |s| \left(1+ \bigl| \operatorname{ln} |s|   \bigr| \right)^p}
 \le 2  (1 + \operatorname{ln} 6)^p {\mathfrak c}'_d
 \intop_{\R} \frac{d \tau}{\left(1+ |\tau | \right)^p}
 \\
 & = \frac{4}{p-1}  (1 + \operatorname{ln} 6)^p {\mathfrak c}'_d
 =: c_6^{(3)}(p,d);
 \end{aligned}
  \end{equation}
  here we have also performed the changes of variables  $s = \frac{1}{2} |\bxi|  \wt{z}_1$ and
  $\tau = \operatorname{ln} | s |$.

Finally,  by \eqref{3.80}--\eqref{3.82} and  \eqref{3.83},  we have the estimate
\begin{equation}
\label{3.84}
|\wt{\rho}^{(1)}_*(\bxi)| \le \mu_+ c_6(p,d,\alpha) |\bxi|^{1+\alpha},\quad |\bxi| \le \delta_1(d,\alpha),
  \end{equation}
with
$$
c_6(p,d,\alpha) =\left({\mathfrak c}_s(p,\alpha) + (d+\alpha) {\mathfrak c}_c(p,\alpha)\right)
 \left(c_6^{(1)}(d,\alpha) + c_6^{(2)}(p,d,\alpha) +c_6^{(3)}(p,d) \right).
 $$

Case 2.  The case $|\bxi| >  \delta_1(d,\alpha)$ is easier.  Let $\Sigma_1^{(1)}$ be defined as above, and denote
 $\Sigma_1^{(4)} = \Sigma_1 \setminus \Sigma_1^{(1)}$.
Estimate \eqref{3.81} remains valid.

If $\n \in \Sigma_1^{(4)}$ and $\z \in \Omega +\n$, then thanks to the inequality
$ |  \langle \wh{\bxi}, \z\rangle | = |\wt{z}_1| > \frac{R}{\sqrt{2}} -\sqrt{d} > \sqrt{d}$,
the following relations hold
$$
| \langle \wh{\bxi}, \n + t (\z-\n)\rangle |
\ge |\wt{z}_1| - \sqrt{d} =   |\wt{z}_1| - \sigma \Bigl(\frac{R}{\sqrt{2}} -\sqrt{d}\Bigr)
> (1-\sigma) |\wt{z}_1| >  \frac{R}{\sqrt{2}} - 2\sqrt{d}
=: \kappa(d,\alpha),
$$
where
$$
\sigma = \sigma(d,\alpha) := \frac{\sqrt{2d}}{{R} - \sqrt{2d}} <1.
$$
Consequently,
$$
| \langle {\bxi}, \n + t (\z-\n)\rangle | \ge (1-  \sigma) |\bxi| |\wt{z}_1| >
  \kappa(d,\alpha) \delta_1(d,\alpha).
$$
and we have
$$
1 + \left| \operatorname{ln} | \langle {\bxi}, \n + t (\z-\n)\rangle | \right| \ge
 (1 + |\operatorname{ln} (\kappa \delta_1)|)^{-1}
\left(1+ \bigl| \operatorname{ln} (1- \sigma) |\bxi| |\wt{z}_1|   \bigr| \right).
$$
In the same way as in \eqref{3.83}  we obtain
\begin{equation}
\label{3.85}
\begin{aligned}
&\sum_{\n \in \Sigma_1^{(4)}} \intop_{\Omega + \n}  d\z\,   \intop_0^1
 \frac{ |\langle \wh{\bxi}, \n + t (\z-\n)\rangle |^\alpha}{|\n + t (\z-\n)|^{d + \alpha}
 (1+ |\operatorname{ln} |\langle \bxi, \n + t (\z-\n)\rangle| |)^p}\,dt
 \\
&\le \sum_{\n \in \Sigma_1^{(4)}} \intop_{\Omega + \n}
 \frac{ 2  (1 + |\operatorname{ln} (\kappa \delta_1) |)^p d\z}{|\z|^{d}\left(1+ \bigl| \operatorname{ln} (1-\sigma) |\bxi| |\tilde{z}_1|   \bigr| \right)^p }
 \\
& \le 2  (1 + |\operatorname{ln} (\kappa \delta_1)|)^p \intop_{|\wt{z}_1| >\frac{R}{\sqrt{2}} -\sqrt{d}}
\frac{d\wt{z}_1}{\left(1+ \bigl| \operatorname{ln} (1-\sigma)  |\bxi| |\wt{z}_1|   \bigr| \right)^p}
\intop_{\R^{d-1}}
 \frac{d\z_\perp}{(\wt{z}_1^2 + |\z_\perp|^2)^{d/2}}
 \\
 & = 2   (1 + |\operatorname{ln} (\kappa \delta_1)|)^p  {\mathfrak c}'_d
 \intop_{|\wt{z}_1|  > \frac{R}{\sqrt{2}} -\sqrt{d} }
\frac{d\wt{z}_1}{ |\tilde{z}_1| \left(1+ \bigl| \operatorname{ln} (1-\sigma) |\bxi| |\wt{z}_1|   \bigr| \right)^p}
\\
& \le 2   (1 + |\operatorname{ln} (\kappa \delta_1)|)^p {\mathfrak c}'_d
 \int_{\R} \frac{d \tau}{\left(1+ |\tau | \right)^p}
  = \frac{4}{p-1}   (1 + |\operatorname{ln} (\kappa  \delta_1)|)^p {\mathfrak c}'_d
 =: c_6^{(4)}(p,d,\alpha).
 \end{aligned}
  \end{equation}
 From \eqref{3.80}, \eqref{3.81} and \eqref{3.85} it follows that
\begin{equation}
\label{3.86}
|\wt{\rho}^{(1)}_*(\bxi)| \le \mu_+ \check{c}_6(p,d,\alpha) |\bxi|^{1+\alpha},\quad \bxi \in \wt{\Omega},\ \
|\bxi| > \delta_1(d,\alpha),
  \end{equation}
with
$$
\check{c}_6(p,d,\alpha) =\left({\mathfrak c}_s(p,\alpha) + (d+\alpha) {\mathfrak c}_c(p,\alpha)\right)
 \left(c_6^{(1)}(d,\alpha) + c_6^{(4)}(p,d,\alpha) \right).
 $$
 Combining \eqref{3.84} and \eqref{3.86} yields
\begin{equation}
\label{3.87}
|\wt{\rho}^{(1)}_*(\bxi)| \le \mu_+ \max \{{c}_6(p,d,\alpha), \check{c}_6(p,d,\alpha) \} |\bxi|^{1+\alpha},\quad \bxi \in \wt{\Omega}.
  \end{equation}

It remains to choose  $p=2$ in \eqref{3.73}, \eqref{3.74} and take into account \eqref{3.71}, \eqref{3.79ab}, \eqref{3.87}
to deduce the required estimate \eqref{3.66}  with constant
$$
C_{10} = \mu_+ \left( c_5(d,\alpha) (\pi \sqrt{d})^{3-\alpha} +  \max \{{c}_6(2,d,\alpha), \check{c}_6(2,d,\alpha) \} \right)
= : \mu_+  c_7(d,\alpha).
$$
This completes the proof of Lemma.
\end{proof}

By Propositions \ref{prop2.5_2}, \ref{prop3.8},  identity \eqref{2.30}  and Lemmata \ref{lem2.6}, \ref{lem3.10},
we have

\begin{proposition}
\label{prop3.12}
Let conditions \eqref{e1.1}, \eqref{e1.2} be fulfilled, and assume that  $1< \alpha < 2$. Then for $|\bxi| \le \delta_0(\alpha,\mu)$ the following estimate holds{\rm :}
\begin{equation}
\label{2.47}
\left\| \A(\bxi) F(\bxi) - \bigl(\mu^0 c_0(d,\alpha) |\bxi|^{\alpha}  + \langle g^0 \bxi, \bxi \rangle \bigr) P \right\|_{L_2(\Omega) \to L_2(\Omega)} \le C_{11} |\bxi|^{1 +\alpha}.
\end{equation}
Here  $\mu^0$ is defined in \eqref{2.33},  $g^0 = \{g^0_{jk}\}$ is a
symmetric matrix with real entries given by
\begin{equation}
\label{3.105}
\begin{aligned}
g^0_{jk} \!&\!=  \frac{1}{4} \int_{\R^d}\! d\y \int_{\Omega}\! d\x\, \frac{\mu(\x,\y)}{|\x-\y|^{d+\alpha}}
 \left((x_j-y_j)(v_k(\x)   -  v_k(\y)) + (x_k-y_k) (v_j(\x) - v_j(\y))\right)
 \\
& +\int_{\R^d} \frac{\mu_*(\z) z_j z_k }{2|\z|^{d+\alpha}}d\z,\quad j,k=1,\dots,d,
 \end{aligned}
\end{equation}
where the last integral on the right-hand side is understood as the sum of integrals over the cells
$\Omega + \n,\ \n \in \Z^d$,
and the quadratic form $ \langle g^0 \bxi, \bxi \rangle$ satisfies the estimate
\begin{equation}
\label{3.106}
| \langle g^0 \bxi, \bxi \rangle | \le C_{12}|\bxi|^2.
\end{equation}
The constants $C_{11} = C_8+ C_{10}$ and $C_{12} =d C_7 + c_4 \mu_+$ are expressed in terms of
$d,$ $\alpha,$ $\mu_-,$ $\mu_+$.
\end{proposition}

\begin{remark}
The matrix  $g^0$  need not be sign-definite.
\end{remark}

\section{Approximation of the resolvent  $(\A + \eps^\alpha I)^{-1}$}\label{Sec3}

\subsection{Approximation of the resolvent of the operator $\A(\bxi)$}\label{sec3.1}

The following result on the  approximation of the resolvent $(\A(\bxi)+\varepsilon^{\alpha}I)^{-1}$ was proved in
\cite[Theorem 4.2]{JPSS24}.
\begin{theorem}[\cite{JPSS24}]
\label{teor2.1}
Let conditions \eqref{e1.1}, \eqref{e1.2} be satisfied, and assume that $1 < \alpha < 2$.
Then for all $\varepsilon>0$ and \hbox{$|\bxi|\le\delta_{0}(\alpha,\mu)$} one has
\begin{equation*}
\left\|(\A(\bxi)+\varepsilon^{\alpha}I)^{-1}-
(\mu^0 c_0(d,\alpha) |\bxi|^\alpha +\varepsilon^{\alpha})^{-1} P \right\|_{L_2(\Omega) \to L_2(\Omega)} \le
C_{13} \eps^{2 - 2\alpha},
\end{equation*}
where $\mu^0$ is defined in  \eqref{2.33},
and the constant $C_{13}$  is expressed in terms of  $d$, $\alpha$, $\mu_-$, $\mu_+$.
\end{theorem}

The derivation of a more accurate approximation relies on Proposition  \ref{prop3.12}.
In what follows we suppose that  $|\bxi|$ is so small that
\begin{equation*}
| \langle g^0 \bxi, \bxi \rangle | \le \frac{\mu^0 c_0}{2} |\bxi|^\alpha.
\end{equation*}
According to \eqref {3.106}, the last inequality holds if
\begin{equation}
\label{e4.3}
|  \bxi | \le \delta_2,  \quad \delta_2(d,\alpha,\mu) := \Big(\frac{\mu^0 c_0}{2 C_{12}}\Big)^{1/(2-\alpha)}.
\end{equation}
Then, by Proposition \ref{prop1.5} we have
\begin{align}\label{e2.39}
\|(\A(\bxi)+\varepsilon^{\alpha}I)^{-1} \|
& \le(\mu_{-} c_0(d,\alpha)|\bxi|^{\alpha}+\varepsilon^{\alpha})^{-1},\
\ \varepsilon>0,\ \ \bxi \in \wt{\Omega};
\\
\label{e2.40}
\left(\mu^0 c_0 |\bxi|^\alpha   + \langle g^0 \bxi, \bxi \rangle + \varepsilon^{\alpha} \right)^{-1}
& \le \Bigl( \frac{1}{2} \mu_{-} c_0(d,\alpha) |\bxi|^{\alpha}+\varepsilon^{\alpha} \Bigr)^{-1},\
\ \varepsilon>0,\ \ |\bxi| \le \delta_2.
\end{align}
Denote
$$
\delta_* = \delta_*(d,\alpha,\mu) := \min \{ \delta_0(\alpha,\mu), \delta_2(d, \alpha, \mu)\},
$$
where $\delta_0(\alpha, \mu)$ is defined in  \eqref{delta0}, and  $\delta_2(d, \alpha, \mu)$ ---  in  \eqref{e4.3}.
The notation $\Sigma (\bxi,\eps)$ stands for
\begin{equation}
\label{Xi=}
\Sigma (\bxi,\eps) := (\A(\bxi)+\varepsilon^{\alpha}I)^{-1} F(\bxi) -
\left(\mu^0 c_0 |\bxi|^\alpha   + \langle g^0 \bxi, \bxi \rangle + \varepsilon^{\alpha} \right)^{-1} P,
 \ \ |\bxi | \le \delta_*, \ \ \eps >0,
\end{equation}
where $\mu^0$ is defined in  \eqref{2.33}, and the entries of the matrix  $g^0$ are given by \eqref{3.105}.
\begin{proposition}
Let conditions  \eqref{e1.1}, \eqref{e1.2} be satisfied, and assume that  $1 < \alpha < 2$.
Then
\begin{equation}
\label{Xi=le}
\| \Sigma(\bxi,\eps) \|_{L_2(\Omega) \to L_2(\Omega)} \le C_{14} \eps^{1-\alpha},
\quad |\bxi|\le\delta_{*}, \quad \eps>0.
\end{equation}
\end{proposition}

\begin{proof}[Proof]
Combining the evident identity
\begin{multline*}
\Sigma(\bxi,\eps) =
F(\bxi)(\A(\bxi)+\varepsilon^{\alpha}I)^{-1}(F(\bxi)-P)
+(F(\bxi)-P) \left(\mu^0 c_0 |\bxi|^{\alpha}  + \langle g^0 \bxi, \bxi \rangle +\varepsilon^{\alpha} \right)^{-1} P
\\
-
F(\bxi)(\A(\bxi)+\varepsilon^{\alpha}I)^{-1}
\left(\A(\bxi)F(\bxi) - (\mu^0 c_0 |\bxi|^{\alpha}  + \langle g^0 \bxi, \bxi \rangle ) P \right)
\left(\mu^0 c_0 |\bxi|^{\alpha}  + \langle g^0 \bxi, \bxi \rangle +\varepsilon^{\alpha} \right)^{-1} P
\end{multline*}
with relations \eqref{F-P_a}, \eqref{2.47}, \eqref{e2.39} and \eqref{e2.40} leads to the estimate
$$
\begin{aligned}
\| \Sigma(\bxi,\eps) \|_{L_2(\Omega) \to L_2(\Omega)} & \le  \frac{C_5 |\bxi|}{\mu_{-} c_0|\bxi|^{\alpha}+\varepsilon^{\alpha}} +  \frac{C_5 |\bxi|}{\frac{1}{2}\mu_{-} c_0|\bxi|^{\alpha}+\varepsilon^{\alpha}}
+  \frac{C_{11} |\bxi|^{1+\alpha}}{(\mu_{-} c_0|\bxi|^{\alpha}+\varepsilon^{\alpha}) \bigl(\frac{1}{2}\mu_{-} c_0|\bxi|^{\alpha}+\varepsilon^{\alpha}\bigr)}
\\
& \le \bigg( \frac{C_5}{(\mu_- c_0)^{1/\alpha}} + \frac{C_5}{\bigl( \frac{1}{2} \mu_- c_0 \bigr)^{1/\alpha}}
+ \frac{C_{11}}{ \mu_- c_0 \bigl( \frac{1}{2} \mu_- c_0 \bigr)^{1/\alpha}} \bigg) \eps^{1-\alpha}
= : C_{14}  \eps^{1-\alpha},
\end{aligned}
$$
for $|\bxi| \le \delta_*$ and $\eps >0$.
\end{proof}

\begin{theorem}\label{teor2.2}
Let conditions \eqref{e1.1} and \eqref{e1.2} be fulfilled, and assume that $1 < \alpha < 2$.
Then, for $\varepsilon>0$ and \hbox{$|\bxi|\le\delta_{*}(d, \alpha,\mu)$},  the following estimate holds{\rm :}
\begin{equation}\label{e2.37a}
\big\|(\A(\bxi)+\varepsilon^{\alpha}I)^{-1}-
\left(\mu^0 c_0 |\bxi|^\alpha   + \langle g^0 \bxi, \bxi \rangle + \varepsilon^{\alpha} \right)^{-1} P \big\|_{L_2(\Omega) \to L_2(\Omega)} \le
C_{15}  \eps^{1 - \alpha};
\end{equation}
here $\mu^0$ is given by \eqref{2.33}, $g^0$ is the matrix whose entries are defined in  \eqref{3.105},
and the constant $C_{15}$  is expressed in terms of  $d$, $\alpha$, $\mu_-$, $\mu_+$.
\end{theorem}

\begin{proof}[Proof]
The inequality
\begin{equation}\label{e2.38}
\|(\A(\bxi)+\varepsilon^{\alpha}I)^{-1}(I-F(\bxi))\|\le
\frac{1}{d_{0}},\ \ |\bxi|\le\delta_{0}(\alpha,\mu),\ \ \eps >0,
\end{equation}
is a direct consequence of the definition of $F(\bxi)$ and Proposition \ref{prop2.1}.
Then \eqref{Xi=}, \eqref{Xi=le} and \eqref{e2.38} yield \eqref{e2.37a} with the constant
$C_{15} = C_{14} + d_0^{-1/\alpha}$.
\end{proof}

For any $N\geqslant 1$, after straightforward rearrangements we have
\begin{equation}
\label{4.11}
\begin{aligned}
\Big(\mu^0 c_0 |\bxi|^\alpha   + \langle g^0 \bxi, \bxi \rangle + \varepsilon^{\alpha} \Big)^{-1} =
\left(\mu^0 c_0 |\bxi|^\alpha    + \varepsilon^{\alpha} \right)^{-1}
\left( 1 + \frac{\langle g^0 \bxi, \bxi \rangle}{\mu^0 c_0 |\bxi|^\alpha    + \varepsilon^{\alpha}}\right)^{-1}
\\
= \left(\mu^0 c_0 |\bxi|^\alpha    + \varepsilon^{\alpha} \right)^{-1}
\sum_{m=0}^N \left( - \frac{\langle g^0 \bxi, \bxi \rangle}{\mu^0 c_0 |\bxi|^\alpha    + \varepsilon^{\alpha}} \right)^m
+ J_N(\bxi,\eps)
\end{aligned}
\end{equation}
with
\begin{equation}
\label{4.12}
J_N(\bxi,\eps) =
\Big( - \frac{\langle g^0 \bxi, \bxi \rangle}{\mu^0 c_0 |\bxi|^\alpha    + \varepsilon^{\alpha}} \Big)^{N+1}
\Big(\mu^0 c_0 |\bxi|^\alpha   + \langle g^0 \bxi, \bxi \rangle + \varepsilon^{\alpha} \Big)^{-1}.
\end{equation}
For $N$ such that  $2 - \frac{1}{N} < \alpha < 2$ one can estimate the right-hand side of \eqref{4.12} with the help of
 \eqref{3.106} and \eqref{e2.40}:
\begin{equation*}
|J_N(\bxi,\eps)| \le  \left( \frac{C_{12}}{\mu^0 c_0} \right)^{N+1} \frac{|\bxi |^{(N+1)(2-\alpha)}}{\frac{1}{2} \mu_- c_0
 |\bxi|^\alpha    + \varepsilon^{\alpha}}.
 \end{equation*}
If $2 - \frac{1}{N} < \alpha \le 2 - \frac{1}{N+1}$, then
\begin{equation}
\label{4.14}
\begin{aligned}
|J_N(\bxi,\eps)| &\le    \left( \frac{C_{12}}{\mu^0 c_0} \right)^{N+1} \frac{|\bxi |^{(N+1)(2-\alpha) - 1}}{\bigl(\frac{1}{2} \mu_- c_0\bigr)^{1/\alpha} }\eps^{1-\alpha} \le {\mathfrak C}_N \eps^{1-\alpha},\quad 2 - \frac{1}{N} < \alpha \le 2 - \frac{1}{N+1},
\\
{\mathfrak C}_N &= \left( \frac{C_{12}}{\mu^0 c_0} \right)^{N+1}
\frac{\delta_*^{(N+1)(2-\alpha) - 1}}{\bigl(\frac{1}{2} \mu_- c_0\bigr)^{1/\alpha} }.
\end{aligned}
\end{equation}
If $2 - \frac{1}{N+1} < \alpha < 2$, we have
\begin{equation}
\label{4.15}
\begin{aligned}
|J_N(\bxi,\eps)| &\le   {\mathfrak C}'_N \eps^{(N+1)(2-\alpha) - \alpha}, \quad 2 - \frac{1}{N+1} < \alpha < 2,
\\
{\mathfrak C}'_N &= \left( \frac{C_{12}}{\mu^0 c_0} \right)^{N+1}
 \frac{1}{\bigl( \frac{1}{2} \mu_- c_0\bigr)^{(N+1)(2-\alpha)/\alpha}}.
\end{aligned}
\end{equation}

The next statement is a consequence of Theorem \ref{teor2.2} and relations  \eqref{4.11}, \eqref{4.14}, \eqref{4.15}.

\begin{theorem}\label{teor4.4}
Let conditions \eqref{e1.1} and \eqref{e1.2} be satisfied, and assume that  $N \in \N$ is such that $2 - \frac{1}{N} < \alpha < 2$.
Then for any $\varepsilon>0$ and \hbox{$|\bxi|\le\delta_{*}(d, \alpha,\mu)$} the following estimate holds{\rm :}
\begin{equation}\label{e2.37}
\begin{aligned}
\Bigl\|(\A(\bxi)+\varepsilon^{\alpha}I)^{-1}-
\left(\mu^0 c_0 |\bxi|^\alpha   + \varepsilon^{\alpha} \right)^{-1} P
- \sum_{m=1}^N K_m(\bxi,\eps) P \Bigr\|_{L_2(\Omega) \to L_2(\Omega)}
\\
\le
C_{16}(N) \begin{cases}
  \eps^{1 - \alpha}, & 2 - \frac{1}{N} < \alpha \le 2 - \frac{1}{N+1},
  \\
 \eps^{(N+1)(2-\alpha) - \alpha}, & 2 - \frac{1}{N+1} < \alpha < 2.
  \end{cases}
  \end{aligned}
\end{equation}
Here  $\mu^0$ is the constant defined in \eqref{2.33},  $g^0$ is the matrix with entries defined in  \eqref{3.105},
and the functions $K_m(\bxi,\eps)$ are given by
$$
K_m(\bxi,\eps) := (-1)^m   \langle g^0 \bxi, \bxi \rangle^m \left(\mu^0 c_0 |\bxi|^\alpha   + \varepsilon^{\alpha} \right)^{-m-1}, \quad m=1,\dots, N.
$$
The constant $C_{16}(N)$  is expressed in terms of the parameters $d$, $\alpha$, $\mu_-$, $\mu_+$ and $N$.
\end{theorem}

The  {\it effective operator\/} $\A^{0}$ is introduced as the self-adjoint operator generated by the quadratic
form in \eqref{e1.3} with the constant coefficient $\mu^0$  defined in
\eqref{2.33}:
\begin{equation}
\label{eff_op}
\A^{0}:= \A(\alpha,\mu^0) = \mu^0 \A_0(\alpha) = \mu^0 c_0(d,\alpha) (- \Delta)^\gamma,
\quad \operatorname{Dom} \A^0 = H^{\alpha}(\R^d).
\end{equation}
With the help of the unitary Gelfand transform the operator  $\A^0$ is decomposed into a direct integral
\begin{equation}
\label{direct_int}
\A^0 = {\mathcal G}^*\Bigl( \int_{\wt{\Omega}} \oplus \A^0(\bxi)\,d\bxi  \Bigr) {\mathcal G}.
\end{equation}
Here
\begin{equation}
\label{eff_op_xi}
\A^0(\bxi) = \A(\bxi;\alpha,\mu^0) = \mu^0 \A_0(\bxi;\alpha) = \mu^0 c_0(d,\alpha) | \D + \bxi|^{\alpha},  \quad \Dom{\mathbb A}^0(\bxi) =  \wt{H}^{\alpha}(\Omega);
\end{equation}
 see \eqref{e1.14}.

The following result which is based on Theorem  \ref{teor2.1} has been obtained in  \cite[Thm 4.3]{JPSS24}:

\begin{theorem}[\cite{JPSS24}]
\label{teor4.5}
Let conditions \eqref{e1.1} and \eqref{e1.2} be fulfilled, and assume that  $1 < \alpha < 2$.
Then for all $\varepsilon>0$ and $\bxi \in \wt{\Omega}$ the estimate
\begin{equation*}
\left\|(\A(\bxi)+\varepsilon^{\alpha}I)^{-1} - (\A^0(\bxi)+\varepsilon^{\alpha}I)^{-1}
\right\|_{L_2(\Omega) \to L_2(\Omega)}\le
{\mathrm C}_1(\alpha,\mu)  \eps^{2 - 2\alpha}
\end{equation*}
holds. The constant  ${\mathrm C}_1(\alpha,\mu)$ is expressed in terms of the parameters
$d$, $\alpha$, $\mu_-$ and $\mu_+$.
\end{theorem}

Here, relying on Theorem \ref{teor4.4}, we provide more accurate approximation of the resolvent
$(\A(\bxi)+\varepsilon^{\alpha}I)^{-1}$. To this end, assuming that  $2 - \frac{1}{N} < \alpha < 2$, we exploit the \emph{correctors}  ${\mathbb K}_m(\eps)$, $m=1,\dots,N$, which are defined as the bounded self-adjoint operators in
$L_2(\R^d)$ given by
\begin{equation}
\label{e4.21}
{\mathbb K}_m(\eps) :=  ( \operatorname{div} g^0 \nabla)^m (\A^0+\varepsilon^{\alpha}I)^{-m-1},
\quad m=1,\dots,N.
\end{equation}
Notice that ${\mathbb K}_m(\eps)$  is a pseudo-differential operator of order $2m-(m+1)\alpha$. Since
 $\alpha > 2 - \frac{1}{N}$, the order $2m-(m+1)\alpha$ is negative for $m=1,\dots,N$.
By means of the Gelfand transform the operatop ${\mathbb K}_m(\eps)$ is decomposed into a direct integral
\begin{equation}
\label{direct_int_Km}
{\mathbb K}_m(\eps) = {\mathcal G}^*\Bigl( \int_{\wt{\Omega}} \oplus {\mathbb K}_m(\bxi,\eps)\,d\bxi  \Bigr) {\mathcal G}, \quad m=1,\dots,N.
\end{equation}
with
\begin{equation}
\label{Km_xi}
{\mathbb K}_m(\bxi,\eps) := (-1)^m \left((\D + \bxi)^* g^0 (\D + \bxi) \right)^m \!(\A^0(\bxi)+\varepsilon^{\alpha}I)^{-m-1}\!, \ \,  \Dom {\mathbb K}_m(\bxi,\eps) \!= \!L_2(\Omega).
\end{equation}

\begin{theorem}
\label{teor4.6}
Let conditions \eqref{e1.1} and \eqref{e1.2} be satisfied, and assume that $N \in \N$ is such that $2 - \frac{1}{N} < \alpha < 2$.
Then for all $\varepsilon>0$ and $\bxi \in \wt{\Omega}$ the following estimate holds{\rm :}
\begin{equation}
\label{e4.23}
\begin{aligned}
\Bigl\|(\A(\bxi)+\varepsilon^{\alpha}I)^{-1}- (\A^0(\bxi)+\varepsilon^{\alpha}I)^{-1} - \sum_{m=1}^N
{\mathbb K}_m(\bxi,\eps) \Bigr\|_{L_2(\Omega) \to L_2(\Omega)}
\\
\le {\mathrm C}_2(N,\alpha,\mu)
 \begin{cases}
  \eps^{1 - \alpha}, & 2 - \frac{1}{N} < \alpha \le 2 - \frac{1}{N+1},
  \\
 \eps^{(N+1)(2-\alpha) - \alpha}, & 2 - \frac{1}{N+1} < \alpha < 2,
  \end{cases}
\end{aligned}
\end{equation}
where the operators $\A^0(\bxi)$ and  ${\mathbb K}_m(\bxi,\eps)$, $m=1,\dots,N,$ are defined
in \eqref{eff_op_xi} and \eqref{Km_xi}, respectively, and the constant ${\mathrm C}_2(N,\alpha,\mu)$
is expressed in terms of the parameters $d$, $\alpha$, $\mu_-$, $\mu_+$ and $N$.
\end{theorem}

\begin{proof}[Proof]
For  $ |\bxi| \le \delta_*$ we use estimate \eqref{e2.37}.
For $ |\bxi| > \delta_*$ it is sufficient to estimate each term under the norm sign on the left-hand side of  \eqref{e2.37}  separately.
By \eqref{e2.39}, we have
$$
\left\|(\A(\bxi)+\varepsilon^{\alpha}I)^{-1} \right\| \le (\mu_- c_0\delta_*^\alpha)^{-1},
\quad \eps>0, \quad \bxi \in \wt{\Omega}, \quad |\bxi| > \delta_*.
$$
It is then clear that
$$
 \left( \mu^0 c_0 |\bxi|^{\alpha} + \eps^\alpha \right)^{-1}  \le (\mu_- c_0\delta_*^\alpha)^{-1},
 \quad \eps >0,\quad \bxi \in \wt{\Omega}, \quad |\bxi| >  \delta_*.
$$
Let us estimate  $|K_m (\bxi,\eps)|$, $m=1,\dots,N,$ for $|\bxi| \ge \delta_*$.

If $2 - \frac{1}{N} < \alpha \le 2 - \frac{1}{N+1}$, then, using the evident estimate
$$
\left\|(\A(\bxi)+\varepsilon^{\alpha}I)^{-1} \right\| \le
(\mu_- c_0\delta_*^\alpha)^{-1/\alpha} (\mu_- c_0 |\bxi|^\alpha + \eps^\alpha)^{1/\alpha -1}
\le (\mu_- c_0)^{-1/\alpha}\delta_*^{-1} \eps^{1-\alpha},
$$
and considering  \eqref{3.106}  we obtain
$$
\begin{aligned}
|K_m (\bxi,\eps)| &\le \frac{C_{12}^m |\bxi|^{2m}}{(\mu_- c_0 |\bxi|^\alpha)^{2m/\alpha} }
(\mu_- c_0 |\bxi|^\alpha + \eps^\alpha)^{-m-1 + 2m/\alpha}
\\
&= \frac{C_{12}^m }{(\mu_- c_0 )^{2m/\alpha} }
(\mu_- c_0 |\bxi|^\alpha + \eps^\alpha)^{1/\alpha -1}
(\mu_- c_0 \delta_*^\alpha )^{-m + (2m-1)/\alpha}
\\
&\le\frac{C_{12}^m }{(\mu_- c_0 )^{m +1/\alpha}  \delta_*^{m(2-\alpha) -1}}
\eps^{1-\alpha}.
\end{aligned}
$$
If $2 - \frac{1}{N+1} < \alpha < 2$, then
$$
\begin{aligned}
\left\|(\A(\bxi)+\varepsilon^{\alpha}I)^{-1} \right\| & \le
(\mu_- c_0\delta_*^\alpha)^{-(N+1)(2-\alpha)/\alpha} (\mu_- c_0 |\bxi|^\alpha + \eps^\alpha)^{(N+1)(2-\alpha)/\alpha -1}
\\
& \le (\mu_- c_0)^{-(N+1)(2-\alpha)/\alpha} \delta_*^{-(N+1)(2-\alpha)} \eps^{(N+1)(2-\alpha) -\alpha};
\end{aligned}
$$
here we have also used the inequality $(N+1)(2-\alpha) < \alpha$.  This yields
$$
\begin{aligned}
|K_m (\bxi,\eps)| &\le
 \frac{C_{12}^m }{(\mu_- c_0 )^{2m/\alpha} }
(\mu_- c_0 |\bxi|^\alpha + \eps^\alpha)^{(N+1)(2-\alpha)/\alpha -1}
(\mu_- c_0 \delta_*^\alpha)^{-m + 2m/\alpha - (N+1)(2-\alpha)/\alpha}
\\
&\le  \frac{C_{12}^m }{(\mu_- c_0 )^{m +(N+1)(2-\alpha)/\alpha}  \delta_*^{(2-\alpha)(N+1-m)}}
\eps^{(N+1)(2 -\alpha) - \alpha}.
\end{aligned}
$$
The operator  $ \left( \mu^0 c_0 |\bxi|^{\alpha} + \eps^\alpha \right)^{-1} P$ can be estimated in the same way.

It follows from the above relations that estimate \eqref{e2.37} remains valid for $\bxi \in \wt{\Omega}$, $|\bxi| \ge \delta_*$.
Finally, for all  $\bxi \in \wt{\Omega}$ and $\eps > 0$ it holds that
\begin{equation}\label{e4.25}
\begin{aligned}
\Bigl\|(\A(\bxi)+\varepsilon^{\alpha}I)^{-1}-
\left(\mu^0 c_0 |\bxi|^\alpha   + \varepsilon^{\alpha} \right)^{-1} P
- \sum_{m=1}^N K_m(\bxi,\eps) P \Bigr\|_{L_2(\Omega) \to L_2(\Omega)}
\\
\le
\wt{C}_{16}(N) \begin{cases}
  \eps^{1 - \alpha}, & 2 - \frac{1}{N} < \alpha \le 2 - \frac{1}{N+1},
  \\
 \eps^{(N+1)(2-\alpha) - \alpha}, & 2 - \frac{1}{N+1} < \alpha < 2,
  \end{cases}
  \end{aligned}
\end{equation}
where the constant $\wt{C}_{16}(N)$ depends only on  $d$, $\alpha$, $\mu_-$, $\mu_+$ and $N$.

Next, due to the evident identity
\begin{equation*}
\label{A0_P}
\A^0(\bxi) P = \mu^0 c_0 |\bxi|^\alpha P,
\end{equation*}
we have
\begin{equation*}
\label{res0_P}
(\A^0(\bxi) + \eps^\alpha I)^{-1}P = \left( \mu^0 c_0 |\bxi|^\alpha  + \eps^\alpha \right)^{-1} P.
\end{equation*}
Similarly,
\begin{equation*}
\label{Km_P}
\mathbb{K}_m(\bxi, \eps) P = K_m(\bxi, \eps) P,\quad m=1,\dots,N.
\end{equation*}
As a result, the relation \eqref{e4.25} can be rewritten in the form
\begin{equation}\label{e4.26}
\begin{aligned}
\Bigl\|(\A(\bxi)+\varepsilon^{\alpha}I)^{-1}- (\A^0(\bxi)+\varepsilon^{\alpha}I)^{-1} P
- \sum_{m=1}^N \mathbb{K}_m(\bxi,\eps) P \Bigr\|_{L_2(\Omega) \to L_2(\Omega)}
\\
\le
\wt{C}_{16} \begin{cases}
  \eps^{1 - \alpha}, & 2 - \frac{1}{N} < \alpha \le 2 - \frac{1}{N+1},
  \\
 \eps^{(N+1)(2-\alpha) - \alpha}, & 2 - \frac{1}{N+1} < \alpha < 2,
  \end{cases}
  \end{aligned}
\end{equation}
for all $\varepsilon>0$ and $\bxi \in \wt{\Omega}$.

Since $|2\pi \n+\bxi| \ge \pi$ for all $\bxi \in \wt{\Omega}$
and $0 \ne \n \in \Z^d$, then,
by means of the discrete Fourier transform, we deduce that
\begin{equation*}
\left\| (\A^0(\bxi)+\varepsilon^{\alpha}I)^{-1}(I-P) \right\| = \sup_{0 \ne \n \in \Z^d}
(\mu^0 c_0 |2\pi \n+\bxi|^\alpha + \eps^\alpha )^{-1}
\le (\mu_- c_0 \pi^\alpha + \eps^\alpha)^{-1};
\end{equation*}
here we have taken into account \eqref{e1.16}.
From this inequality, for $\alpha\in\big(2-\frac1N,2-\frac 1{N+1}\big]$,  we derive the upper bound
$$
\left\| (\A^0(\bxi)+\varepsilon^{\alpha}I)^{-1}(I-P) \right\|
\le (\mu_- c_0 \pi^\alpha)^{-1/\alpha} (\mu_- c_0 \pi^\alpha + \eps^\alpha)^{1/\alpha -1} \le
(\mu_- c_0 )^{-1/\alpha} \pi^{-1} \eps^{1 - \alpha}
$$
and, for  $\alpha\in \big(2 - \frac{1}{N+1},2\big)$,   the upper bound
$$
\begin{aligned}
\left\| (\A^0(\bxi)+\varepsilon^{\alpha}I)^{-1}(I-P) \right\|
&\le (\mu_- c_0 \pi^\alpha)^{-(N+1)(2-\alpha)/\alpha} (\mu_- c_0 \pi^\alpha + \eps^\alpha)^{(N+1)(2-\alpha)/\alpha  -1} \\
&\le
(\mu_- c_0 )^{-(N+1)(2-\alpha)/\alpha} \pi^{-(N+1)(2-\alpha)} \eps^{(N+1)(2-\alpha) - \alpha}.
\end{aligned}
$$

The discrete Fourier transform is also used in order to estimate the norm $\|\mathbb{K}_m(\bxi,\eps) (I-P)\|$
for  $m=1,\dots,N$:
\begin{equation*}
\begin{aligned}
\left\|  \mathbb{K}_m(\bxi,\eps) (I-P) \right\| &=
\sup_{0 \ne \n \in \Z^d} |\langle g^0 (2\pi \n+\bxi), 2\pi \n+\bxi\rangle|^m
(\mu^0 c_0 |2\pi \n+\bxi|^\alpha + \eps^\alpha )^{-m-1}
\\
&\le   \frac{C_{12}^m}{(\mu_- c_0)^{2m/\alpha}} (\mu_- c_0 \pi^\alpha + \eps^\alpha)^{-m -1 + 2m/\alpha}.
\end{aligned}
\end{equation*}
For $2 - \frac{1}{N} < \alpha \le 2 - \frac{1}{N+1}$ this yields
$$
\begin{aligned}
\left\|  \mathbb{K}_m(\bxi,\eps) (I-P) \right\| &\le
\frac{C_{12}^m}{(\mu_- c_0)^{2m/\alpha}} (\mu_- c_0 \pi^\alpha)^{-m  + (2m-1)/\alpha}
(\mu_- c_0 \pi^\alpha + \eps^\alpha)^{1/\alpha -1}
\\
& \le
\frac{C_{12}^m}{(\mu_- c_0)^{m + 1/\alpha}  \pi^{1 - m(2-\alpha)}}
\eps^{1 -\alpha},
\end{aligned}
$$
and, for  $2 - \frac{1}{N+1} < \alpha < 2$,
$$
\begin{aligned}
\left\|  \mathbb{K}_m(\bxi,\eps) (I-P) \right\| & \le
\frac{C_{12}^m}{(\mu_- c_0)^{2m/\alpha}} (\mu_- c_0 \pi^\alpha)^{-m  + 2m/\alpha - (N+1)(2-\alpha)/\alpha}
(\mu_- c_0 \pi^\alpha + \eps^\alpha)^{(N+1)(2-\alpha)/\alpha -1}
\\
& \le
\frac{C_{12}^m}{(\mu_- c_0)^{m+1/\alpha} \pi^{(N+1 -m)(2-\alpha)}}
\eps^{(N+1)(2-\alpha) -\alpha}.
\end{aligned}
$$
Finally, the desired inequality \eqref{e4.23} follows from the above estimates for the norms of operators
$(\A^0(\bxi)+\varepsilon^{\alpha}I)^{-1}(I-P)$ and
$ \mathbb{K}_m(\bxi,\eps) (I-P)$, $m=1,\dots,N$, and from \eqref{e4.26}.
\end{proof}

\subsection{Approximation of the resolvent  $(\A+\varepsilon^{\alpha}I)^{-1}$}

The following result was obtained in \cite[Theorem 4.4]{JPSS24} as a consequence of decompositions
\eqref{e1.15}, \eqref{direct_int} and Theorem  \ref{teor4.5}.

\begin{theorem}[\cite{JPSS24}]
\label{teor4.7}
Let conditions \eqref{e1.1}, \eqref{e1.2} be satisfied, and assume that $1 < \alpha < 2$.
Then for all $\eps >0$ we have
\begin{equation*}
\left\|(\A+\varepsilon^{\alpha}I)^{-1}- (\A^0+\varepsilon^{\alpha}I)^{-1}
\right\|_{L_2(\R^d) \to L_2(\R^d)}\le
{\mathrm C}_1(\alpha,\mu)  \eps^{2 - 2\alpha},
\end{equation*}
where $\A^0$ is the homogenized operator introduced in  \eqref{eff_op},
and the constant  ${\mathrm C}_1(\alpha,\mu)$ is expressed in terms of the parameters $d$, $\alpha$, $\mu_-$
and $\mu_+$.
\end{theorem}
In the present work we provide a more accurate approximation of the resolvent $(\A+\varepsilon^{\alpha}I)^{-1}$
which relies on Theorem \ref{teor4.6}.
\begin{theorem}
\label{teor4.8}
Let conditions \eqref{e1.1}, \eqref{e1.2} be fulfilled, and assume that
 $N \in \N$ and  $2 - \frac{1}{N} < \alpha < 2$.
Then for any $\eps>0$ the following estimate holds:
\begin{equation}
\label{th4.8_1}
\begin{aligned}
\Bigl\|(\A+\varepsilon^{\alpha}I)^{-1}- (\A^0+\varepsilon^{\alpha}I)^{-1} - \sum_{m=1}^N
{\mathbb K}_m(\eps) \Bigr\|_{L_2(\R^d) \to L_2(\R^d)}
\\
\le {\mathrm C}_2(\alpha,\mu)
 \begin{cases}
  \eps^{1 - \alpha}, & 2 - \frac{1}{N} < \alpha \le 2 - \frac{1}{N+1},
  \\
 \eps^{(N+1)(2-\alpha) - \alpha}, & 2 - \frac{1}{N+1} < \alpha < 2.
  \end{cases}
\end{aligned}
\end{equation}
Here the correctors ${\mathbb K}_m(\eps)$, $m=1,\dots,N,$ are given by  \eqref{e4.21}.
The constant ${\mathrm C}_2(\alpha,\mu)$ is expressed in terms of the parameters $d$, $\alpha$, $\mu_-$ and $\mu_+$.
\end{theorem}

\begin{proof}[Proof]
According to \eqref{e1.15},  \eqref{direct_int} and \eqref{direct_int_Km}, the operator
$(\A+\varepsilon^{\alpha}I)^{-1}- (\A^0+\varepsilon^{\alpha}I)^{-1} - \sum_{m=1}^N
{\mathbb K}_m(\eps) $ can be decomposed by the Gelfand transform into a direct integral over the operators
$(\A(\bxi) +\varepsilon^{\alpha}I)^{-1}- (\A^0(\bxi)+\varepsilon^{\alpha}I)^{-1} - \sum_{m=1}^N {\mathbb K}_m(\bxi,\eps)$.  Therefore,
\begin{multline*}
\Bigl\|(\A+\varepsilon^{\alpha}I)^{-1}- (\A^0+\varepsilon^{\alpha}I)^{-1} - \sum_{m=1}^N {\mathbb K}_m(\eps)
\Bigr\|_{L_2(\R^d) \to L_2(\R^d)}
\\
=
\sup_{\bxi \in \wt{\Omega}} \Bigl\|(\A(\bxi)+\varepsilon^{\alpha}I)^{-1}- (\A^0(\bxi)+\varepsilon^{\alpha}I)^{-1}
- \sum_{m=1}^N {\mathbb K}_m(\bxi,\eps) \Bigr\|_{L_2(\Omega) \to L_2(\Omega)},
\end{multline*}
and estimate \eqref{th4.8_1} follows from \eqref{e4.23}.
\end{proof}

\section{Homogenization of L\'evy-type operators}

\subsection{Main result}
Assuming that conditions \eqref{e1.1}, \eqref{e1.2} are satisfied and $1< \alpha <2$, we consider
the family of operators $\A_\eps := \A(\alpha,\mu^\eps)$, $\eps >0$, in $L_{2}(\R^d)$ which are defined in
\eqref{Aeps_Intro}, \eqref{q_form_Intro}.  Thus $\A_\eps$ is a self-adjoint operator in  $L_{2}(\R^d)$ generated by the closed quadratic form
\begin{equation*}
a_{\varepsilon}[u,u] = \frac{1}{2} \int_{\R^d}  \int_{\R^d} d\x \, d\y \, \mu^\eps(\x,\y)
\frac{|u(\x)-u(\y)|^2}{|\x - \y|^{d+\alpha}},\ \ u\in H^\gamma(\R^d),\ \ \varepsilon>0,
\end{equation*}
with $\mu^\eps(\x,\y):= \mu(\x/\varepsilon,\y/\varepsilon)$.
Recall that the effective operator  $\A^0$ is given by \eqref{eff_op}, and the effective coefficient
$\mu^0$  -- by  \eqref{2.33}.

The scaling transformation $T_{\varepsilon}$ is introduced as the family of unitary operators
defined by
\begin{equation*}
T_{\varepsilon}u(\x):=\varepsilon^{d/2}u(\varepsilon \x),\ \
\x\in\R^d,\ \ u\in L_{2}(\R^d), \ \ \eps >0.
\end{equation*}
It is an easy exercise to check that
$$
\A_{\varepsilon} = \varepsilon^{-\alpha} T_{\varepsilon}^* \A T_{\varepsilon},\quad \eps >0.
$$
Therefore,
\begin{equation}\label{e3.2}
(\A_{\varepsilon}+I)^{-1}= T_{\varepsilon}^{*}\varepsilon^{\alpha}(\A+\varepsilon^{\alpha}I)^{-1}T_{\varepsilon},
\ \ \varepsilon>0.
\end{equation}
The effective operator  satisfies a similar identity
$$
\A^0 = \varepsilon^{-\alpha} T_{\varepsilon}^* \A^0 T_{\varepsilon},\quad \eps >0,
$$
and hence
\begin{equation}\label{e3.2_eff}
(\A^0+I)^{-1}= T_{\varepsilon}^{*}\varepsilon^{\alpha}(\A^0+\varepsilon^{\alpha}I)^{-1}T_{\varepsilon},
\ \ \varepsilon>0.
\end{equation}
Combining \eqref{e3.2}, \eqref{e3.2_eff} and the unitarity of the operator $T_\eps$ we conclude that
$$
\|(\A_{\varepsilon}+I)^{-1}-(\A^{0}+I)^{-1}\|_{L_2(\R^d) \to L_2(\R^d)} =
\eps^\alpha \|(\A + \varepsilon^\alpha I)^{-1}-(\A^{0}+\eps^\alpha I)^{-1}\|_{L_2(\R^d) \to L_2(\R^d)}.
$$
This relation and Theorem \ref{teor4.5}  yield the following result, see   \cite[Theorem 6.1]{JPSS24}:
\begin{theorem}[\cite{JPSS24}]
\label{teor3.1}
Assume that conditions  \eqref{e1.1} and  \eqref{e1.2} hold, and $1< \alpha <2$.
Then, for any $\varepsilon>0$ the following estimate is valid{\rm :}
\begin{equation}\label{e3.1}
\|(\A_{\varepsilon}+I)^{-1}-(\A^{0}+I)^{-1}\|_{L_2(\R^d) \to L_2(\R^d)}\le
{\mathrm C}_1(\alpha,\mu)
 \eps^{2 - \alpha};
 \end{equation}
 here the constant ${\mathrm C}_1(\alpha,\mu)$ is expressed in terms of  $d$, $\alpha$, $\mu_-$, $\mu_+$.
\end{theorem}

Making use of Theorem \ref{teor4.8} we obtain a more accurate approximation of the resolvent   $(\A_{\varepsilon}+I)^{-1}$.

\begin{theorem}
\label{teor3.2}
Let conditions \eqref{e1.1}, \eqref{e1.2} be fulfilled, and assume that  $N \in \N$ is such that
$2 - \frac{1}{N} < \alpha < 2$. Then, for any $\varepsilon>0$ the following estimate holds{\rm :}
\begin{equation}\label{e3.2b}
\begin{aligned}
\Bigl\| (\A_{\varepsilon}+I)^{-1}-(\A^{0}+I)^{-1} - \sum_{m=1}^N  \eps^{m(2-\alpha)} \mathbb{K}_m \Bigr\|_{L_2(\R^d) \to L_2(\R^d)}
\\
\le
{\mathrm C}_2(N,\alpha,\mu) \begin{cases}
  \eps, & 2 - \frac{1}{N} < \alpha \le 2 - \frac{1}{N+1},
  \\
 \eps^{(N+1)(2-\alpha)}, & 2 - \frac{1}{N+1} < \alpha < 2.
  \end{cases}
\end{aligned}
 \end{equation}
 Here $\A^0$ is the effective operator defined in \eqref{eff_op}, and the  correctors ${\mathbb K}_m,$
 $m=1,\dots,N,$ are given by
 $$
 {\mathbb K}_m :=  ( \operatorname{div} g^0 \nabla)^m (\A^{0}+I)^{-m-1},\quad m=1,\dots,N.
 $$
 The constant ${\mathrm C}_2(N,\alpha,\mu)$ is expressed in terms of  $d$, $\alpha$, $\mu_-$ and $\mu_+$.
\end{theorem}

\begin{proof}[Proof]
The correctors ${\mathbb K}_m$ and the operators  ${\mathbb K}_m(\eps)$ defined in  \eqref{e4.21} can be expressed in  terms of each other by means of the scaling transformation $T_{\varepsilon}$. The corresponding relation reads
\begin{equation}\label{Km_eff}
\mathbb{K}_m = T_{\varepsilon}^{*}\varepsilon^{ - m(2-\alpha) + \alpha} {\mathbb K}_m(\eps)T_{\varepsilon},
\quad m=1,\dots,N.
\ \ \varepsilon>0.
\end{equation}
Since the operator $T_\eps$ is unitary, from \eqref{e3.2}, \eqref{e3.2_eff} and \eqref{Km_eff} it follows that
$$
\begin{aligned}
& \Bigl\|(\A_{\varepsilon}+I)^{-1}-(\A^{0}+I)^{-1} - \sum_{m=1}^N  \eps^{m(2-\alpha)} \mathbb{K}_m
\Bigr\|_{L_2(\R^d) \to L_2(\R^d)}
\\
& =
\eps^\alpha \|(\A + \varepsilon^\alpha I)^{-1}-(\A^{0}+\eps^\alpha I)^{-1}
- \sum_{m=1}^N   \mathbb{K}_m(\eps) \|_{L_2(\R^d) \to L_2(\R^d)}.
\end{aligned}
$$
Combining this relation with Theorem \ref{teor4.8}  we obtain the required estimate \eqref{e3.2b}.
\end{proof}

\subsection{Concluding remarks}

1.  It follows from Theorem \ref{teor3.2} that the precision $O(\eps^{2-\alpha})$ in estimate \eqref{e3.1}
is order-sharp.

2.  For any $\alpha\in (1,2)$ one can choose $N \in \N$ in such a way that
$2 - \frac{1}{N} < \alpha \le 2 - \frac{1}{N+1}$.  Then, taking into account first $N$ correctors, we obtain an approximation of the resolvent  $(\A_{\varepsilon}+I)^{-1}$ of order $\eps$.

3. From the explicit formulae for the constants ${\mathrm C}_1(\alpha,\mu)$ and ${\mathrm C}_2(N,\alpha,\mu)$  it is easy to deduce that ${\mathrm C}_1(\alpha,\mu)$ depends only on  $d$, $\alpha$,
$\mu_-$ and $\mu_+$, while ${\mathrm C}_2(N,\alpha,\mu)$ depends also on $N$.
Moreover, both constants tend to infinity, as $\alpha\to 1$ or $\alpha\to 2$.

4. The statements of Theorems \ref{teor3.1} and  \ref{teor3.2} remain valid if the periodicity lattice $\Z^d$ is replaced
with an arbitrary periodic lattice in $\R^d$.  In this case the constants in estimates  \eqref{e3.1} and  \eqref{e3.2b} will
also depend on the parameters of the lattice.

\end{document}